\documentclass[12pt]{article}
\usepackage{amsfonts,amsmath,amssymb,amsthm,amscd}
\usepackage{authblk}
\usepackage{xcolor}
\usepackage[all,cmtip]{xy}
\usepackage{tikz-cd}
\usepackage{enumitem}
\usepackage[normalem]{ulem}
\usepackage{placeins} 
\setitemize{itemsep=0pt}
\usepackage[
  style=alphabetic, isbn=false, doi=false, url=false
]{biblatex}
\usepackage[toc]{appendix}
\usepackage{hyperref}

\evensidemargin=\oddsidemargin

\theoremstyle{plain}
\newtheorem{theorem}{Theorem}[section]
\newtheorem{lemma}[theorem]{Lemma}
\newtheorem{proposition}[theorem]{Proposition}
\newtheorem{corollary}[theorem]{Corollary}
\newtheorem{complement}[theorem]{Complement}

\theoremstyle{definition}
\newtheorem{definition}[theorem]{Definition}
\newtheorem{remark}[theorem]{Remark}
\newtheorem*{note*}{Note}

\newtheorem*{claim*}{Claim}
\newtheorem*{question*}{Question}

\newcommand{\id}{\operatorname{id}}
\newcommand{\DDD}{{\mathcal{D}yn} }

\renewcommand{\Re}{\operatorname{Re}}
\renewcommand{\Im}{\operatorname{Im}}
\renewcommand{\cal}[1]{{\mathcal{#1}}}
\renewcommand{\emptyset}{\varnothing}

\newcommand{\on}{\operatorname}
\newcommand{\ov}{\overline}

\newcommand{\C}{\mathbf{C}}
\newcommand{\D}{\mathbf{D}}
\renewcommand{\H}{\mathbf{H}}

\newcommand{\N}{\mathbf{N}}

\newcommand{\R}{\mathbf{R}}

\newcommand{\Z}{\mathbf{Z}}

\newcommand{\wt}{\widetilde}
\newcommand{\wh}{\widehat}
\newcommand{\dom}{\operatorname{Dom}}

\newcommand{\att}{\mathrm{att}}
\newcommand{\rep}{\mathrm{rep}}
\newcommand{\Per}{\mathrm{Per}}
\newcommand{\tends}{\longrightarrow}

\newcommand{\e}{\texorpdfstring}
\newcommand{\h}{\mathbf{h}}

\definecolor{arnaudcolor}{rgb}{0.5 0 1}

\definecolor{pascalecolor}{rgb}{0 0.7 0}

\begin{document}
\title{Some invariant classes for parabolic renormalization in the multicritical case}

\author[1]{A.\ Chéritat}
\author[1]{P.\ Roesch}
\affil[1]{\small Université de Toulouse, INSA Toulouse, CNRS, Institut de Mathématiques de Toulouse, UMR5219}

\date{\today}
\maketitle

\begin{abstract} 
Parabolic renormalization associates to a holomorphic map $f$ with a parabolic fixed point another holomorphic map with a parabolic fixed point. This procedure is essential for understanding the phenomenon of parabolic enrichment, which occurs when one perturbs $f$ appropriately.
Shishikura defined in \cite{Shi} (see also \cite{LY}) a class of maps that is stable under this parabolic renormalization operator.
These maps have only one critical point in their immediate basins.
We extend here this result to the more general classes of maps conjugated on their immediate basins to finite Blaschke products.
For these classes, we introduce an analogue of Milnor's mapping schemes \cite{MilnorH} and describe the action of parabolic renormalization on the scheme.
This shall be a starting point to the fine study of perturbation of these bigger classes of maps.
\end{abstract}

\setcounter{tocdepth}{2}
\tableofcontents

\section{Introduction}

\subsection{Background and motivation}\label{sub:bkgd}

In one-dimensional holomorphic dynamical systems, parabolic implosion is central to the understanding of dynamical enrichment: it explains, in the setting of perturbation of a map with a parabolic point, the existence of geometric limits, i.e.\ the fact that high iterates of perturbed maps converge when the perturbation tends to $0$, to so-called Lavaurs maps, that are not in the original dynamical systems.
This describes enrichment of the Julia sets and lots of consequences have been drawn from this, which we will not try to list here.

One of the tools of parabolic implosion is \emph{parabolic renormalization}.
To a map $f$ with a parabolic point it associates another map $Rf$ with a parabolic point, obtained from its Fatou coordinates, see Section~\ref{sub:ren}.
It allows at least two things: firstly to control the geometric limits, secondly, by iterating renormalization, to catch finer geometric limits. 

In order to control iterates of renormalization, one seeks for properties that are preserved, in other words, for invariant classes.
For instance, Shishikura proved\footnote{In \cite{Shi1}, this follows from the proposition and the two lemmas of Section~4.5.} the following result, where $\DDD_2$ is the set of holomorphic maps with $f(0)=0$, $f'(0)=1$, $f''(0)\neq 0$ (i.e.\ having a parabolic fixed point at $0$ with only one attracting axis) and for which, on its immediate parabolic basin $B^*$, the restriction $f: B^*\to B^*$ is proper and has only one critical point of local degree $2$ (i.e.\ conjugated on $B^*$ to the degree $2$ Blaschke product $G(z)=\frac{3z^2+1}{z^2+3}$).

\begin{theorem}[Shishikura]\label{thm:SLY}
    \[R(\DDD_2)\subset\DDD_2\]
\end{theorem}

The book \cite{LY} takes this a step further by proving invariance by renormalization of a subclass of $\DDD_2$, called $\mathbf{P}^0$.
It consists of maps $f$ defined on Jordan domains and equivalent\footnote{I.e.\ such that there is $\lambda\in \C^*$ and a conformal mapping $\phi$ fixing $0$ from $\dom f$ to $\dom RG$, such that $\lambda f = RG \circ \phi$.} to the renormalization $RG$ of $G(z)=\frac{3z^2+1}{z^2+3}$ (or equivalently to the renormalization $Rf_0$ of the Cauliflower map $f_0(z)=z+z^2$).

\bigskip

The present work finds its motivation too in the study of particular slices of the parameter space of cubic polynomials.
One of the early studies of this parameter space is by 
Branner and Hubbard in \cite{BH}.
This space has been intensively studied since then.
The study of $1$-dimensional complex slice $\Per_1(0)$ of cubic maps which have a critical fixed point was started by Milnor in 1991 (this was only published in 2009 in \cite{MilnorCubic}).
One of the authors of the present article studied it in \cite{Roesch1}.
This has been extended by holomorphic motions to  slices with an attracting fixed point with multiplier $\lambda$: $\Per_1(\lambda)$, by Petersen and Tan Lei in \cite{PetersenTan}. For $\vert \lambda\vert=1$ these slices have been studied by several people, including: Blokh, Oversteegen and Timorin in \cite{blokh} and in  more details by Zhang in \cite{Runze} for the parabolic cases and  Yang and Zhang in  \cite{RunzeSiegel} for the Brujno cases.

We would like to try to understand what happens if we iterate parabolic renormalization on and on, starting from a polynomial in $\Per_1(e^{2\pi ip/q})$.
By analogy with Theorem~\ref{thm:SLY}, is it possible to find one or several invariant classes, analogue to $\DDD_2$?
Note that the class of finite-type maps as defined by A.~Epstein provides such an invariant class, but we seek both for a wider class and a finer description.

We start with a set of observations, of mathematical nature and of experimental computer pictures of the slices.
For $\lambda\in\C^*$, let us parametrize $\Per_1(\lambda)$ using the same convention as in \cite{Zakeri}:
for each $c\in\C^*$ we denote $P_c$ the unique cubic polynomial such that:
\begin{itemize}
  \item $P_c$ fixes $0$ with multiplier $\lambda$,
  \item the finite critical points of $P_c$ are $1$ and $c$.
\end{itemize}
Figure~\ref{fig:Per1Slice} shows the bifurcation locus in this parametrization in the case $p/q=2/5$, together with some color-coded information.

The parabolic point has at most two cycles of attracting axes because there are at most two finite critical points.
Let us classify the polynomials $P_c\in \Per_1(e^{2\pi ip/q})$ using an analogue of Milnor's classification:
\begin{itemize}
  \item \textbf{Type E.} (Exceptional) $P_c$ has two cycles of petals at $0$.
\end{itemize}
This happens only for finitely many parameters in $\Per_1(e^{2\pi ip/q})$.
For the maps that are not of type E,
there is at least one critical point in the cycle of immediate basins and the classification is:
\begin{itemize}
  \item \textbf{Type A.} (Adjacent) The two critical points belong to the immediate basin of the same axis (this includes the case where the two critical points coincide). 
  \item \textbf{Type B.} (Bitransitive) The two critical points belong to the immediate basins of  different axes.
  \item \textbf{Type C.} (Capture) Both critical points belong to the parabolic basin of $0$ but only one critical point belongs to an immediate basin.
  \item \textbf{Type D.} (Disjoint) Only one critical point belongs to the parabolic basin of $0$.
\end{itemize}

Type A, B and C form open subsets of this parameter space.
\begin{definition}\label{def:pearl}
We call \emph{pearl necklace} of $\Per_1(e^{2\pi ip/q})$ the set of components of type A and B.
\end{definition}

\begin{figure}[htbp]
    \centering
    \includegraphics[width=\textwidth]{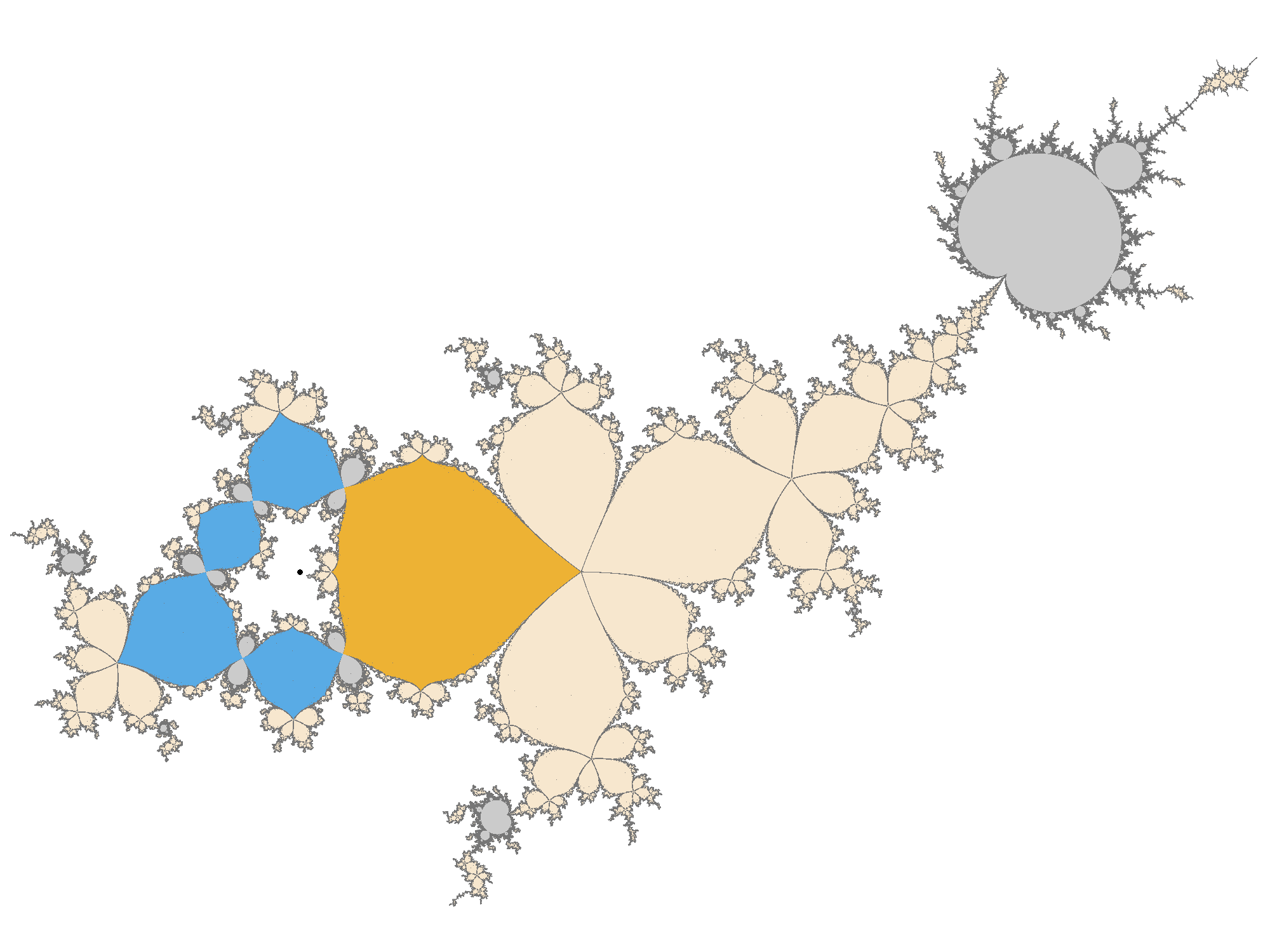}
    \caption{Slice $\Per_1(e^{2\pi i 2/5})$ of the cubic polynomials.
    We took Zakeri's parametrization \cite{Zakeri}: $P_c = \lambda z(1-\frac{1+c}{2c}z+\frac{1}{3c}z^2)$, for which $0$ is fixed of multiplier $\lambda$ and the set of finite critical points is $\{1,c\}$. 
    In blue the bitransitive case, in amber the adjacent case, in champagne the capture case and in different shades of white and grey the disjoint case: white when one critical point escapes to infinity, dark grey for the bifurcation locus, light grey for the rest. A tiny black dot marks the parameter $c=0$, which does not belong to the parameter space $\C^*$ of the slice. The 5 exceptional parameters are the contact points of the pearls (blue or amber components). A deeper description of the slices $\Per_1(e^{2\pi i p/q})$ can be found in \cite{runze3}.}
    \label{fig:Per1Slice}
\end{figure}

The pearl necklace is the initial motivation of the work that lead to the present article.
In this regard, the distinction between classes C and D is less important so we call C/D the union of these two classes.
The pearls are known to come in $q$ components indexed by the number of attracting axes one must skip to go from the immediate basin containing one of the critical points to the one containing the other. See for instance~\cite{blokh} and~\cite{Runze}. 
Only one pearl has type A. The pearls are simply connected and we call their union a necklace because their closures form a chain, i.e.\ if $q\geq 3$ the closure of two different pearls either intersect at one point if their index differ from $1$ and do not intersect otherwise.

In Figure~\ref{fig:Per1Slice} the colours code an indication of where the different classes are.

We are interested in parabolic enrichment for maps in the slice $\Per_1(e^{2\pi ip/q})$.
It has the interesting effect of leading to parabolic enrichment not only of the Julia sets but also of the bifurcation locus.
To simplify, we focus on the classification A,B, C/D, E.
Consider one of the two finite continued fraction developments of $p/q$: $[a_0,\cdots,a_m] = a_0+1/(a_1+1/\cdots)$.
Consider another rational number $p'/q'$ and let \[\theta_n = [a_0,\cdots,a_m+n+\frac{p'}{q'}].\]
Then $\theta_n\tends p/q$ as $n\to+\infty$, from the right or from the left according to which of the two expansions of $p/q$ has been chosen.

\medskip
\noindent{\bf Question}: What is the limit of the pearl necklace of $\Per_1(e^{2\pi i\theta_n})$, as $n\to +\infty?$
\medskip

The expected limit is expressed in terms of an appropriate parabolic renormalization operator, denoted $R_{p'/q',\pm}$  with $\pm=+$ or $\pm=-$ according to whether $\theta_n>p/q$ or $\theta_n<p/q$ (this depends on the parity of $m$).
The cases $\pm=+$ and $\pm=-$ are respectively called \emph{upper} and \emph{lower} parabolic renormalization.
This operator is so that $R_{p'/q',\pm} P_c$ has multiplier $e^{2\pi ip'/q'}$ (instead of $1$ for $R P_c$) at its fixed point $0$.
Our expectation is that the limit of the pearl necklace is related to the pearl necklace of the family $c\in\C^*-E \mapsto g_c := R_{p'/q',\pm} P_c$, which we define below.
Here $E$ denotes the set of $c\in\C^*$ so that $P_c$ is of type E, in which case $R_{p'/q',\pm} P_c$ is not defined.
The map $g_c$ has a parabolic point at $0$, of multiplier $e^{2\pi i p'/q'}$ and at most two cycles of attracting axes.
One way to define the pearl necklace for the family $c\mapsto g_c$ is as follows:

\begin{definition}\label{def:pnRf}
We call pearl necklace for the family $g_c$ 
the set of $c$ for which $g_c$ has only one cycle of attracting axes and such that the $q'$-th iterate $g_c^{q'}$ has degree bigger than $2$ on the immediate basin $B^*(A)$ of any attracting axis $A$ (this is independent of the axis).
\end{definition}

More precisely the results that we will prove imply in particular that $g_c$ is proper on its immediate basin, and that $g_c^{q'}|_{B^*(A)}$ is conjugate to either a degree 2 Blaschke product (we say that it has composition type $2$), a degree 3 Blaschke product (composition type $3$), or the composition of two degree 2 Blaschke products (composition type $2\circ 2$).
This distinction is actually better adapted to the renormalization operator than the A, B, \ldots, E classification, see Appendix~\ref{app:comp}.
More precisely we still want to keep an exceptional class, where the map has more than one cycle of petals, and call it type E too, and reserve the denomination composition type $3$, $2\circ 2$ and $2$ to the case where there is only one cycle of petals.
We expect the pearl necklace of the family $c\mapsto g_c$ to be the limit, in some sense, of the sequence of the pearl necklaces of the $\theta_n$.

\begin{figure}[htbp]
\centering
\noindent\includegraphics[width=\textwidth]{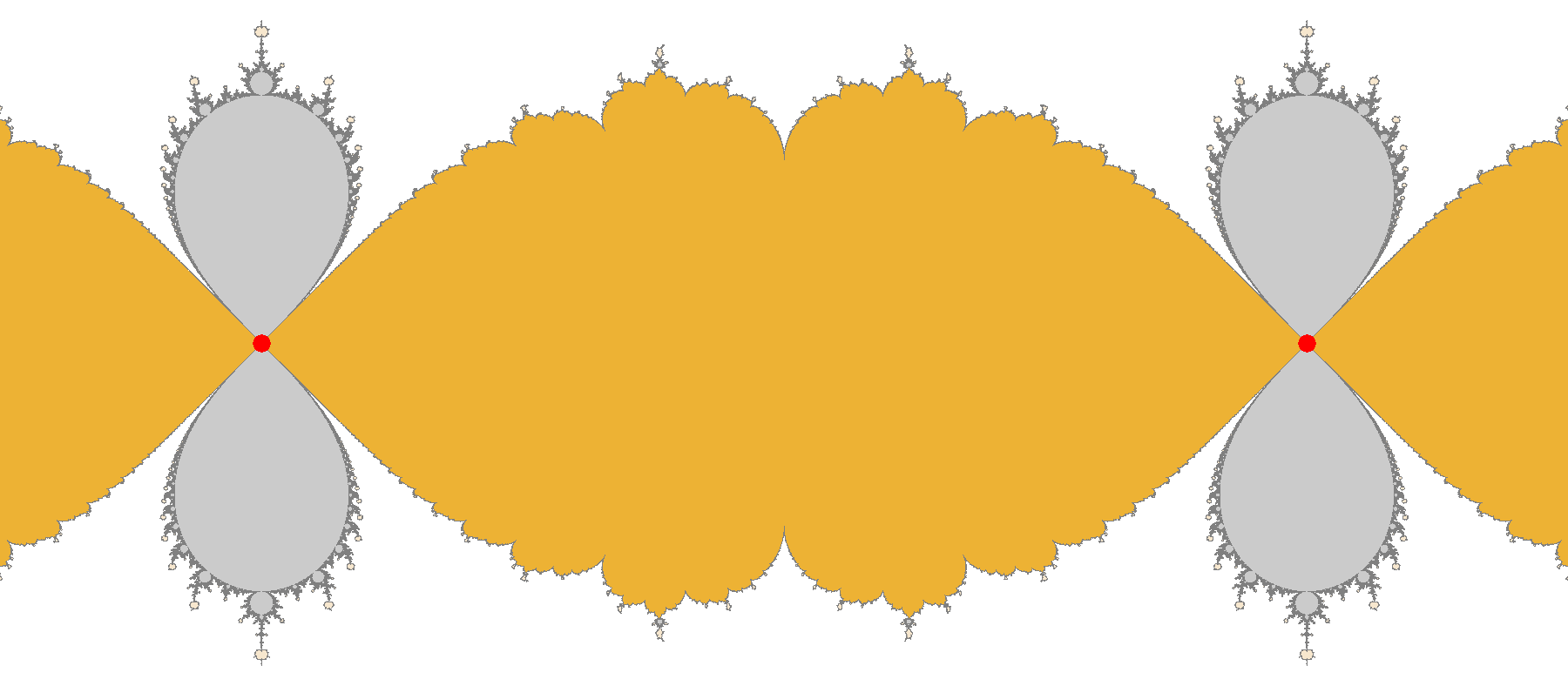}
\centering
\noindent\includegraphics[width=\textwidth]{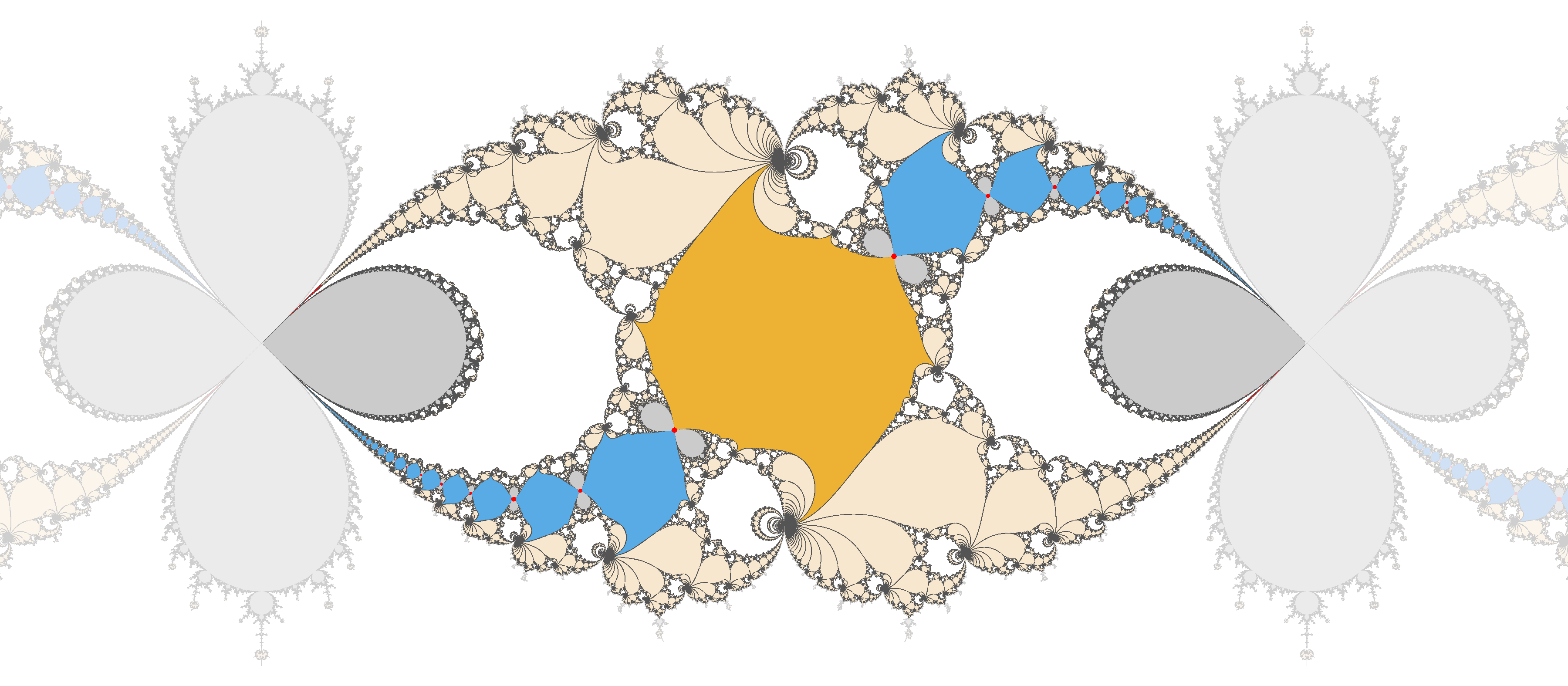}
\caption{$\Per_1(1)$ in log-coordinates, an enrichment and its infinite pearl necklace.}
\label{fig:per1implo}
\end{figure}

The top part of Figure~\ref{fig:per1implo} shows $\Per_1(1)$ in coordinates $\ell=\log(c)/i$, centred on parameter $\ell=0$ (i.e.\ $c=1$, for which $P_c(z)=z-z^2+z^3/3$ has only one finite critical point) and coloured according to the behaviour of $P_c$.
The white part is where one critical point escapes.
The image is periodic by $z\mapsto z+2\pi$.
In $c$-parameter, the necklace has a unique pearl, which is of type A, and is attached to itself at the parameter $c=-1$, which corresponds in the top image to the parameters $\ell \equiv \pi \bmod 2\pi$ indicated by red dots.
The tiny champagne parts are C-type components.
Light grey shows hyperbolic components and dark grey the boundary of the white part, which is also the bifurcation locus. 
\\
The bottom part of Figure~\ref{fig:per1implo} shows an enriched parameter space for $\Per_1(1)$, i.e.\ the image in coordinate $\ell=\log c$ of the parameter space for
\[ c\in\C^*-E\mapsto g_c := R_{p'/q',\pm} P_c
.\]
In the picture, we have chosen $p'/q' = 0/1$.
Inside the A-type component of the family $c\mapsto P_c$, we see a lot of new components.
The white ones are where one critical point escapes under the map $g_c$: Runze Zhang studied these white components in \cite{RunzeImplo}.
It is also where one critical value of $g_c$ escapes the domain of definition of $g_c$ under finitely many iterations of $g_c$.
In red, parameters for which $g_c$ has two cycles of attracting axes (in other cases, it has only one cycle of attracting axes).
In black, we show the bifurcation locus of the family $g_c$.
In grey, parameters for which one critical value of $g_c$ tends to an attracting cycle of $g_c$.
The other colours correspond to both critical values of $g_c$ being attracted to its parabolic point at $0$ and are chosen according to the type of $g_c$ on its basin: amber for type $3$, blue for type $2\circ 2$ and champagne for type $2$.

\medskip

Among other things, we prove here that, under renormalization, a type $2$ map can only give a type $2$ map (a fact already proved by Shishikura), a type $2\circ 2$ can give types $2\circ 2$, $2$ and E and a type $3$ map can give any of the three types or E. This is summed up in the following diagram.

\begin{equation}\label{eq:deg3diagram}
\begin{tikzcd}
  E & 3 \arrow[l] \arrow[loop right, looseness=6, out=25, in=-25] \arrow[dl] \arrow[d, crossing over]\\
  2\circ 2 \arrow[loop left, looseness=5] \arrow[r] \arrow[u] & 2 \arrow[loop right, looseness=6, out=25, in=-25]
\end{tikzcd}
\end{equation}

\begin{figure}
\begin{center}
\begin{tikzpicture}
\node at (0,0) {\includegraphics[width=12cm]{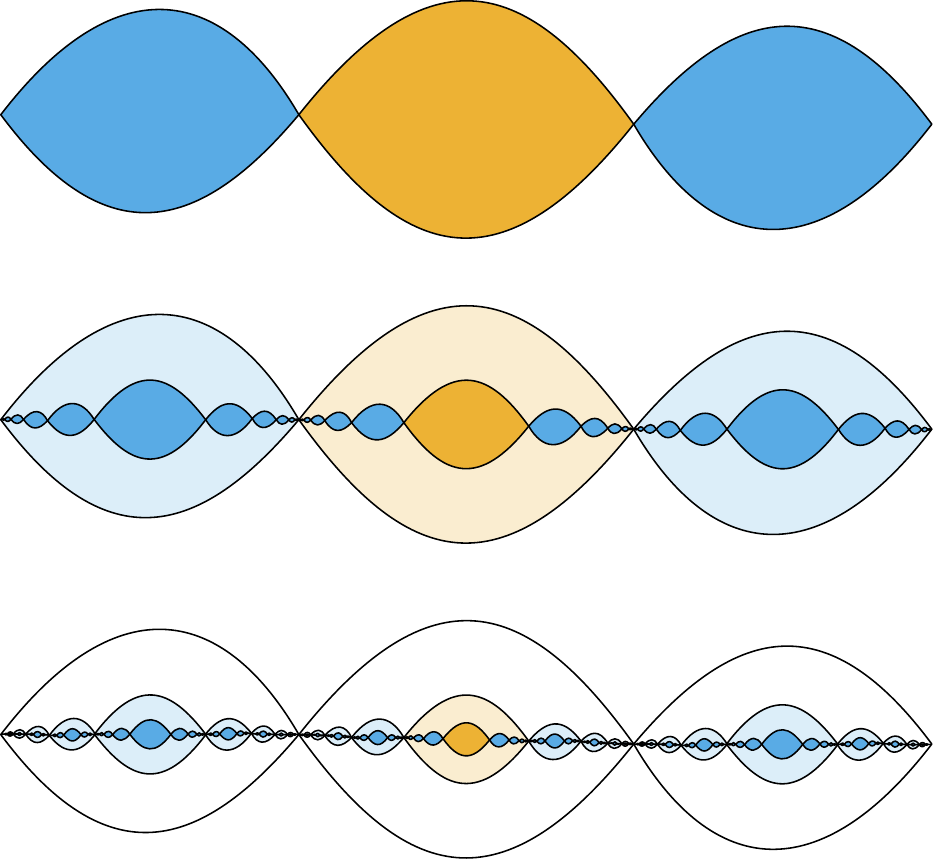}};
\end{tikzpicture}
\end{center}
\caption{\small Schematic representation of the pearl necklaces associated to successive enrichments in the parameter plane. The top row represents the pearl necklace of the family $c\mapsto P_c$, in the space $\Per_1(\exp(2\pi i 1/3))$, drawn in the $\log(c)/i$ coordinate, i.e.\ as on Figure~\ref{fig:per1implo}. 
The second row shows the pearl necklace of the family $c\mapsto R_{0/1} P_c$.
It is necessarily contained in the pearls of the family $c\mapsto P_c$.
In this simplified diagram we omit the contact points between the regions and their sub-region, though are an interesting aspect of the process.
The third row shows the pearl necklace of the family $c\mapsto R_{0/1} R_{0/1} P_c$.
The amber pearls represent the maps that are of composition type $3$ while the blue ones of composition type $2\circ 2$.
Type E parameters (two immediate basin components) are at the contact points between pearls.
The rest is of type $2$ (iff.\ there is only one critical point in the immediate basin).
}
\label{fig:iteratedParamEnrich}
\end{figure}

One may iterate parabolic renormalization and study the bifurcation locus and analogue of pearl necklace for $R_{p''/q'',s''} (R_{p'/q',s'} P_c)$, with $s',s''\in \{+,-\}$.
Figure~\ref{fig:iteratedParamEnrich} gives a schematic illustration.
By the composition type diagram above, these successive pearl necklaces are nested.
This suggest an interesting question: what is the limit (intersection) of these nested pearl necklaces when one iterates renormalization more and more?
Can we use this to create interesting non-Brjuno parameters $\theta$ for which the bifurcation locus in the slice $\Per_1(e^{2\pi i\theta})$ and the associated polynomials exhibit interesting/unusual behaviours?

In this article, we in particular prove and generalise Diagram~\ref{eq:deg3diagram} to maps with more critical points in their immediate basin.

\FloatBarrier

\subsection{Summary of results}\label{sub:content}

We will define classes of maps, denoted $\DDD_\tau$, associated to a composition type $\tau$ (see Definition~\ref{def:BiGrl}).
The composition type is a generalisation of the types we denoted $2$, $2\circ 2$ and $3$ in Section~\ref{sub:bkgd}.
A composition type is denoted $d_n\circ\cdots \circ d_1$ and is just a notation for the sequence $(d_1,\ldots, d_n)$.
A map in $\DDD_\tau$ is a holomorphic map $f$ defined on an open subset of $\C$ containing $0$, with a parabolic point at the origin with exactly one cycle of attracting axes, such that, denoting $r/s$ the rotation number of $f$ in irreducible terms, $f^s$ is conjugated on the immediate basin of any of its attracting axes (this is independent of the axis) to a Blaschke product, which is the composition of Blaschke products of degree $d_1$, $d_2$, \ldots, $d_n$.

We will define a notion of \emph{descendant} of composition types (Definition~\ref{def:desc}), and deduce from Theorems~\ref{thm:invGrl}, \ref{thm:main:grl} and~\ref{thm:main:grl:2}, that a map $f\in \DDD_\tau$ whose parabolic renormalization $R_{p/q,\pm}f$ has only one cycle of attracting axes, belongs to $\DDD_{\tau'}$ for a descendant $\tau'$ of $\tau$.
Notably, there is only finitely many critical points in the immediate basin of $R_{p/q,\pm} f$.
The definition of descendant is stated through another notion, we call \emph{gleaning} (see Definition~\ref{def:glean}), related to how critical points are distributed in the various \emph{virtual basins}, defined in Section~\ref{sub:grl} and denoted $V_n$.
A composition type is its own descendant.

The virtual basins $V_n$ are subsets of the immediate basin of $f$, indexed by $n\in \Z$, and were defined in the work of Douady and Lavaurs as specific connected components of the interior of the filled-in Julia set  (a.k.a.\ filled-in Julia-Lavaurs sets) of geometric limits (a.k.a.\ Lavaurs maps), see for instance Lemma~3.3.4 page~17 of \cite{lavaurs}.
We take here another approach: we define virtual basins as the direct image of the immediate basin of $Rf$ by the extended repelling Fatou parametrization $\psi_\rep$.

A key point of the proof of the main theorem is to justify that, though $Rf$ has infinitely many critical points, its restriction to the immediate basin is proper and has only finitely many critical points.
In \cite{Shi}, to prove this, Shishikura used the fact that $Rf$ has only one critical value but this is not the case any more here.
A substantial part of the work here focuses on proving that $f$ is proper from $V_n$ to $V_{n+1}$ and that $Rf:B^*_{Rf} \to B^*_{Rf}$ is equivalent to $f^{q-p}:V_p\to V_q$ for some $p<q\in\Z$.

One thing to note is that, with the level of generality in which we work, we for instance allow for maps whose immediate basins may have non locally connected boundary (see Appendix~\ref{app:non-tame} and figure~\ref{fig:zez}). 
(Actually we allow for maps whose definition domains boundaries are not locally connected either.)
As a consequence, we cannot use the techniques of \cite{LY} and deduce that the $V_n$ have locally connected boundary.
As such it may happen that $V_n$ is not contained, for any $n\in\Z$, in a repelling petal (or in an attracting petal, or in neither), a phenomenon related to the fact that the immediate basin of an attracting axis of $Rf$ may spiral arbitrarily far in one direction (or in the other, or back and forth).
See Figure~\ref{fig:large} for a sketch and Figure~\ref{fig:zez4} for an actual example.

\begin{figure}[htbp]
\centering
\includegraphics[width=10cm]{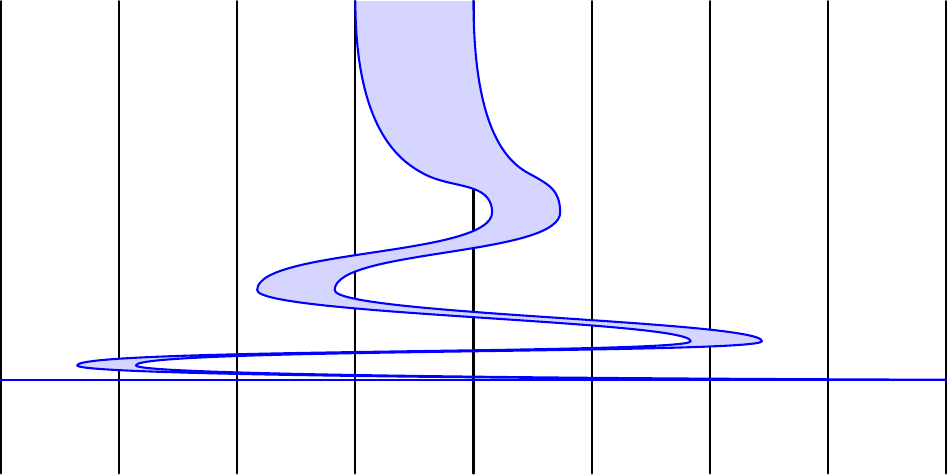}
\caption{A non-realistic sketch of the lift $U$ by $E:z\mapsto e^{2\pi iz}$ of the immediate basin of some $Rf$, such that $U$ would have a real part that is not bounded from below, nor above.}
\label{fig:large}
\end{figure}

\subsection{Acknowledgements}

This research was carried-out in Toulouse University.
Many discussions with Runze Zhang about enrichment in the parabolic slices and with Dimitri Le Meur about invariant classes for parabolic renormalization, helped in the maturation of the ideas of the present article.
The thesis of Le Meur included a study of when two horn maps are cover-equivalent, leading to the alternative proof proposed here of the fact that $Rf$ has finite type for $f\in \DDD_\tau$.

\section{Definitions and main result}\label{sec:def}

To simplify, we first treat the case when both $f$ and $Rf$ have rotation number $0$, i.e.\ $r/s=0/1$ and $p/q=0/1$ using the notations of Section~\ref{sub:content}.
We also choose to work with upper renormalization.
This will be the case through Sections~\ref{sec:def}, \ref{sec:tools} and~\ref{sec:pf-inv}.
The general case is treated in Section~\ref{sec:grl}.

\medskip

In order to state the main result, we need to introduce several definitions.

\subsection{Parabolic fixed points}\label{sub:ext}

We assume the classical theory of Leau and Fatou of parabolic fixed points known\footnote{See for instance Chapter~10 in \cite{Milnorbook}.}
and this section is here to fix the notations.

\medskip

For any function $f$ we denote by $\dom(f)$ its domain of definition. 

\begin{definition}\label{def:F}
Denote by $\cal F$ the class of $\C$-valued holomorphic maps $f$ defined on an open neighbourhood $\dom(f)\subset\C$ of $0\in \C$ such that $f(0)=0$ and $f'(0)=1$.
We also denote
\[ \cal E = \{ f\in\cal F\,\mid\, f''(0)=0\}. \]
All these maps have parabolic fixed point at the origin; $\cal E$ stands for \emph{exceptional} and refers to maps having more than one attracting axis.
\end{definition}

If $f\in\cal F$ and $f\neq\on{id}$ near $0$ then power series expansion of $f$ at $0$ takes the form $f(z) = z(1+ cz^{q}+\ldots)$ for some $c\in\C^*$.
We denote its basin of attraction as
\[ B = B(0) = \{ z\in \dom(f) \mid f^n(z)\to 0 \hbox{ and  } \forall n\in \N,\, f^n(z)\neq 0 \} .\]
Recall that any point in $B$ has an orbit that tends to $0$ tangentially to one of the $q$ half lines defined by $cz^{q} \in \R_{<0}$ and called \emph{attracting axes}.
Local orbits under $f^{-1}$ tend to $0$ tangentially to one of the $q$ \emph{repelling axes}, defined by $cz^q \in \R_{>0}$.

\begin{lemma}\label{lem:nonc}
  Let $f\in \cal F$. We assume that $f\neq\id$ near $0$ (otherwise $B$ is not defined).
  Then $B\neq \C^*$.
\end{lemma}
\begin{proof}
  Otherwise $f^n$ would converge uniformly to $0$ on the unit circle, which implies by the Cauchy integral formula that $(f^n)'(0)\tends 0$.
  But $f'(0)=1$.
\end{proof}

The classical theory provides $2q$ domains, one for each axis, that cover a punctured neighbourhood of $0$, more precisely $q$ \emph{attracting petals} $P_\att$, together with $q$ holomorphic conjugacies $\phi_\att$, called \emph{attracting Fatou coordinates}, defined on $P_\att$ satisfying $\phi_\att\circ f = \phi_\att+1$, and $q$ \emph{repelling  petals} $P_\rep$, together with $q$ holomorphic conjugacies $\phi_\rep$, called \emph{repelling Fatou coordinate}, defined on $P_\rep$ satisfying $\phi_\rep\circ f^{-1}=\phi_\rep+1$.
In each attracting petal, orbits tend to $0$ tangentially to its associated attracting axis, Similarly in each repelling petal, backward orbits tend to $0$ tangentially to its associated repelling axis.

Let \[ B(P_\att) = \{ z \in \dom(f) \mid f^n(z) \hbox{ eventually belongs to } P_\att \}. \]
It is also equal to the set of points whose orbit tends to $0$ tangentially to the corresponding attracting axis, so if we denote $\Delta_\att$ the axis we can also denote $B(\Delta_\att)=B(P_\att)$.
The set $B(0)$ is the disjoint union of those $q$ basins.

One can uniquely extend $\phi_\att$ to the whole basin of the corresponding petal as a holomorphic map still satisfying the semi conjugacy relation $\phi_\att\circ f = \phi_\att+1$.
Similarly, one can uniquely extend $\psi_\rep=\phi_\rep^{-1}$ on the following domain:
\[ \Xi_\rep:\{z+n \,\mid\, z\in \H_{<r}, \ n\in \N, f^n(\psi_\rep(z)) \hbox{ is defined}\} \]
where $\H_{<r} =\{z\in\C\mid \Re{z}<r\}$, into a holomorphice map satisfying
\begin{equation}\label{eq:psirepsc}
  \psi_\rep (z+1) = f\circ\psi_\rep(z)\text{ for every }z\in\Xi_\rep\text{ such that }z+1\in\Xi_\rep.
\end{equation}
The extended maps, like the orignal maps, are unique up to composition with translation on the left for $\phi_\att$ and on the right for $\psi_\rep$. 
Such extensions are central in the work of Douady-Lavaurs on parabolic enrichment (see for instance \cite{lavaurs,DouadyImplosion}) and the subsequent research in the field.

\begin{center}
\it
From now on $\phi_\att$ and $\psi_\rep$ refer to extended attracting and repelling Fatou coordinates.
\end{center}

The \emph{immediate basin} of an attracting petal $P_\att$ is the connected component of its basin $B(P_\att)$ that contains $P_\att$ and is denoted $B^*(P_\att)$.
It only depends on the attracting axis so we can denote it $B^*(\Delta_\att)$.
If $f\in\cal F$ and $f''(0)\neq 0$, there is only one attracting axis, and in this case we denote its immediate basin $B^*$, or $B^*_f$ if there are more than one map $f$.

\subsection{Parabolic renormalization}\label{sub:ren}

Recall that we work with the case where $f$ and the map $Rf$, which we are about to define, have rotation number $0/1$ and with the choice of upper renormalization.
We will thus denote $R=R_{p/q,\pm}=R_{0/1,+}$.

\medskip

For a map $f\in\cal F\setminus \cal E$, i.e.\ such that $f''(0)\neq 0$, we define in this section its parabolic renormalization. 

Let $\phi_\att$ and $\psi_\rep$ denote the extensions defined in Section~\ref{sub:ext}.
In a neighbourhood of $+i\infty$ the following composition $$h_f=\phi_\att\circ\psi_\rep$$ is well defined and commutes with the translation by $1$ and is called the lifted \emph{horn map}.
Hence it can be viewed as a map of $\C/\Z$ defined in a neighbourhood of  $+i\infty$ which is called the horn map.
By abuse of notation we denote it also by $h_f$.

For $\sigma\in\C$, let $T_\sigma (z)=z+\sigma$.
This translation is well-defined on the cylinder $\C/\Z$ and called $T_\sigma$  by abuse of notation.
Let $E(z) = e^{2\pi i z}$ and denote $\pi : \C\to\C/\Z$ the canonical projection.
It induces a conformal isomorphism $\ov E:\C/\Z\to\C^*$ via the commuting diagram
\[
\begin{tikzcd}
\C \ar[rd, "E"] \ar[d, "\pi"'] &  \\
\C/\Z \ar[r,"\ov E"'] & \C^*
\end{tikzcd}
\]
and sends the upper end of the cylinder $\C/\Z$ to $0$.
Now projecting $T_\sigma \circ h_f$ via $\ov E$, we get a map
\[\ell_\sigma := \ov E \circ\, T_\sigma\circ h_f\circ \ov E^{-1}\]
which is defined on a neighbourhood of $0$.
Since Écalle \cite{Ec} and Voronin \cite{Vo} we know
$0$ is a removable singularity and, more precisely, $\ell_\sigma$ can be extended to a holomorphic function satisfying $\ell_\sigma(0)=0$ and $\ell_\sigma'(0)\neq 0$.
\[
\begin{tikzcd}
\C/\Z  \ar[r, "T_\sigma\circ h_f"] \ar[d, "\ov{E}"'] & \C/\Z  \ar[d, "\ov{E}"]\\
\C^* \ar[r," \ell_\sigma"'] &\C^* & 
\end{tikzcd}
\]
Let $\dom^*(\ell_\sigma)$ be the connected component containing $0$ of $\dom {\ell_\sigma}$ (it is independent of $\sigma$).
Now there exists some $\sigma_0$, that depends on $f$, such that $\ell_{\sigma_0}'(0)=1$, and we make the following

\begin{definition}\label{def:ren}
The \emph{parabolic renormalization} $Rf$ of $f\in \cal F\setminus\cal E$ is the restriction
\[ Rf = \ell_{\sigma_0}|_{\dom^*(\ell_{\sigma_0})}.\]
\end{definition}

Note that actually this is the upper parabolic renormalization and that there is also a lower parabolic renormalization, whose treatment is analogous, where $E=e^{2\pi iz}$ is replaced by $E(z)=e^{-2\pi iz}$ in the formulas, so as to send the lower end of the cylinder to $0$ instead of the upper end.

Let us stress that $Rf$ is only defined up to conjugacy by a scaling $z\mapsto \lambda z$, $\lambda \in\C^*$.
Indeed, we did not normalize the repelling and the attracting Fatou coordinates, which are only unique up to addition of constants, and the condition $\ell_\sigma'(0)=1$ imposes only one relation on these two constants.

Let us give a few immediate classical properties:

\begin{proposition}\label{prop:icph} A point $u\in\dom h_f$ is a critical point of $h_f$ if and olny if $\exists n\in\N$ such that 
$w:=\psi_\rep(u-n)$ is precritical for $f$, i.e.\ there exists $n\in\N$ such that $f^n(w)\in \on{Crit}(f)$.
The set of critical values of $Rf$ is contained in the set of $z=E(\sigma_0+\phi_\att(v))$ where $v$ varies in the set of critical values of the restriction $f:B^*_f\to B^*_f$.
The set of critical points of $Rf$ is the set $E(\on{Crit}(h_f)\cap \Omega)$.
It is also the set of $z = E(u)\in E(\Omega)$ such that $w=\psi_{\rep}(u)$ is defined and pre-critical for $f$.
In this case, the local degree of $Rf$ at such $z$ is given by the product\footnote{This infinite product is well-defined because for $n$ negative enough, $\psi_\rep(u+n)$ is in a repelling petal of $f$ and for $n$ positive enough, $\psi_\rep(u+n)$ is in a attracting petal, so the local degree is $1$ but for a finite set of values of $n\in\Z$.} of local degrees of $f$ at the points $\psi_\rep(u+k)$ for all $k\in\Z$.
\end{proposition}

\begin{remark}
If $f''(0)=0$, i.e.\ if there are more than one attracting axis, there is more than one parabolic renormalizations that we could associate to $f$. Here we choose not to define parabolic renormalization in this case.
\end{remark}

In section~\ref{sub:ren} we defined \[\Omega = E^{-1}(\dom Rf).\]
Note that $\Omega$ is also the connected component of $\dom h_f$ that contains an upper half plane.

\begin{lemma}\label{lem:prutibsf}
  $\psi_\rep(\Omega) \subset B^*_f$
\end{lemma}
\begin{proof}
From the definition of $\dom h_f$, it follows that 
$\psi_\rep(\Omega)\subset \psi_\rep(\dom h_f) \subset B_f$.
The image by $\psi_\rep$ of a high enough upper half plane is a sepal for the parabolic point of $f$, in particular it meets some attracting petal of $f$.
So the connected set $\psi_\rep(\Omega)$ is contained in the component of $B_f$ that contains the petal.
\end{proof}

\begin{corollary}\label{cor:cccuh}
  The set $\Omega$ is the connected component of $\psi_\rep^{-1}(B^*_f)$ containing an upper half plane.
\end{corollary}

\subsection{Renormalization for  Blaschke products}\label{sub:rcb}

Parabolic renormalization as defined in Section~\ref{sub:ren} does not apply directly to  Blaschke products having two attracting axes. We define for them in this section a specific version adapted to our purposes.

\begin{definition}\label{def:Grond} We call $\cal G$, the set of finite Blaschke products $G$ whose extension $F$ has a 
parabolic fixed point at $z=1$, with two attracting petals.
\end{definition} 

In such a situation, $F$ necessarily has two repelling axes at $z=1$, one pointing up and one pointing down.
There are four associated extended horn maps. To define $RG$ we select one as follows. 
Let $P_\rep$ be a repelling petal for the axis pointing up.
The repelling Fatou coordinate $\phi_\rep:P_\rep\to\C$ is chosen so that $\phi_\rep(1/\overline z) = \overline{\phi_\rep(z)}$, or equivalently so that the unit circle is mapped to the real line.
As before, the map $\psi_\rep=\phi_\rep^{-1}$ can be extended on $\C$ with $\psi_\rep^{-1}(\D)=\H$.
Let $P_\att$ be an attracting petal contained in the unit disk and $\phi_\att$ an attracting Fatou coordinate on this petal. 
Extend $\phi_\att$ on the whole basin $\D$.
Then the composition $h_G := \phi_\att\circ\psi_\rep$ of the extended maps has a domain of definition equal to the upper half plane $\H$.

We then define a map $RG:\D\to\C$ by setting $RG(0)=0$ and for $z\in\D^*$,
\[RG (z) = \ov E \circ\, T_\sigma\circ h_G\circ \ov E^{-1}(z),\] 
for the $\sigma \in\C$, unique modulo $\Z$, such that $RG'(0) = 1$.

\subsection{Blaschke models}\label{sub:blmod}

It is known that Blaschke products provide model maps for the restriction of parabolic maps to their immediate basin, under some conditions which
we describe here.

Consider a map $f\in\cal F$ with an immediate basin $B^*$ such that:
\begin{itemize}
  \item $B^*$ is simply connected,
  \item $f: B^*\to B^*$ is proper.
\end{itemize}
By the Riemann mapping theorem there exists a conformal map $\phi:B^* \to\D$.
The map $\phi\circ f\circ \phi^{-1}$ is a proper holomorphic self map of $\D$, and such a map is necessarily a finite Blaschke product $G$.

The map $G$ is the restriction to $\D$ of a rational map $F:\wh\C \to \wh\C$ that commutes with $z\mapsto 1/\bar z$.
Moreover:
\begin{lemma}\label{lem:bppp}
  The map $F$ has a parabolic fixed point on $\partial \D$, with two attracting axes.
\end{lemma}
\noindent In other words, $G\in \cal G$ where $\cal G$ is defined in Definition~\ref{def:Grond}.

Lemma~\ref{lem:bppp} is proved in particular cases in \cite{Orsay2}, Exposé~IX, Section~II.1, Proposition~4.c and \cite{LY} Section~2.4, Theorem~2.9, and we believe the proofs adapt.
Just to be safe, we include in Appendix~\ref{app:pflembpp} a proof of our case, which is basically the one of \cite{Orsay2} with the final argument modified.

By the Denjoy-Wolff theorem\footnote{See the proof of Lemma~\ref{lem:bppp} for an alternative proof.}, the disk $\D$ is contained the basin of one attracting axis, and its reflection w.r.t.\ the unit circle in the basin of the other attracting axis.
The basin of the two axes are open and disjoint.
It follows that $\D$ is the whole basin of one axis, and its reflection the whole basin of the other, and that $J(F)$ is the unit circle.
The map $F$ has no other non-repelling cycles.

\medskip

To define our invariant classes we first introduce some classes of Blaschke products.

\begin{definition}\label{def:BiGrl}
  Given a finite sequence of integers $d_1$, \ldots, $d_n$ with:
  \begin{itemize}
    \item $n\geq 1$,
    \item $\forall k$, $d_k\geq 2$,
  \end{itemize}
  we denote $\cal B_{d_n\circ \cdots \circ d_1}$ the set of maps of the form
  \[G = G_n \circ \cdots \circ G_1\]
  such that:
  \begin{itemize}
    \item $G_k$ is a degree $d_k$ Blaschke product,
    \item the extension $F$ of $G$ has on the unit circle a parabolic point with two attracting axes.
  \end{itemize}
  We call $d_n\circ \cdots \circ d_1$ the \emph{composition type} of $\cal B_{d_n\circ \cdots \circ d_1}$.
\end{definition}

\begin{remark}
  A word of caution: different sequences $d_i$ with the same value of their product do not necessarily yield disjoint sets $\cal B_{d_n\circ \cdots \circ d_1}$.
  In other words, a given Blaschke product may decompose into several ways.
  Obviously $\cal B_{d_n\circ \cdots \circ d_1} \subset \cal B_{d_n\times\cdots\times d_1}$ but there are also more subtle intersections, like $z^6 = (z^3)^2 = (z^2)^3$.
\end{remark}

\begin{remark}
  We could have imposed that the extension of each $G_n$ has at $z=1$ a parabolic point with 2 attracting axes, which ensures this is the case too for $G$, and which is always realizable up to möbius changes of variables. However, if $n>1$, the dynamical system associated to a single $G_k$ is not really relevant.
\end{remark}

\subsection{Composition type}\label{sub:dd}

We use the classes of Blaschke products introduced in Section~\ref{sub:blmod} to define classes of maps for which we will investigate invariance properties under the parabolic renormalization operator (defined in Section~\ref{sub:ren}).

\begin{definition}\label{def:dynGrl}
  Denote by $\tau = d_n\circ \cdots \circ d_1$ (this is a purely symbolic notation for the $n$-uplet $(d_1,\ldots,d_n)$) and call it a \emph{composition type}.
  We define $\DDD_\tau$ as the set of maps $f\in \cal F$ such that $f''(0)\neq 0$ (i.e.\ $f\notin \cal E$, i.e.\ it has only one attracting axis) and the restriction $f:B^*\to B^*$ is conjugated to an element of  $\cal B_\tau$.
  For a map $f\in \cal E$ with $f\neq\id$ near $0$, given any attracting axis $A$ and associated immediate basin $B^*(A)$,
  if $f:B^*(A)\to B^*(A)$ is conjugated to an element of  $\cal B_\tau$ then we say that $f$ has type $\tau$ on $B^*(A)$.
\end{definition}

In particular, for such maps, the immediate basin is simply connected.

\medskip

We are going to define a notion of \emph{descendant}. For this we start with the following notion, which we call \emph{gleaning} and present, first through an analogy, then formally in Definition~\ref{def:glean}.
Consider $m$ bowls numbered from $i=1$ to $i=m$, containing each $c_i\geq 1$ pebbles.
A person fills $n$ purses one by one with pebbles taken from the bowls.
For each purse, the pebbles are to be taken from a single bowl, and at least one must be taken.
The bowls may be visited in any order, and there is no necessity to take every pebble in a bowl, nor to visit every bowl.
The purses are numbered from $1$ to $n$ in the order in which they were filled.
In the end, the purse number $i$ contains  $c'_i\geq 1$ pebbles, all taken from a bowl whose index we denote $\beta(i)$.

\begin{definition}\label{def:glean}
  Consider a finite sequence $(c_1,\ldots,c_m)$ of positive integers with $m\geq 1$.
  A \emph{gleaning} of $(c_1,\ldots,c_m)$ is any $(c'_1, \ldots,c'_n)$ for which
  \begin{enumerate}
      \item $n\geq 1$,
      \item $\forall i\in \{1,\ldots,n\}$, $c'_i\geq 1$, 
      \item $\exists\beta : \{1,\ldots,n\} \to \{1,\cdots,m\}$
      such that $\forall j$,
      \[\sum_{i\in\beta^{-1}(\{j\})} c'_i \leq c_j.\]
  \end{enumerate}
\end{definition}

\begin{figure}[ht]
    \centering
    \includegraphics[scale=0.75]{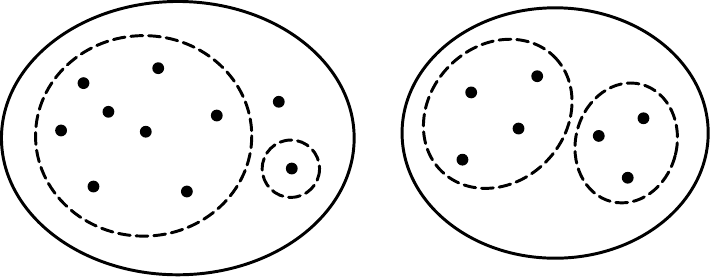}
    \caption{$(1,3,4,8)$ as a gleaning of $(10,7)$. The black dots are the pebbles, the bowls are draw in solid lines and the purses in dashed lines.}
    \label{fig:glanage-1}
\end{figure}

For instance, start from $(c_1,c_2)=(10,7)$. Then $(c'_1,\ldots,c'_4)=(1,3,4,8)$ is a gleaning, where the 8 pebbles are necessarily taken from the bowl with 10, then the 4 and the 3 from the other one, and the last 1 from the initial bowl.
See Figure~\ref{fig:glanage-1}.
So are $(10,7)$, $(1,\ldots,1)$ too provided there are at least 1 and at most 17 entries.
On the other hand, $(4,8,4)$ cannot be a gleaning of $(10,7)$.

Note that any permutation of a gleaning is a gleaning, and that the set of gleanings of a given finite sequence $s$ is identical if we apply any permutation to the entries of $s$.
Also, the gleaning of a gleaning of $s$ is a gleaning of $s$.

\begin{definition}\label{def:glean2}
  Consider a finite sequence $(c_1,\ldots,c_m)$ of positive integers with $m\geq 1$.
  A finite collection of gleanings $c'[1],\ldots,c'[p]$ of $(c_1,\ldots,c_m)$ is called \emph{disjoint} if 
  $\forall j$,
      \[\sum_{k=1}^p \sum_{i\in\beta^{-1}(\{j\})} c'_i[k] \leq c_j.\]
  The idea is that each gleaning of the collection is supposed to take different pebbles.
\end{definition}

\medskip

The bowls represent the disks $\D$ in a composition of Blaschke products, the pebbles are their critical points and the purses will be virtual parabolic basins.
Recall that a Blaschke product of degree $d$ has $d-1$ critical points (we count with multiplicity).

\begin{definition}\label{def:desc}
  A \emph{descendant} of the composition type $d_m\circ \cdots \circ d_1$ is any composition type $d'_n\circ \cdots \circ d'_1$ such that $(d'_1-1,\ldots,d'_n-1)$ is a gleaning of $(d_1-1,\ldots,d_m-1)$.
\end{definition}

\begin{definition}\label{def:desc2}
  A finite collection of descendants of $d_m\circ \cdots \circ d_1$ is called \emph{disjoint} if the corresponding collection of gleanings of $(d_1-1,\ldots,d_m-1)$ is disjoint.
\end{definition}

Note that, unlike the terminology may suggest, $d_m\circ \cdots \circ d_1$ is always one of its own descendants. Note also the order reversal in the writings between $d_m\circ \cdots \circ d_1$ and $(d_1-1,\ldots,d_m-1)$.

For instance descendants of the composition type $3$ are $2\circ 2$, $3$ and $2$.
This is because gleanings of $(2)$ are $(1,1)$, $(2)$ and $(1)$.
Descendants of $2\circ 2$ are $2\circ 2$ and $2$.
The only descendant of $2$ is $2$.

As for gleanings, any permutation of a descendant is a descendant and the set of descendant of a given finite sequence $s$ is unchanged if we permute the entries of the sequence $s$.
Also, a descendant of a descendant of $s$ is a descendant of $s$.

Here is a table of the descendants of the composition type $4$ together with the corresponding gleanings:
\[
\begin{array}{cc}
 \text{Gleaning of (3) } \  &  \text{Descendant of 4} \\
   (3)   &  4\\
      (2)   &3  \\
         (1)   & 2 \\
        (1,2)    & 3\circ 2  \\
        (2,1)    & 2\circ 3  \\
(1,1)    & 2\circ 2 \\
(1,1,1)    & 2\circ 2\circ 2 
\end{array}
\]

\subsection{Main result}

We are now ready to state the main result of this article, which we prove in Section~\ref{sec:pf-inv}.

\begin{theorem}\label{thm:invGrl}
Let $\tau = d_m\circ \cdots \circ d_1$ be a composition type.
If $f\in \DDD_\tau$ then $Rf$ has at least one attracting axis and there exists a map associating to each attracting axis $A$ of $Rf$ a descendant $\tau'[A]$ of $\tau$ such that $Rf$ has type $\tau'[A]$ on the immediate basins of $A$ and such that the set of $\tau'[A]$ for $A$ varying in the attracting axes of $Rf$ forms a disjoint collection (of descendants of $\tau$).
\end{theorem}

Recall that the composition type is not necessarily unique.
If $f$ has several types, the statement above ensures for each of these type the existence of a choice of descendant composition types satisfying the conclusion.

\begin{corollary}[Invariance]\label{crr:inv}
Let $\tau = d_m\circ \cdots \circ d_1$ be a composition type.
If $f\in \DDD_\tau$ and moreover $Rf\notin\cal E$ then $Rf\in \DDD_{\tau'}$ for some descendant $\tau'$ of $\tau$.
\end{corollary}

The invariance of $\DDD_2$ (Theorem~\ref{thm:SLY}) is a particular case of Theorem~\ref{thm:invGrl} since the only disjoint descendant collection of the composition type $2$ is composed of only one descendant, equal to the composition type $2$ (this prevents the existence of more than one attracting axis for $Rf$).

\medskip

Here are the finite sequences (up to permutations) that are \emph{gleanings} of $(4)$ organized by the elementary operations : suppression of $1$ (the horizontal arrows) or simple breaking (decomposition of an integer into a sum of two positive integers). Gleanings of $(4)$ include $(4)$. See also Figure~\ref{fig:glanage-2}.
\[
\begin{tikzcd}[column sep=small]
 (4)\arrow[d]\arrow[dr]&&&\\
 (2,2)\arrow[dr] &(3,1)  \arrow[d] \arrow[r ] & (3) \arrow[d]  &    \\
  & (2,1,1)  \arrow[d] \arrow[r ] & (2,1) \arrow[d] \arrow[r ] &   (2) \arrow[d] \\
   &(1,1,1,1) \arrow[r] &   (1,1,1) \arrow[r] &  (1,1) \arrow[r ] & (1)
\end{tikzcd}
\]

\begin{figure}[ht]
    \centering
    \includegraphics[scale=0.75]{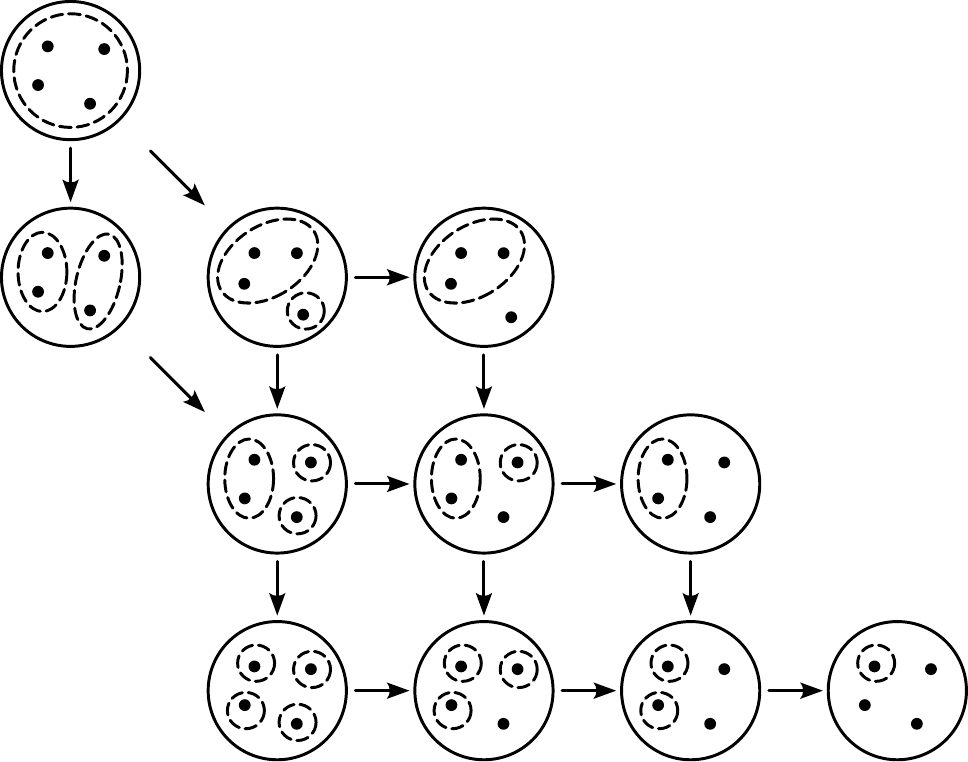}
    \caption{All the gleanings up to permutation, depicted width dashed contours, of a single bowl with 4 pebble.}
    \label{fig:glanage-2}
\end{figure}

As a consequence, here are the corresponding \emph{descendants} of $5$, still up to permutation. 

\[
\begin{tikzcd}[column sep=small]
 5\arrow[d]\arrow[dr]&&&\\
 3\circ 3\arrow[dr] & 2\circ 4  \arrow[d] \arrow[r ] & 4 \arrow[d]  &    \\
  & 2\circ 2\circ 3  \arrow[d] \arrow[r ] & 2\circ 3 \arrow[d] \arrow[r ] &   3 \arrow[d] \\
   &2\circ 2\circ 2\circ 2 \arrow[r] &   2\circ 2\circ 2 \arrow[r] &  2\circ 2 \arrow[r ] & 2
\end{tikzcd}
\] 

\medskip

Note that in this article, we do not address the question of whether all descendants are realizable, i.e.: \emph{given a descendant $\tau'$ of $\tau = d_m\circ \cdots\circ d_1$, does there exist a map $f\in\DDD_{\tau'}$ such that $Rf\in\DDD_{\tau}$?}

\begin{remark}
  We could use a finer notion.
  Indeed, we do not distinguish a critical point of multiplicity $m>1$ from $m$ critical points of multiplicity one: they are counted as $m$ pebbles in both cases. However, in the gleaning process that happens in the proof of Theorem~\ref{thm:invGrl}, these $m$ pebbles should never be separated.
  The appropriate image should be that those pebbles have an (integer) weight, and for instance a degree $4$ Blaschke product with a double critical point and two simple ones would be denoted not as type~$(3)$ but as type $(2+1+1)$.
  We could also keep track of the identity of each pebble.
  However to keep things simple we choose not to go to these levels of fineness.
\end{remark}

\section{Tools}\label{sec:tools}

This section describes known results about branched coverings and simply connected domains that we will use to prove the main theorem.
More precisely, one of the main difficulties is to prove that the immediate basin $B^*_{Rf}$ of $Rf$ is simply connected, that $Rf$ is proper on this set, and to recognize the branched covering that $Rf$ induces on $B^*_{Rf}$.

\subsection{Singular values and branched coverings}

Here we allow a Riemann surface to be non-connected (a better name would be a complex manifold of dimension $1$).
A path $\gamma:[a,b)\to X$, $a<b\in\R\cup\{+\infty\}$ is said to \emph{leave every compact subset} if for all compact $K\subset X$, $\gamma(t)\notin K$ for all $t$ close enough to $b$, i.e.\ it leaves and never comes back.
In modern terms, if and only if $\gamma(t)$ tends to infinity in the one-point (Alexandrov) compactification of $X$.

We recall the following basic definitions and a well-known, yet very important, theorem.

\begin{definition}
    Let $X,Y$ be Riemann surfaces.
    Let $f : X\to Y$ be holomorphic.
    A \emph{critical value} of $f$ is the image by $f$ of a critical point.
    An \emph{asymptotic value} of $f$ is any $y\in Y$ such that there exists a path $\gamma: [a,b)\to X$ such that,\footnote{Up to a reparametrization we can assume that $[a,b)=[0,1)$.} as $t\to b$, $\gamma$ leaves every compact subset of $X$ and $f(\gamma(t))\to y$.
    A \emph{regular value} of $f$ is any $y\in Y$ for which there exists an open neighbourhood $V$ that is well-covered\footnote{I.e. the restriction $f: f^{-1}(V) \to V$ is a covering.} by $f$ over $V$.
    And a \emph{singular value} is any $y$ that is not a regular value.
    We denote by $cv(f)$, $av(f)$, $sv(f)$ the sets of critical, asymptotic and singular values of $f$.
\end{definition}

\begin{remark}\label{rem:cov}
The restriction $f:X-f^{-1}(sv) \to Y-sv$, where $sv=sv(f)$, is a covering and actually $Y-sv$ is the biggest open subset of $Y$ over which $f$ can be a covering.  
\end{remark}

The following result appears first in \cite{EL}, after Theorem~4.5 on page~621, where they indicate that it can be derived from works of Iversen and Nevanlinna. A direct proof from first principles can be found in Section~2 of \cite{Rempe}.
\begin{theorem}\label{thm:sv}
    The set of singular values is the closure of the union of the set of critical values and the set of asymptotic values:
    \[sv(f) = \overline{av(f)\cup cv(f)} = \overline{av(f)} \cup \overline{cv(f)}.\]
\end{theorem}

Unlike the critical values, the notion of asymptotic and singular value depend on the set $Y$: for instance the injection $z\mapsto z:\D\to\C$ has for set of asymptotic values the unit circle, and the same singular values, while $z\mapsto z: \D\to \D$ has no singular values and no asymptotic values.
The following proposition relates the critical and asymptotic values of a map and its restrictions.
The standard restriction in the domain is denoted as usual: $f|_U$.
For the restriction to a subset $V$ of the range, we need that $f$ takes values in $V$, and then denote $f|^V$ the restriction.
For the restriction both in the domain and range, that we denote $f|_U^V$, we require that $f(U)\subset V$.

\begin{proposition}\label{prop:av}
  Let $f:X\to Y$ be a holomorphic map between Riemann surfaces.
  \begin{enumerate}
    \item For open $V\subset Y$, such that $f(X)\subset V$,
  \[av(f|^V) = av(f)\cap V.\]
    \item\label{item:prop:av:2} For open $U\subset X$,
    \[av(f|_U) \subset av(f) \cup f(\partial U).\]
    \item For open subsets $U\subset X$ and $V\subset Y$ such that $f(U)\subset V$,
    \[ av(f|^V_U) \subset V\cap(av(f) \cup f(\partial U)).\]
  \end{enumerate}
\end{proposition}
The boundary of $U$ is taken relative to $X$.
\begin{proof}
1.\  
We want to compare $av(f|^V)$ and $av(f)$: in both cases, they are the limits as $t\to 1$ of $f|^V(\gamma(t))=f(\gamma(t))$ for the same set of paths $\gamma:[0,1)\to X$: those that leave every compact subset of $X$.
The only difference is that in one case we consider the possible limits in the space $Y$ and in the other case, in the subset $V$.
But $f|^V(\gamma)$ converges  in $V$ iff ($f(\gamma)$  converges in $Y$ and the limit is in $V$).  The result follows. 

2.\ 
Let $y\in av(f|_U) $, then
there exists a path    $\gamma: [0,1)\to U$ such that, as $t\to 1$, $\gamma$ leaves  every compact subset of $U$ and $f(\gamma(t))\to y$.
Either,  $\gamma$ leaves  every compact subset of $X$ then  $y\in av(f) $, or there exists a compact set $K\subset X$ and a sequence $t_n$ such that $\gamma(t_n)\in K$. Up to passing to a subsequence we may assume that $\gamma(t_n)$ converges to some $x\in K$. Since $\gamma(t_n)\in U$, $x\in \ov U$.
Since  $\gamma$ leaves  every compact subset of $U$,    $x\in \partial U\cap K$, so that $y\in f(\partial U)$.

Point 3.\ follows directly from 1.\ and 2.
\end{proof}

\begin{lemma}\label{lem:bb}
  Let $X$ be a Riemann surface.
  Consider an analytic map $g:\dom(g)\subset X \to X$ having a parabolic fixed point tangent to the identity and let $B^*$ be the immediate basin of one of its petals. Then $g(\partial B^*)$ cannot intersect $B^*$.
\end{lemma}
\begin{proof}
    Otherwise if $z_0\in \dom (f)\cap \partial B^*$ and $g(z_0)\in B^*$ then for some connected open neighbourhood $V$ of $z_0$, $g(V)\subset B^*$. Then $B^*\cup V$ is a connected subset of the basin containing the petal, so is contained in $B^*$, so $z_0 \in B^*$. But $B^*$ is open and $z_0\in \partial B^*$ cannot be in $B^*$. This is a contradiction.
\end{proof}

\begin{corollary}\label{cor:svbas}
  Under the conditions of Lemma~\ref{lem:bb},
  the set of asymptotic values of the restriction $g$ to $B^*$ in the domain and range, is a subset of $B^*\cap av(g)$.
\end{corollary}
\begin{proof}
  By the third point of Proposition~\ref{prop:av},
  $av(g|_{B^*}^{B^*})\subset B^*\cap(av(g) \cup g(\partial B^*)) = (B^*\cap av(g)) \cup (B^*\cap g(\partial B^*))$. By Lemma~\ref{lem:bb}, $B^*\cap g(\partial B^*)=\emptyset$.
\end{proof}

We will have to deal with branched covering of infinite degree, notion which  is not standardized. We use here the notion of \emph{branched coverings} given by Shishikura in~\cite{Shi1}, Section~4.5, which is adapted to our context and we give a characterization.

\begin{definition}\label{def:covereq}[Equivalence and Cover-equivalence]
    (In this definition all maps are assumed holomorphic.)
    We say that two maps $f:X\to Y$ and $g:X'\to Y'$ are \emph{equivalent} if there exist two bijections $\phi: X\to X'$ and $\psi:Y\to Y'$ such that $\psi \circ f=g\circ\phi$.
    \[
    \begin{tikzcd}
    X \ar[r, "\phi"] \ar[d, "f"] & X' \ar[d, "g"] \\
    Y \ar[r,"\psi"] & Y'
    \end{tikzcd}
    \]
    We say that two maps $f:X\to Y$ and $g:X'\to Y$ are \emph{cover-equivalent} (even if they are not coverings) if there exists a bijection $\phi: X\to X'$ such that $f=g\circ\phi$.
    \[\begin{tikzcd}[column sep=tiny]
    X \arrow[rr , "\phi"] \arrow[dr, "f"'] & & X' \arrow[dl, "g"]\\
    & Y &
    \end{tikzcd}\]
\end{definition}

\begin{remark}
  If $f$ or $g$ are not surjective, then we may consider them as maps to smaller target sets.
  However one should be aware that, with our definition, the fact that $f$ and $g$ are equivalent depends on the chosen target sets, since we require the targets to be isomorphic.
\end{remark}

The next lemma is classical too, we find convenient to use the following version, which is\footnote{In \cite{CE}, the case $d=0$ is wrong: if $f$ is constant on $W$ then the set $W$ does not have to be isomorphic to $\D$. But this case is not used neither in \cite{CE} nor here.} Lemma~5.3 from \cite{CE}:
\begin{lemma}\label{lem:rcdocv1}
  Let $f:X\to Y$ be holomorphic. Let $V\subset Y$ be an open subset isomorphic to $\D$, $y\in V$ and $\psi:V\to \D$ an isomorphism such that $\psi(y)=0$.
  Assume that $V\cap sv(f) \subset \{y\}$.
  Let $U$ be a connected component of $f^{-1}(V)$. Then either $f$ is constant on $U$ or the restriction $\psi\circ f : U\to \D$ is cover-equivalent to one of the following models:
  \begin{itemize}
      \item $\rho_k : \D\to \D$, $z\mapsto z^k$, $k\geq 1$,
      \item $\rho_k^* : \D^*\to \D$, $z\mapsto z^k$, $k\geq 1$,
      \item $\rho_\infty : \H_{-}\to \D$, $z\mapsto \exp(z)$ where $\H_-$ is the left half-plane.
  \end{itemize}
\end{lemma}

Under these assumptions, if $U$ contains a critical point then $f$ is non-constant on $U$ and the model is some $\rho_k$ with $k\geq 2$.

We complete this lemma by the following known fact:
\begin{complement}\label{cpl:av}
  In the conditions of the previous lemma, if $f$ is non-constant on $U$ and the model is $\rho_k^*$ or $\rho_\infty$, then $y$ is an asymptotic value of $f$.
\end{complement}
\begin{proof}
  First note that since for $\rho_k^*$ and $\rho_\infty$ there is no preimage of $0$, the set $U$ contain no preimage of $y$.
  By Lemma~\ref{lem:rcdocv1}, $\psi\circ f =\rho \circ \phi$ for some isomorphism $\phi: U\to \dom(\rho) = \D^*$ or $\H_-$.

  If $\rho=\rho_k^*$, consider the path $t\in[0,1)\mapsto \gamma(t) = \phi^{-1}(1-t)$.
  If $\rho=\rho_\infty$, consider the path $\gamma(t) = \phi^{-1}(-t) : [1,+\infty) \to X$.
  In both cases, $f\circ\gamma(t)\to y$ as $t\to 1$, respectively $t\to \infty$. 
  The path $\gamma$ leaves every compact of $X$ for otherwise any accumulation point $x$ as $t\to 1$ would be a preimage of $y$, and $U\cup\{x\}$ would be connected but $U$ is a connected component of $f^{-1}(V)$ so $x\in U$, so in the model there is a preimage of $0$, which is a contradiction.
  Since $\gamma$ leaves every compact of $X$ and $f\circ\gamma$ tends to $y$, the point $y$ is an asymptotic value.
\end{proof}

\begin{definition}\label{def:brc}
  A map $f:X\to Y$ is called a \emph{branched covering} if every point $y\in Y$ has a neighbourhood $V$ isomorphic to $\D$ such that for every connected component $U$ of $f^{-1}(V)$, the restriction $f: U\to V$ is equivalent to $\rho_d:z\mapsto z^d : \D\to\D$ for some $d\geq 1$ (that is allowed to depend on $U$) via isomorphisms $\phi:U\to\D$ and $\psi: V\to \D$ such that $\psi(y)=0$ (in the case $d=1$ this means $f$ is an isomorphism from $U$ to $V$).
\end{definition}

\begin{remark}
  This definition is not formulated exactly as in \cite{Shi1} but it is equivalent.
\end{remark}

A holomorphic function is called \emph{nowhere constant} if there is no point near which it is constant. For instance, a branched covering is nowhere constant.

\begin{corollary}\label{cor:ramcov}
  A holomorphic map $f:X\to Y$ is a branched covering if and only if $f$ is nowhere constant, has no asymptotic value and its set of critical values has no accumulation point in $Y$.
\end{corollary}
\begin{proof}
  The direct implication is easy, yet we detail a proof here.
  Assume $f$ is a branched covering and consider any $y\in Y$ and a neighbourhood $V$ of $y$ as in Definition~\ref{def:brc}.
  From the definition we get that there is no critical value of $f$ in $V-\{y\}$ so the set of critical values has no accumulation point.
  Consider a path $\gamma:[0,1)\to X$ such that $f\circ\gamma(t)$ tends to $y$ as $t\to 1$.
  Then $\exists t_0<1$ such that $t>t_0$ $\implies$ $f\circ\gamma(t)\in V$.
  For such $t$, $\gamma(t)\in f^{-1}(V)$ and by connectedness, $\gamma((t_0,1)) \subset U$ for some connected component $U$ of $f^{-1}(V)$.
  The equivalence to $z^d$ then implies that $\gamma$ stays in a compact subset of $U$ for $t$ big enough.
  In particular, $y$ is not an asymptotic value.
  
  The converse implication is more elaborate.
  Assume $f$ is nowhere constant, has no asymptotic value and that its set of critical value has no accumulation point.
  For any $y\in Y$, there thus a neighbourhood $V$ of $y$ isomorphic to $\D$ such that $V-\{y\}$ contains no critical value.
  Let $\psi:V\to\D$ be an isomorphism sending $y$ to $0$.
  By Theorem~\ref{thm:sv}, $f$ has no singular value in $V-\{y\}$.
  Let $U$ be a connected component of $f^{-1}(V)$.
  By Lemma~\ref{lem:rcdocv1}, $\psi\circ f:U\to D$ is cover-equivalent to one of the models $\rho=\rho_k$, $\rho_k^*$ or $\rho_\infty$, i.e.\ $\psi\circ f|_U = \rho\circ\phi$ for some isomorphism $\phi:U\to \D$.
  If the model were $\rho_k^*$ or $\rho_\infty$, by Complement~\ref{cpl:av}, $y$ would be an asymptotic value of $f$, contradicting our assumptions.  
  So $\rho=\rho_k$.
  Finally, $k\neq 0$, otherwise $f$ would constant on an open set.
\end{proof}

The conclusion of the corollary above can serve as an equivalent definition of branched covering.

Note that the nowhere constant assumption is there to rule out the case where $X$ would have a compact connected component on which $f$ is constant. If $X$ has a non-compact connected component on which $f$ is constant, then $f$ has an asymptotic value.

The \emph{degree} of a branched covering $f:X\to Y$ over a \emph{connected} Riemann surface $Y$ is the cardinality of the fibre of any $y\in Y-cv(f)$, which is independent of $y$.

\begin{lemma}\label{lem:proper-is-bc}
  A nowhere constant proper holomorphic map $f:X\to Y$ is a branched covering
  and it has finite degree over every connected component of $Y$.  
\end{lemma}
\begin{proof}
  We use the criterion of Corollary~\ref{cor:ramcov}.
  There cannot be an asymptotic value otherwise a compact neighbourhood of it would have a preimage that is non-compact as it contains a path leaving every compact subset of $X$, leading to a contradiction.
  Any accumulation point of the set of critical values would have a compact neighbourhood whose preimage, compact by properness, contains infinitely many critical points, so $f$ is constant on some connected component of $X$, a contradiction.
  
  If $Y$ is connected then the number of preimages of any point is finite since this preimage is compact and without accumulation point, otherwise $f$ would be constant on a connected component of $f$.
\end{proof}

\begin{lemma}\label{lem:noavcomp}
  Consider two nowhere constant holomorphic maps $f:X\to Y$ and $g:Y\to Z$.
  Assume $f$ and $g$ have no asymptotic values.
  Then $g\circ f$ has no asymptotic values.
\end{lemma}
\begin{proof}
  Assume by way of contradiction that $\gamma:[0,1)\to X$ is a path leaving every compact subset of $X$ and such that $g\circ f\circ \gamma$ converges to some $z\in Z$.
  The path $f\circ \gamma$ cannot leave every compact of $Y$, otherwise $z$ would be an asymptotic value of $g$.
  There is thus an accumulation point $y\in Y$ of $f\circ \gamma(t)$ as $t\tends 1$.
  But since $g$ is nowhere constant and holomorphic at $y$, $g\circ f\circ \gamma$ can only converge if $f\circ \gamma$ converges to $y$.
  So $y$ is an asymptotic value of $f$, contradicting the assumptions.
\end{proof}

\begin{lemma}\label{lem:union}
  Let $f:X\to Y$ be holomorphic. Assume there exists two \emph{increasing} sequences of open subsets $U_n\subset X$ and $V_n\subset Y$ such that $X=\bigcup_n U_n$, $Y=\bigcup_n V_n$, $f(U_n)\subset V_n$ and the restriction $f:U_n\to V_n$ has no asymptotic value in $V_n$.
  Then $f$ has no asymptotic value.
\end{lemma}
\begin{proof}
  By way of contradiction, assume a path $\gamma:[0,1)\to X$ leaves every compact subset of $X$ and $f\circ\gamma(t)$ converges to $y\in Y$ as $t\tends 1$.
  For every $n$ let $t_n = \max \{t\in [0,1]$ such that $\forall s\in[0,t)$, $\gamma(s)\in U_n\}$. Then $t_n$ is a non-decreasing sequence tending to $1$.
  Note that $\gamma|_{[0,t_n)}$ leaves every compact subset of $U_n$, so its image by $f$ cannot converge in $V_n$ by hypothesis.
  If $t_n=1$ this means $y\notin V_n$.
  If $t_n<1$ this means $f\circ\gamma(t_n)\notin V_n$.
  
  Since $Y=\bigcup V_n$, let $n_0$ be such that $y\in V_{n_0}$.
  If $t_{n_1}=1$ for some $n_1$ then let $n_2=\max(n_1,n_0)$: $t_{n_2}=1$ so $y\notin V_{n_2}$, so $y\notin V_{n_0}$, leading to a contradiction.
  If $t_n<1$ for every $n$, then $y_n:=f(\gamma(t_n))\notin V_n$ hence for $n\geq n_0$, $y_n\notin V_{n_0}$ but $y_n\tends y\in V_{n_0}$ and $V_{n_0}$ is open, which is a contradiction.
\end{proof}

\noindent\textit{About branched covering from $\D$ to itself.} 

\smallskip

Let $f:\D\to \D$ be a branched covering:
\begin{enumerate}
  \item\label{item:bc1} $f$ has finite degree $\iff$ it is proper $\iff$ it is a finite Blaschke product $\iff$ it has finitely many critical points,
  \item\label{item:bc2} if $f$ has one or no critical values, then $f$ has finite degree,
  \item\label{item:bc3} if $f$ has two or more critical values, then $f$ is not necessarily of finite degree.
\end{enumerate}

Point~\eqref{item:bc1} is classical, we just indicate hints of proofs.
If $f$ has finite degree then it is proper: every point $z$ in the range has a neighbourhood that is well-ramified-covered by a finite number of topological disks, and a smaller compact neighbourhood $z$ has thus a compact preimage.
If it is proper then it is a Blaschke product: by the Schwarz reflection principle it extends to a rational map preserving the unit circle, factoring-out the zeroes by Möbius maps preserving the unit disk gives a rational map without zeroes, hence constant.
If it is a Blaschke product, then it has finitely many critical point: as the restriction of a rational map.
If it has finitely many critical point, then it has finite degree: introduce disjoint slits from the critical values to $\partial \D$; their complement $U\subset \D$ has a preimage by $f$, on each connected component of which $f$ is an isomorphism to $\D$;
if a component $V$ has no critical point on its boundary, $f$ is actually an isomorphism from this component to $\D$ hence $V=\D$ and $f$ has degree $1$; otherwise there can be only finitely many components, hence $f$ has finite degree.

Point~\eqref{item:bc2} is classical too and follows from Lemma~\ref{lem:rcdocv1}.

For an example of infinite degree $f$ as in Point~\eqref{item:bc3}, consider the now called Joukowsky map $J: z\mapsto z+1/z$ that sends the annulus $A$ between $1/r$ and $r$ (for any $r>1$) to the domain $U$ bounded by an ellipse $E$ as a branched covering with two critical values $2$ and $-2$.
Consider the conformal mapping $\phi_2:U\to \D$ and a universal covering $\phi_1: \D\to A$. Then the composition $\phi_2 \circ J \circ \phi_1 :\D\to\D$ has infinitely many critical points but only two critical values, and is a branched covering of infinite degree.

In the proofs of the main theorem (Theorem~\ref{thm:main:grl:2}),
and its special case $R(\DDD_2)\subset\DDD_2$ (Theorem~\ref{thm:SLY}), we get at some point that
the restriction $Rf:B^*_{Rf}\to B^*_{Rf}$ is a branched covering.
The classical argument for the special case of $\DDD_2$ is that we know that the restriction has exactly one critical value, so by Point~\eqref{item:bc2} above, it is conjugate to a Blaschke product.
In the general case, we cannot anymore use such an argument.
Let $V_n$ denotes the virtual basins, defined in Section~\ref{sub:grl}.
We will instead prove that $f:V_n\to V_{n+1}$ is proper and that $Rf$ is equivalent to $f^{s-r}:V_r \to V_s$ for some $r<s\in\Z$.

\subsection{Finite type maps}\label{sub:ftm}

This section is used in our second proof of Proposition~\ref{prop:rfft}.

The notion of finite type maps was introduced by Adam Epstein in \cite{AdamThesis}.
\begin{definition}\label{def:ftm}
  A \emph{finite type map} is an analytic map $f : \dom f \subset \cal S \to \cal S'$ where $\cal S$ and $\cal S'$ are two Riemann surfaces, $\cal S'$ is compact, $f$ is open,\footnote{I.e.\ $f$ is nowhere constant.} $f$ has no removable isolated singularities, and the set of singular values of $f$ is finite.
\end{definition}

\begin{lemma}\label{lem:rftm}
Let $f$ be a finite type map as above and $U$ a connected component of $\dom f$, $f|_U : U\to \cal S'$ is also a finite type map.
\end{lemma}
\begin{proof}
Obviously the restriction is still nowhere constant, has no isolated singularities, and $cv(f|_U) \subset cv(f)$.
By item~\ref{item:prop:av:2} of Proposition~\ref{prop:av}, $av(f|_U) \subset av$ since $f (\partial U)=\emptyset$.
From Theorem~\ref{thm:sv}, $sv(f|_U) = \ov{cv(f|_U)\cup av(f|_U)} \subset \ov{cv(f)\cup av(f)} = sv(f)$, so is finite.
\end{proof}

We detail now a construction that closely resembles the construction of the parabolic renormalization, but which we choose not to call this way.
The reason is that parabolic enrichment, in the presence of more than one petal, is more complicated and we do not want to deal with these complications in the present article.

\medskip

Construction: Consider a parabolic fixed point $p$ of a holomorphic map $f$, with any number of attracting axis.
Let $A$ be an attracting axis of $p$, and $R$ be an adjacent repelling axis.
Let $\phi_\att: B(A)\to \C$ be an extended attracting Fatou coordinate.
Let $\psi_\rep : \dom(\psi_\rep)\to \C$ be an extended repelling Fatou coordinate for the repelling axis $R$.
Let $h = \phi_\att\circ\psi_\rep$.
Then $T_1(\dom h) = \dom h$ and $h\circ T_1 = T_1\circ h$.
Denote $\h : \dom (h)/\Z \to \C/\Z$ the quotient map.
Recall the isomorphism $\ov E:\C/\Z\to\C^*$ defined in Section~\ref{sub:ren} and let $g^*:=\ov E \circ \h\circ \ov E^{-1}:\dom g\subset\C^*\to\C^*$.
The classical theory of horn maps implies that $g^*$ is defined in a neighbourhood of $0$ if $R$ precedes $A$ in trigonometric order, and in a neighbourhood of $\infty$ if $R$ follows $A$ (in the case $f$ has only one attracting axis, it is defined in a neighbourhood of both points $0$ and $\infty$), and  that $g^*$ extends to these points into a holomorphic function fixing them, that we denote $g : \dom g \subset \wh\C\to \wh\C$.

\begin{theorem}[\cite{AdamThesis}]\label{thm:Ep}
  If $f$ is a finite type map then $g$ is a finite type map. 
\end{theorem}

In the case of a parabolic point with only one attracting axis, the parabolic renormalization is the restriction of $\lambda g$ to the component of $\dom(g)$ that contains $0$, for some appropriate $\lambda\in \C^*$.
The theorem above plus Lemma~\ref{lem:rftm} imply that, if $f:\dom f\subset \C\to \wh\C$ is a finite type map with a parabolic point with one petal, then $Rf:\dom Rf\subset \C\to \wh\C$ is a finite type map too.

This reasoning also holds for renormalization of Blaschke products as defined in Section~\ref{sub:rcb}:
\begin{corollary}\label{cor:Ep:bla}
  For all Blaschke product $G$ in the class $\cal G$ defined in Definition~\ref{def:Grond}, the map $RG$ is a finite type map.
\end{corollary}

\subsection{Horn map equivalence}

This section is used in our second proof of Proposition~\ref{prop:rfft}.

Particular cases of the theorem below are proved in \cite{LY}, Proposition~3.16.

\begin{theorem}\label{thm:LeMeur}
  Let $f$ and $g$ have parabolic points of multiplier $1$ and consider an attracting axis $A$ of $f$ and $A'$ of $g$.
  Assume that the restrictions $f:B^*(A)\to B^*(A)$ and $g:B^*(A')\to B^*(A')$ are conjugate on their immediate basins.
  Consider the repelling axis of $f$ that follows $A$ in clockwise order and do the same for $g$ and $A'$.
  Let $h_f =\phi_\att \circ\psi_\rep$.
  We recall that it commutes with $T_1$ and we denote $\h_f: \dom(h_f) /\Z\subset\C/\Z \to \C/\Z$ the quotient map.
  Let $\mathbf{h}^+_f$ be the restriction of $\mathbf{h}_f$ to the connected component of its domain that contains an upper half cylinder.
  Then $\mathbf{h}^+_f$ and $\mathbf{h}^+_g$ are cover-equivalent, up to a translation, in the following sense:
  \[T_{\sigma} \circ \mathbf{h}^+_f = \mathbf{h}^+_g \circ \psi\]
  for some $\sigma\in\C$ and some isomorphism $\psi$ between the domains of $\mathbf{h}^+_f$ and $\mathbf{h}^+_g$.
  Moreover $\psi$ has an erasable singularity fixing the top end of the cylinder.
\end{theorem}

This theorem  follows from the direct implication of Complement~1.9 in \cite{CLM} and its generalisation to parabolic points with more than one attracting axis in Theorem~7.5 in the same reference.

\subsection{Path lifting}\label{sub:pathlift}

If $F:\wh\C\to \wh\C$ is a rational map\footnote{This also holds for branched coverings.} then for any $x\in X$ and any path $\gamma:[0,1]\to Y$ with $\gamma(0)=F(x)$, there exists a lift $\wt\gamma:[0,1]\to X$ stemming from $x$, i.e.\ $\wt\gamma$ is continuous, $F\circ \wt\gamma = \gamma$ and $\wt\gamma(0)=x$.
If $\wt\gamma$ goes through a critical point, the lift is \emph{not unique}.
If $\alpha$ is rectifiable then any lift is rectifiable.

Let $G$ be a Blaschke product. It is the restriction to $\D$ of a rational map $F$ for which $F^{-1}(\D)=\D$. It follows that paths can be lifted too for $G$.

Let $f$ have a parabolic point at $0$ with only one attracting axis and such that, on its immediate basin $B^*_f$, $f$ is conjugated to a Blaschke product. Then, similarly as above, paths can be lifted inside $B^*_f$. Let $P_\rep$ be a repelling petal for $f$.
 
\begin{lemma}\label{lem:akiprep}
  Let $f$ be as above.
  Let $\alpha$ be a rectifiable path in $B^*_f$, let $x$ be a point of $\alpha$ and $x_k$ a backward orbit of $x=x_0$ contained in $B^*_f$ converging to the parabolic point  $0$.
  Denote $\alpha_0=\alpha$ and define inductively
  $\alpha_k$ to be a lift of $\alpha_{k-1}$ by $f$ in $B^*_f$ stemming from $x_k$.
  Then there exists $K$ such that for all $k\ge K$, $\alpha_k$ is included in $P_{rep}$.
\end{lemma}
\begin{proof} Let us denote by $PC$ the postcritical set of $f|_{B^*_f}$.
 For all $k$ big enough, say larger than some $k_1$, $\alpha_k$ is disjoint from $PC$.
 Consider the hyperbolic metric $\rho$ on $B^*_{f}-PC$.
 Then the length of $\alpha_k$ for $\rho$ is non-increasing for $k\ge k_1$.

 The Euclidean length of $\alpha_k$ tends to $0$ because it has bounded hyperbolic length and because the sequence $x_{k}\in\alpha_k$ tends to the boundary of $B^*_f-PC$.
 Every point $\alpha(t)$ gives a sequence $\alpha_k(t)$ which is an inverse orbit for $f$ that tends to the parabolic point.
 Such inverse orbits must enter $P_\rep$.
 Then $\alpha_k(t)$ will remain in $P_\rep$ for all nearby $t$ and all bigger $k$.
 By compactness of $\alpha$, this proves the claim.
\end{proof}

Recall that 
\[\Omega = E^{-1}(\dom Rf)\] was defined by Lemma~\ref{lem:prutibsf}, $\psi_\rep$ sends $\Omega$ in $B^*_f$ and $T_1\Omega = \Omega$.
\begin{corollary}\label{cor:psisurj}
  The restriction $\psi_\rep: \Omega \to  B^*_f$ is surjective.
\end{corollary}
\begin{proof}
  For $z\in B^*_f$ let us find in $\Omega$ a preimage of $z$ under $\psi_\rep$.
  Choose any $x\in P_\rep$ such that $u:=\phi_\rep(x)$ has high enough imaginary part so that $u\in \Omega$.
  Let $x_n=\psi_\rep(u-n)$.
  In particular $x_0=x$.
  Let $\alpha$ be a path from $x_0$ to $ z$.
  By Lemma~\ref{lem:akiprep}, denoting $\alpha_k$ for $k\in\N$ a sequence of iterated lifts of $\alpha$ under $f$, starting from $x_k$, then $\alpha_k\subset P_\rep$ for all $k\geq K$ for some $K\in\N$.
  Then $T_K\circ\phi_\rep\circ \alpha_K$ is a path in $\Omega$ from $u$ to some point $u'$.
  We claim that $\psi_\rep(u')=z$.
  Indeed $\alpha_0 = f^K \circ \alpha_K = f^K \circ\psi_\rep\circ\phi_\rep\circ \alpha_K = \psi_\rep\circ T_K\circ\phi_\rep\circ \alpha_K $.
  The first equality follows from the $\alpha_k$ being a lift of $\alpha_{k-1}$ for all $k>1$.
  The second from $\alpha_K\subset P_\rep$.
  For the last one: $T_1(\Omega) =\Omega$, so $u-n\in \Omega$, so $\phi_\rep\circ\alpha_K$, which is contained in $\psi_\rep^{-1}(B^*_f)$, is in fact contained in $\Omega$ (which is a connected component of $\psi_\rep^{-1}(B^*_f)$ by Corollary~\ref{cor:cccuh}), so $T_k\circ\phi_\rep\circ\alpha_K\subset \Omega$ for all $k\geq 0$, and finally $\psi_\rep \circ T_1 = f\circ \psi_\rep$ holds on the set $\Omega$ (see eq.~\eqref{eq:psirepsc}).
\end{proof}

\begin{lemma}\label{lem:blsep}
  Let $G$ be a Blaschke product.
  Let $K\subset\D$ be a compact set, $a\in \D-K$ and assume that $K$ does not separate $a$ from $\partial \D$.
  Then $G^{-1}(K)$ does not separate any point of $G^{-1}(\{a\})$ from $\partial \D$.
\end{lemma}
\begin{proof}
  Recall that $G$ is the restriction of a rational map $F$, with $F^{-1}(\partial \D) = \partial \D$.
  A path from $a$ to a point of $\partial \D$ in $\ov{\D}-K$ lifts into paths in $\D$ from each point in $G^{-1}(\{a\})$ to $\partial \D$.
\end{proof}

We recall that, by the Denjoy-Wolff theorem, holomorphic self-maps of $\D$ come in two types:
\begin{enumerate}
  \item either there is a fixed point in $\D$,
  \item or any compact subset of $\D$ has an orbit that tends uniformly to $\partial \D$.
\end{enumerate}
\begin{corollary}\label{cor:blnsep}
  Let $G$ be a Blaschke product of the second type above.
  Let $K\subset\D$ and $L\subset\D$ be compact.
  Then for all $n$ big enough $G^{-n}(K) \cap L =\emptyset$ and $L$ and $\partial \D$ are in the same connected component of $\ov{\D}-G^{-n}(K)$.
\end{corollary}
\begin{proof}
  Since $G^n(L)$ tends uniformly to $\partial \D$ it is disjoint from $K$ for all $n$ big enough.
  Equivalently, $L$ is disjoint from $G^{-n}(K)$.
  Moreover there is some $m\in\N$ such that $\forall n\geq m$, $G^n(L)$ lies in the component of $\ov\D-K$ that contains $\partial \D$.
  Since $K$ does not separate any point $a\in G^n(L)$ from $\partial \D$, Lemma~\ref{lem:blsep} implies $G^{-n}(K)$ does not separate any point in $L$ from $\partial \D$.
\end{proof}

\subsection{Simply connected domains and holomorphic maps in \e{$\C$}{C}}

The classical results here are a special case of what Douady called the theory of full maps.

\begin{note*} In this section, we call an open subset of $\C$ \emph{simply connected} if all its connected components are simply connected in the usual sense.
\end{note*}

\begin{lemma}\label{lem:topo}
  Let $g:U\to \C$ be a holomorphic map with $U\subset \C$, $D$ be a Jordan domain with $\overline{D}\subset U$ and $V\subset \C$ be a simply connected domain. If $g(\partial D) \subset V$ then $g(D)\subset V$.
\end{lemma}
\begin{proof}
Assume that there exists a point $z'=g(z)\notin V$ for some $z\in D$. Let $\gamma=g(\partial D)$: since $\gamma \subset V$, we have in particular $z'\notin \gamma$.
By the argument principle, the number of preimages of $z'$ by $g$ in $D$, counted with multiplicity, is equal to the winding number of $\gamma$ around $z'$.
Since $z'=g(z)$ there is at least one preimage so this winding number is at least one.
On the other hand $\gamma\subset V$ which is contractible in $V$ so the winding number is equal to $0$ because the contraction takes place in $V$ so avoids $z'$.
We get a contradiction.
\end{proof}

\begin{corollary}\label{cor:ucgd}
  Let $V\subset \C$ be simply connected and $\gamma$ be a simple closed curve bounding a Jordan domain $D$. Then $D\subset V$.
\end{corollary}
\begin{proof}
Apply the previous lemma to $U=\C$ and $g=\id$.
\end{proof}

\begin{corollary}\label{cor:topo}
    If $U$ and $V$ are simply connected open subsets of $\C$ and $g:U\to \C$ holomorphic then every connected component $W$ of $g^{-1}(V)$ is simply connected.
\end{corollary}
\begin{proof} 
We use here that a connected open subset of $\C$ is simply connected if and only if every \emph{simple} closed curve is contractible.
Let $\gamma\subset W$ be a simple closed curve, in particular $g(\gamma)\subset V$ and $\gamma \subset U$.
Let $D\subset \C$ be the Jordan domain bounded by $\gamma$.
Since $U$ simply connected, $D\subset U$ by Corollary~\ref{cor:ucgd}.
By Lemma~\ref{lem:topo}, $g(\overline D)\subset V$ so $\overline D\subset W$.
So $\gamma $ is contractible in $W$. \end{proof}

\begin{corollary}\label{cor:psc}
Let $f: U\to \C$ be  a holomorphic map  with $U\subset \C$ simply connected. 
Then for all $n\in\N$ the domain of $f^n$ is a simply connected open subset of $\C$.
\end{corollary}
\begin{proof}
Use $\dom(f^{n+1}) = f^{-1}(\dom (f^n))$.
\end{proof}

\begin{proposition}\label{prop:scparabo}
    Let $f:U\to \C$ be a holomorphic map with $U\subset \C $ simply connected. Assume that $f$ has a parabolic fixed point $p$ in $U$. 
    Then every connected component $W$ of the parabolic basin is simply connected.
\end{proposition}
\begin{proof} 
We may replace $f$ by an iterate and assume that $f$ is tangent to the identity. This does not change the basin of the parabolic point and now it is the disjoint union of the basins of $P$ for a finite collection of pairwise disjoint petals $P$ such that $f(P)\subset P$. 
The open set $P$ is simply connected, so $f^{-n}(P)$ too by Corollary~\ref{cor:psc}.
Now the basin of $P$ is the increasing union of these simply connected \emph{open} sets so is simply connected: any closed curve in the union is contained in one of them by compactness.
\end{proof}

\section{Proof the main Theorem}\label{sec:pf-inv}

In this section we prove Theorem~\ref{thm:invGrl}.
Let $\tau = d_m\circ \cdots \circ d_1$ and $f\in \DDD_\tau$.
The theorem states that $Rf$ has at least one attracting axis and there exists a disjoint collection of descendants of $\tau$, one for each attracting axis $A$ of $Rf$, that we denote $\tau'[A]$, and such that $Rf$ has type $\tau'[A]$ on the immediate basin of $A$.

\subsection{Outline of the proof.}\label{se:outline}
    
Let $A$ be an attracting axis of the parabolic fixed point $0$ of $Rf$ and denote $B^*_{Rf} =  B^*_{Rf} (A)$.
We will prove (Proposition~\ref{prop:rfrc}) that the restriction $Rf : B^*_{Rf} \to B^*_{Rf}$ is a branched covering.
For this we prove that $Rf:\dom{Rf}\to\wh\C$ is a finite-type map with asymptotic values $0$ and $\infty$ and finitely many critical values (Proposition~\ref{prop:rfft}).
We propose two proofs of this.
One uses an adaptation, with slight differences, of Shishikura's proof that, in the special case of $\DDD_2$, $Rf:\dom{Rf}-\{0\}\to\C^*$ is a branched covering with finitely many critical values.
In the other proof, we consider the Blaschke product $G$ to which $f$ is conjugate on $B^*_f$; the map $RG$ is of finite type in the sense of A.~Epstein by \cite{AdamThesis}; a theorem of Le Meur from \cite{CLM} allows us to transfer this to $Rf:\dom{Rf}\to\wh\C$, which is hence of finite type too.

We then introduce domains $V_n$ defined as follows: first we lift $B^*_{Rf}$ by $E:z\mapsto e^{2\pi iz}$ and obtain a sequence of domains $U_n$ with $U_n = U_0+n$.
Then we let $V_n=\psi_{rep}(U_n)\subset B^*_f$.
The map $f$ sends $V_n$ in $V_{n+1}$ and we will prove that actually $V_n$ is simply connected and $f$ proper from $V_n$ to $V_{n+1}$.

If $m>1$ we will also introduce intermediate sets to decompose $f:V_n\to V_{n+1}$. More precisely, by hypothesis, $f:B^*_f\to B^*_f$ is of type $d_m \circ \cdots \circ d_1$, so conjugate by some map $\phi:B^*_f\to \D$ to a composition of Blaschke products $G = G_m \circ \cdots \circ G_1$ of degree $\on{deg} G_i = d_i$.
We introduce sets indexed by rational numbers, $W_{k/m}$ for $k\in\Z$, such that:
\[\forall n\in\Z, W_{n} = \phi(V_n)\text{ and } \forall k\in\Z, W_{\frac{k+1}{m}}=G_{i+1}(W_{\frac{k}{m}})
\]
where $i$ is the remainder of the Euclidean division of $k$ by $m$.
(This is coherent since $G_m\circ\cdots\circ G_1(\phi(V_n)) = G\circ\phi(V_n) = \phi\circ f(V_n) = \phi(V_{n+1})$, i.e.\ $G(W_n)=W_{n+1}$).
We prove that, like for the $V_n$, the $W_{k/m}$ are simply connected and $G_{i+1}$ is proper from $W_{k/m}$ to $W_{(k+1)/m}$.
So $G_{i+1}:W_{k/m} \to W_{(k+1)/m}$ is equivalent to some (other) Blaschke product $\D\to \D$.

Since $f$ has only finitely many critical points in $B^*_f$, only finitely many $V_p$ will contain a critical point of $f$.
Similarly, only finitely many $W_{k/m}$ contain a critical point of $G_i$ (where $i$ depends on $k$ as above).
For the others, $G_{i+1}:W_{k/m} \to W_{(k+1)/m}$ is a bijection.
In Corollary~\ref{cor:portrait critique} we prove that $Rf : B^*_{Rf}\to B^*_{Rf}$ is equivalent to $f^{q-p}:V_p \to V_q$ for $p\in\Z$ small enough and $q\in \Z$ big enough and in particular they have  the same critical portrait.
Then if $m>1$ we prove in Proposition~\ref{prop:ce2} that $Rf$ is equivalent to a composition of $G_{i+1}:W_{k/m} \to W_{(k+1)/m}$ for a finite sequence of consecutive values of $k\in\Z$.

Among the $G_{i+1}:W_{k/m} \to W_{(k+1)/m}$ that appear in this finite sequence, the ones that are not isomorphisms are those such that $W_{k/m}$ contains one or several critical point of $G_{i+1}$.
Let $m'$ be the number of values of  $k$ for which this happens and denote $k_1<\ldots<k_{m'}$ these values.
Let $d'_j$ be the degree of $G_{i+1}:W_{k_j/m} \to W_{(k_j+1)/m}$, so $d'_j-1$ is the number of critical points of $G_{i+1}$, counted with multiplicity, in $W_{k_j/m}$.
Then $(d'_1-1,\ldots,d'_{m'}-1)$ is a gleaning of $(d_1-1,\ldots, d_{m}-1)$.

Finally, the collection of all attracting axes of $Rf$ give a disjoint collection of gleanings because the corresponding $V_n$ families are pairwise disjoint.

\subsection{Covering properties of \e{$Rf$}{Rf}}

We recall that $f\in\DDD_\tau$ with $\tau = d_m \circ \cdots \circ d_1$.

\begin{proposition}\label{prop:rfft}
  The map $Rf : \dom(Rf)\to\wh\C$ is a finite-type map.
  Its asymptotic values are $0$ and $\infty$.
  Its critical values are the elements of $E(T_\sigma \circ \phi_\att(cv))$ where $\sigma$ and $\phi_\att$ were defined in  Section~\ref{sub:ren} and   $cv$ is the set of critical values of $f:B^*_f\to B^*_f$. 
  As a consequence $Rf$ has finitely many critical values,\footnote{On the other hand, bear in mind that $Rf$ has infinitely many critical points in $\dom Rf$, a fact that we will not use.} and at least one.
\end{proposition}

\begin{proof}[First proof]
  We adapt the proof of \cite{Shi} in Appendix~\ref{app:Mitsu-proof}. 
\end{proof}

We find interesting to provide a second proof based on Epstein's work and a cover-equivalence statement proved in \cite{CLM}.

\begin{proof}[Second proof]
By hypothesis,  $f$ is conjugated to  a Blaschke product  $G$  on the immediate basin $B^*_f$.
It follows that the restriction of $f$ to $B^*_f$ has finitely many critical values.
The critical points of $Rf$ are given by Proposition~\ref{prop:icph}: they are the points $z = E(u)\in\dom(Rf)$ such that $w=\psi_{\rep}(u)$ is defined and pre-critical for $f$.
It follows that the corresponding critical value is an element of $E(T_\sigma \circ \phi_\att(cv))$.
It follows that $Rf$ has finitely many critical values.
Actually every element of the latter set is a critical value of $Rf$ because by Corollary~\ref{cor:psisurj}, every critical point of the restriction $f: B^*_f \to B^*_f$ is the image by $\psi_\rep$ of some point in $\Omega$.

The rest of the proposition follows from Le Meur's work, Theorem~\ref{thm:LeMeur} and Epstein's work, Corollary~\ref{cor:Ep:bla}. We detail this below.

For the reader's convenience, we recall here definitions made earlier, and specify some notations.
Denote $\phi_\att[f]$ the extended attracting Fatou coordinate of $f$ and $\psi_\rep[f]$ the extended repelling Fatou parametrization.
The horn map is defined as
\[h_f = \phi_\att[f]\circ\psi_\rep[f] : \dom h_f \subset\C \to \C.\]
It commutes with $T_1$ and we denote $\h_f: \dom h_f /\Z\subset\C/\Z \to \C/\Z$ the quotient map.

Similarly, denote $\phi_\att[G]$ the extended attracting Fatou coordinate on the component $\D$ of the parabolic basin of $G$ (the corresponding attracting axis points towards the centre of $\D$) and $\psi_\rep[G]$ the extended repelling Fatou parametrization associated to the repelling axis that precedes the attracting axis in trigonometric order.
Let
\[h_G = \phi_\att[G]\circ\psi_\rep[G]\]
and $\h_G$ be the quotient map.

Recall that $\dom Rf = \{0\}\cup U$ where $U$ is the connected component of $\ov E(\dom \h_f)$ which contains a punctured neighbourhood of $0$, that $Rf(0)=0$ and that $\forall z\in U$,
\[Rf(z) = \ov E\circ T_{\sigma_f}\circ \h_f\circ\ov E^{-1}(z)\]
where $\sigma_f$ is the complex number, unique modulo $\Z$, such that $(Rf)'(0) = 1$.
In Section~\ref{sub:rcb} we defined the map $RG:\D\to\C$ by $RG(0)=0$ and $\forall z\in\D^*$,
\[RG(z) = \ov E \circ\, T_{\sigma_G}\circ \h_G\circ \ov E^{-1}(z),\]
$\sigma_G$ is the complex number, unique modulo $\Z$, such that $(RG)'(0) = 1$.
 
We now apply Theorem~\ref{thm:LeMeur}: the conjugacy between $f$ on $B^*_f$ and $G$ on $\D$ implies that there exists a conformal isomorphism $\psi: \on\dom \h_f^+ \to \on\dom \h_g^+$ such that $T_{\sigma} \circ \mathbf{h}^+_f = \mathbf{h}^+_g \circ \psi$, i.e.
\[ \mathbf{h}^+_f \circ \psi^{-1} = T_{-\sigma} \circ \mathbf{h}^+_G \]
where $\h^+$ denotes the restriction of $\h$ to the component of its domain that contains an upper half plane.

Note that the definition of $Rf$ and $RG$ reformulate into
\[R^*f = \ov E\circ T_{\sigma_f}\circ \h^+_f\circ\ov E^{-1}\]
\[R^*G = \ov E\circ T_{\sigma_G}\circ \h^+_G\circ\ov E^{-1}\]
with the star indicating that we omit $0$ from the domain.

Let $\phi = \ov E \circ \psi^{-1} \circ \ov E^{-1}$. This map extends into a holomorphic map fixing $0$, which we still denote $\phi$.
We have $R^*f \circ \phi = \ov E^{-1}\circ T_{\sigma_f} \circ \h_f^+ \circ \psi^{-1}\circ \ov E = \ov E^{-1}\circ T_{\sigma_f-\sigma} \circ \h_G^+ \circ \ov E = \ov E^{-1}\circ T_{\sigma_f-\sigma-\sigma_G} \circ \ov E \circ R^*G$.
Denote $\lambda=\ov E(\sigma_f-\sigma-\sigma_G)$.
Then $\ov E^{-1}\circ T_{\sigma_f-\sigma-\sigma_G} \circ \ov E(z)=\lambda z$.
So $R^*f\circ\phi = \lambda \, R^*G$ and by continuous extension, $Rf \circ \phi = \lambda\, RG$, i.e. the following diagram commutes:
\[
\begin{tikzcd}
  \dom RG \arrow[r , "\phi"] \arrow[d, "RG"']  & \dom Rf\arrow[d, "Rf"] \\
  \wh\C \arrow[r,"\times\lambda"] & \wh\C
\end{tikzcd}
\]

By Corollary~\ref{cor:Ep:bla},
$RG$ is a finite type map over $\wh{\C}$ with asymptotic values $0$ and $\infty$.
By the diagram, $Rf$ is a finite type map over $\wh\C$, with asymptotic values $0$ and $\infty$.
\end{proof}

\begin{remark}\label{rem:atLeastOneAxis}
  Since $Rf$ is defined on a connected set and has at least one critical point, it follows that $Rf$ cannot be the identity near $0$: it thus has attracting and repelling axes.
\end{remark}

Let $A$ be an attracting axis of the parabolic fixed point $0$ of $Rf$ and denote
\[B^*_{Rf} =  B^*_{Rf} (A).\]

\begin{proposition}\label{prop:rfrc}
   The restriction $Rf : B^*_{Rf} \to B^*_{Rf}$ is a branched covering.
\end{proposition}
\begin{proof}
Proposition~\ref{prop:rfft} describes the asymptotic and critical values of $Rf:\dom Rf \to \wh \C$.
By corollary~\ref{cor:svbas}, the set of asymptotic values of the restriction  $Rf:B^*_{Rf}\to B^*_{Rf}$ is contained in the set of asymptotic values of $f$. Since $B^*_{Rf}$ does not contain $0$ nor $\infty$, it follows that $Rf:B^*_{Rf}\to B^*_{Rf}$ has no asymptotic values.
We also saw in Proposition~\ref{prop:rfft} that $Rf:\dom Rf\to \hat\C$ has finitely many critical values, so its set of critical values has no accumulation point in $B^*_{Rf}$. 
By Corollary~\ref{cor:ramcov} it is a branched covering.
\end{proof}

The following adaptation of a theorem of Fatou is already known in more generality (see for instance \cite{AdamThesis} or Theorem~4 in \cite{Arnaud}), still we recall the argument in the specific present case.

\begin{corollary}\label{cor:nbax}
  For every attracting axis $A$ of $Rf$ at $0$, its immediate basin $B^*_{Rf}$ contains a critical value of the restriction of $Rf$ to $B^*_{Rf}$.
  In particular there is at least one critical point of $Rf$ in $B^*_{Rf}$.
\end{corollary}
\begin{proof}
  If not, the branched covering $Rf: B^*_{Rf} \to B^*_{Rf}$ would be a covering.
  The inverse Fatou coordinate $\psi_\att$, initially defined on the half-plane $\Re z>0$ and mapping to an attracting petal contained in $B^*_{Rf}$, could then be extended recursively to $\Re z> -n$ for all $n\in\N$ into an injective holomorphic map.
  We would then get a global extension $\psi_\att :\C \to B^*_{Rf}$ which is non-constant, in contradiction with Liouville's theorem and the fact that $B^*_{Rf}$ omits three points ($B_{Rf}\subset C^*$ but $B_{Rf}\neq \C^*$ by Lemma~\ref{lem:nonc} applied to $Rf$).
\end{proof}

\subsection{The domain \e{$\dom Rf$}{Dom Rf} is simply connected}\label{sub:sc}

Let us recall that we assume $f\in \DDD_\tau$ for some $\tau = d_m\circ \cdots \circ d_1$.

Note that \emph{we do not assume that the domain of $f$ is simply connected nor connected}.
On the other hand, $B^*_f$ is assumed isomorphic to $\D$ so it is simply connected.

We warn the reader that the connected components of $U:=\psi_\rep^{-1}(B^*_f)$ are not necessarily all simply connected.

\begin{proposition}\label{prop:drfsc}
  The domain $\dom Rf$ is simply connected.
\end{proposition}
\begin{proof}
We first prove that $\Omega$ is simply connected.
By definition (see Section~\ref{sub:ren}), $\Omega := E^{-1}(\dom Rf)$, so $T_1(\Omega) = \Omega$.
On the other hand, $\psi_\rep(\Omega)\subset B^*_f$ (Lemma~\ref{lem:prutibsf}).
Note that we cannot directly apply Corollary~\ref{cor:topo} to $\psi_\rep$ because $\dom(\psi_\rep)$ is not necessarily simply connected.
Consider a simple closed curve $\gamma$ in $\Omega$ and denote by $D\subset\C$ the Jordan domain it bounds.
Let $n\in\N$ such that $T_{-n}(\gamma)$ is contained in the image $V$ of a repelling petal by $\phi_\rep$.
The set $V$ is simply connected, $B^*_f$ too and $V\subset \dom \psi_\rep$, so $(\psi_\rep|_V)^{-1}(B^*_f)$ is simply connected by Corollary~\ref{cor:topo}.
It follows that $D-n \subset (\psi_\rep|_V)^{-1}(B^*_f)$ by Corollary~\ref{cor:ucgd}.
In particular $T_{-n}(D)\subset \dom h_f$.
By $T_1$-invariance of $\Omega$, $T_{-n}(\partial D)= T_{-n}(\gamma) \subset \Omega$.
Since $\Omega$ is a component of $\dom h_f$, it follows that $T_{-n}(D) \subset \Omega$.
By $T_1$-invariance $D\subset \Omega$.
Since $\gamma$ is contractible in $D\subset \Omega$, this proves that $\Omega$ is simply connected.

Then we recall that
$\dom  Rf = E(\Omega) \cup \{0\}$.
The restriction $E : \Omega \to E(\Omega)$ is a universal covering and its deck transformations are generated by $T_1$.
So $\pi_1(E(\Omega)) = \Z$.
Now consider a disk $D$ neighbourhood of $0$ such that $D-\{0\}\subset E(\Omega)$ and call $\iota$ the inclusion map between these last two sets.
Then the generator of $\pi_1(D-\{0\})$ maps by $\iota_*$ to the generator or $\pi_1(E(\Omega))$, so by Van Kampen's theorem $\pi_1(E(\Omega) \cup D)$ is reduced to its neutral element.
Now $E(\Omega) \cup D = E(\Omega) \cup \{0\} = \dom Rf$.
\end{proof}

By Proposition~\ref{prop:scparabo}:

\begin{corollary}\label{cor:sc}
    The immediate basin $B^*_{Rf}$ of any attracting axis of $Rf$ at $0$ is simply connected.
\end{corollary}

Knowing that $B^*_{Rf}$ is simply connected (Corollary~\ref{cor:sc}) and a strict subset of $\C$ (Lemma~\ref{lem:nonc}), hence isomorphic to $\D$ and that the restriction $Rf:B^*_{Rf} \to B^*_{Rf}$ is a branched covering (Proposition~\ref{prop:rfrc}) with finitely many critical values may seem a lot, but note that, for instance, this is not enough to imply that $Rf$ has only finitely many critical points, i.e.\ that it is conjugated to a finite Blaschke product, in $B^*_{Rf}$. Unless there is only one critical value in $B^*_{Rf}$ as we will see in the particular situation of Section~\ref{sub:dyn2}.

\subsection{If \e{$f\in \DDD_2$}{f in Dyn2} then \e{$Rf$}{Rf} is in \e{$\DDD_2$}{Dyn2}.}\label{sub:dyn2} 

Here we reprove the result of Shishikura: Theorem~\ref{thm:SLY}.

Under the assumption that $f\in \DDD_2$ the map $Rf$ has exactly one critical value, and by Remark~\ref{rem:atLeastOneAxis} and Corollary~\ref{cor:nbax} there is exactly one attracting axis of $Rf$ at $0$ (so $Rf\notin \cal E$).
Denote $B^*_{Rf}$ the immediate basin of this axis.
Corollary~\ref{cor:nbax} also tells that the range of the restriction $Rf :B^*_{Rf} \to B^*_{Rf}$ contains a critical value.
We saw in Proposition~\ref{prop:rfrc} that the restriction $Rf$ is a branched covering from $B^*_{Rf}$ to $B^*_{Rf}$.
Its thus has no asymptotic value and we saw it has has a unique critical value.
We also saw in Section~\ref{sub:sc}, that the immediate basin $B^*_{Rf}$ is isomorphic to $\D$.

Under these conditions Lemma~\ref{lem:rcdocv1} implies that $Rf:B^*_{Rf}\to B^*_{Rf}$ is equivalent to some $z\mapsto z^d:\D\to\D$ with $d\geq 2$.
Hence there is only one critical point in $B^*_{Rf}$.
In particular, $Rf: B^*_{Rf}\to B^*_{Rf}$ is proper.
By Proposition~\ref{prop:icph}, this critical point has degree $2$: indeed the critical point of $f$ in $B_f$ is attracted to the parabolic point, hence cannot be periodic, so the degree of the composition $f^m$ in Proposition~\ref{prop:icph} can only be $2$.

The conjugate $G$ of $Rf$ by the Riemann mapping $\phi: B^*_{Rf}\to \D$ is hence a holomorphic map from $\D$ to itself equivalent to $z\mapsto z^2$. It is thus equal to a Blaschke product of degree $2$. To prove that  $G$ is conjugated to 
$B_2(z) = \frac{z^2 + \frac13}{1+\frac{1}{3}z^2}$, it is enough to prove that its extension $F$ to $\widehat{\C}$ has a two petal parabolic fixed point on the boundary. 
This is given by Lemma~\ref{lem:bppp}.

\begin{remark} 
Let $d\geq 3$.
It is also known (under some mild supplementary restrictions) that if $f\in\DDD_d$ and if moreover $f$ is unicritical, then 
$Rf\in\DDD_d$ and $Rf$ is unicritical.
The argument above also works in this case, without the restrictions.
\end{remark}

\subsection{General case}\label{sub:grl}

Here we prove the general case of Theorem~\ref{thm:invGrl}.
Let $\tau = d_m\circ \cdots \circ d_1$ and $f\in \DDD_\tau$.
In \ref{rem:atLeastOneAxis}, we proved that $Rf$ has at least one attracting axis.
There remains to prove that there exists a disjoint collection of descendants of $\tau$, one for each attracting axis $A$ of $Rf$, that we denote $\tau'[A]$, and such that $Rf$ has type $\tau'[A]$ on the immediate basin of $A$.

Note that the proof below also gives an alternate proof of the case treated in Section~\ref{sub:dyn2}.

We fix now, and until near the end of the present section, an attracting axis $A$ of $Rf$ and for readability we omit, except in a few places, its mention in the mathematical notation: for instance $B^*_{Rf}$ will refer to $B^*_{Rf}(A)$.

We will study the dynamics of $f$ on a sequence of domains $V_n$ corresponding to $B^*_{Rf}$ (they correspond to the so-called \emph{virtual parabolic basins} of Lavaurs maps, see \cite{DouadyImplosion}; they are called \emph{parabolique au bout} in \cite{lavaurs}); will not prove this fact).
They may be defined either by pulling $E^{-1}(B^*_{Rf})$ by $\phi_\att$ or pushing by $\psi_\rep$. We chose the second approach.

Recall that we denote $E: \C\to \C^*$, $z\mapsto e^{2\pi iz}$. 
We saw (Corollary~\ref{cor:sc}) that $B^*_{Rf}$ is simply connected.
Since it does not contain $0$ nor $\infty$, it has a simply connected lift $U_0\subset\C$ by $E$.
For $n\in\Z$, let
\[U_n = T_n(U_0) = U_0 + n.\]
These sets are pairwise disjoint, for otherwise $B^*_{Rf}$ would have a non-contractible loop but by Corollary~\ref{cor:sc} it is simply connected.
This simple connectivity also implies that $U_n$ is simply connected.
Recall that we denote
\[\Omega = E^{-1} (\dom (Rf))\]
and that $T_1\Omega\subset\Omega$.
Note that
\[U_n\subset \Omega.\]
Recall by Section~\ref{sub:ren} that $Rf$ is   semi-conjugate by $E$ of $T_{\sigma_0}\circ h_f|_{\Omega}$ for some $\sigma_0\in\C$.
Since $Rf(B^*_{Rf}) \subset B^*_{Rf}$,
it follows by connectedness that there is some $n_0\in\Z$ such that $Rf(U_0) \subset U_{n_0}$.
Let
\[ \wt R = T_{\sigma_0-n_0} \circ h_f|_{\Omega} .\]
Then $\wt R(U_n) \subset U_n$ for all $n$.
Note that $\wt R$ is a lift of $Rf$ by $E$, i.e.\ the following diagram commutes:
\[
\begin{tikzcd}
  \Omega \arrow[r , "\wt R"] \arrow[d, "E"']  & \C \arrow[d, "E"] \\
  \dom Rf \arrow[r,"Rf"] & \C^*
\end{tikzcd}
\]

We call \emph{virtual basins} the following open subsets of $\C$:
\[V_n = \psi_\rep(U_n).\]
Recall that $\psi_\rep$ denotes the extended repelling Fatou coordinates and note that $\Omega \subset \dom \psi_\rep$.
We have $V_n \subset \psi_\rep(\Omega)$ so by Lemma~\ref{lem:prutibsf},
\[V_n \subset B^*_f.\]
Moreover
\begin{equation}\label{eq:fvv}
  f(V_n) = V_{n+1}.
\end{equation}
Indeed $f(V_n) = f\circ \psi_\rep(U_n) =^* \psi_\rep\circ T_1(U_n) = \psi_\rep(U_{n+1}) = V_{n+1}$ where the starred equality holds because both $U_n$ and $T_1(U_n)$ are contained in the domain of $\psi_\rep$ where the conjugacy  holds (see Eq.~\ref{eq:psirepsc}).

\begin{lemma}\label{lem:Vndisjoint1}
The sets $V_n$, $n\in\Z$, are pairwise disjoint.
\end{lemma}
\begin{proof}
  Let $\sigma = \sigma_0 - n_0$.
  Recall that $\wt R (U_n) \subset U_n$.
  But $\wt R (U_n) = T_\sigma\circ\phi_{\att}\circ \psi_\rep(U_n) = T_\sigma\circ\phi_{\att}(V_n)$ so if there were $z\in V_n \cap V_m$ the point $T_\sigma\circ\phi_{\att} (z)$ would belong to $\wt R(U_n)$ and to $\wt R(U_m)$, hence to $U_n\cap U_m$. This would contradict the fact that  the $U_k$ are disjoint.
\end{proof}

\begin{lemma}\label{lem:Vndisjoint2}
Given two different attracting axes $A$, $A'$ of $f$, 
$\bigcup_{n\in\Z} V_n[A]$ and $\bigcup_{n\in\Z} V_n[A']$ are disjoint.
\end{lemma}
\begin{proof}
  Indeed, for an axis $A$ of $Rf$, $E\circ T_{\sigma_0} \circ \phi_\att(V_n[A]) = E\circ T_{\sigma_0} \circ \phi_\att \circ \psi_\rep (U_n[A]) = E\circ T_{\sigma_0} \circ h_f (U_n[A]) = Rf\circ E(U_n[A]) = Rf(B^*_{Rf}(A)) \subset B^*_{Rf}(A)$,
  but for two different axes $A$, the immediate basins $B^*_{Rf}(A)$ are disjoint.
\end{proof}

We are about to state two key propositions. Their proofs would be way easier if the $V_n$ were completely contained in a repelling petal for $n$ negative enough. However, we want our study to apply to cases where this will not hold, as for instance happens for $f(z)=ze^z$ (see Section~\ref{app:non-tame}).
For each proposition we are going to repeat similar arguments, which we sum up in the following lemma, whose proof is non-trivial.

\begin{lemma}\label{lem:akiv}
Choose any repelling petal $P_\rep$ of $f$ whose image in repelling Fatou coordinates $\phi_\rep$ is a left half-plane $H$.
Normalize $\phi_\rep$ so that it is an inverse branch of $\psi_\rep$.
Consider any rectifiable path $\alpha$ contained in $V_N$ for some $N\in\Z$, hence in $B^*_f$, and starting from a point $x_0=\psi_\rep(u_0)$ where $u_0\in U_N$.
Let $u_k = T_{-k}(u_0) = u_0-k$ and $x_k = \psi_\rep(u_k)$.
Then $x_k\in B^*_f$, $(x_k)$ is an inverse orbit and $x_k\tends 0$.
Let $\alpha_0=\alpha$ and define inductively\footnote{$\alpha_k$ exists and is contained in $B^*_f$ by the discussion in Section~\ref{sub:pathlift}.} $\alpha_k$ to be a lift of $\alpha_{k-1}$ by $f$ in $B^*_f$ stemming from $x_k$.
So by Lemma~\ref{lem:akiprep}, for all $k$ big enough, $\alpha_k$ is included in $P_\rep$.
Choose $K$ big enough so that $\forall k\geq K$,
$\alpha_k\subset P_\rep$ and $u_k\in H$.
Then
  \begin{enumerate}
    \item[i.]\label{item:u} $\forall k\geq K$, $\phi_\rep(\alpha_k)\subset U_{N-k}$,
    \item[ii.]\label{item:v} $\forall k\in \N$, $\alpha_k\subset V_{N-k}$.
  \end{enumerate}
\end{lemma}
\begin{proof}  
From $u_k\in U_{N-k}\subset \Omega\subset \psi_\rep^{-1}(B^*_{f})$ we get that $x_k\in B^*_{f}$.
The fact that $x_k$ is an inverse orbit follows from $T_1$-invariance of $\Omega$ and the relation $\psi_\rep \circ T_1 = f\circ \psi_\rep$ which holds on a set containing $\Omega$, see eq.~\ref{eq:psirepsc}.
For all $k$ big engouh, $u_k\in H $,  $x_k = \phi_\rep^{-1}(u_k)$, and by classical estimates on Fatou coordinates, $x_k\tends 0$.

For convenience, we denote $\wt \phi_\att = T_{\sigma_0-n_0}\circ \phi_\att$, so that $\wt R = \wt \phi_\att \circ \psi_\rep$. 
From $\alpha\subset V_{N}$ we get $\wt \phi_\att(\alpha) \subset \wt \phi_\att(V_{N}) =\wt \phi_\att(\psi_{rep}(U_{N}))= \wt R(U_{N}) \subset U_{N}$.
So $\wt \phi_\att(\alpha_K) = T_{-K}\circ \wt \phi_\att(\alpha) \subset U_{N-K}$.
Recall $\alpha_K$ is contained in the repelling petal $P_\rep$.
Let $\alpha' = \phi_\rep(\alpha_K)$.
Then $\alpha_K = \psi_\rep(\alpha')$ and $\wt R(\alpha') = \wt\phi_\att\circ\psi_\rep(\alpha') = \wt\phi_\att(\alpha_K) \subset U_{N-K}$.
Recall $E\circ \wt R = Rf \circ E$, so $Rf(E(\alpha')) \subset E(U_{N-K}) = B^*_{Rf}$.
So $E(\alpha') \subset B_{Rf}$.
But $\alpha' = \phi_\rep(\alpha_K)$ contains the point $u_{K}=\phi_\rep(x_{K})$, so $E(\alpha')$ contains the point $E(u_{K})$, which is in $B^*_{Rf}=E(U_{N-K})$.
So actually $E(\alpha')\subset B^*_{Rf}$.
So $\alpha'\in E^{-1}(B^*_{Rf})$, which is the disjoint union of the connected open sets $U_n$, $n\in\Z$.
Since $\alpha'$ is connected and contains $u_{K}$, we get $\alpha'\subset U_{N-K}$ (hence point~\textit{i} holds), and hence
$\alpha_K =\psi_\rep(\alpha') \subset \psi_\rep(U_{N-K}) = V_{N-K}$.
Then for all $0\leq k\leq K$ we have $\alpha_k = f^{K-k}(\alpha_K) \subset f^{K-k}(V_K)=V_k$ by eq.~\ref{eq:fvv}.
Since the result holds for arbitrarily high $K$, we get point~\textit{ii}.
\end{proof}

\begin{proposition}\label{prop:simpleconnexe}
  The sets $V_n$, $n\in\Z$, are simply connected.
\end{proposition}
\begin{proof}
Consider a closed loop $\gamma$ in $V_{n_1}$ and let us prove that it is contractible in $V_{n_1}$.
Let us denote by $PC$ the postcritical set of $f|_{B^*_{f}}$.
Since $PC$ is discrete, we can always deform $\gamma$ so that it is $C^1$ (hence rectifiable) and avoids $PC$.

Consider any point $w\in B^*_{Rf}$. The set $E^{-1}(\{w\})$ consists in the points $u_n = u_0+n\in U_n$, $n\in\Z$.
Let $x_n = \psi_\rep(u_n) \in V_n$.
Consider a simple path $\alpha$ in $V_{n_1}$ from $x_{n_1}$ going to any point $y_{n_1}$ on $\gamma$.
By Lemma~\ref{lem:akiv}, we can lift the curve $\alpha$ into a path sequence $\alpha_k$ of successive pull backs by $f$, stemming from $x_{n_1-k}$ and ending at some point that we denote $y_{n_1-k}$ and $\exists K_0\in\N$ such that for all $k\geq K_0$, $\alpha_k$ is included in $P_\rep$. 
By Lemma~\ref{lem:akiv} \textit{ii}, $\alpha_k \subset V_{n_1-k}$ for all $k\in\N$. 
By Lemma~\ref{lem:akiv} \textit{i}, for all $k\geq K_0$, $\phi_\rep(\alpha_k) \subset U_{n_1-k}$.
Recall that $y_{n_1-k}$ is an endpoint of $\alpha_k$, so belongs to the sets above.
For $k\geq K_0$, let $v_{n_1-k} := \phi_\rep(y_{n_1-k})$, so that $y_{n_1-k}= \psi_\rep(v_{n_1-k})$.

We want to iteratively pull back the closed curve $\gamma$ into closed curves passing through $y_{n_1-k}$.
A priori the pull-back may not close ($\gamma$ may surround post-critical points) but, given $k$, if instead we walk along $\gamma$ a finite number of times $n(k)$ depending on $k$, the pull-back will close on itself.
This holds because by hypothesis $f|_{B^*_{f}}$ is conjugated to a finite degree branched covering of $\D$.
We call $\gamma_k$ this pull-back for $n(k)$ minimal.

We claim that $n(k)$ stays bounded.
Indeed the curve $\gamma$ separates only finitely many points of PC from $\infty$: since every critical point of a parabolic Blaschke product is attracted to $\partial \D$, $\exists m\in\N$ such that if $c$ is a critical point of $f|_{B^*_f}$ and $f^k(c)$ is separated from $\infty$ by $\gamma$, then $k\leq m$.
When $k>m$ then by Corollary~\ref{cor:blnsep} applied to the set $L=cv$, where $cv$ denotes the set of critical values of $f|_{B^*_f}$, and to $K=\gamma$, the curve $\gamma_k$, which is contained in $f^{-k}(K)$, cannot separate any critical value of $f|_{B^*_f}$ from $\infty$, and then $n(k+1)=n(k)$, i.e.\ further pull-backs of $\gamma_k$ do not require following it several times (indeed $\gamma_k$ is contractible in the complement of $cv$ and this contraction can be pulled-back which implies that the endpoint of a simple pull-back of $\gamma_k$ ends at its starting point).
This proves the claim.

We can then apply Lemma~\ref{lem:akiv} \textit{i} again and get that $\exists K_1\in\N$ such that $\forall k\geq K_1$, $\gamma_k$ is contained in $P_\rep$ and $\phi_\rep(\gamma_k)\subset U_{n_1-k}$.
Since $U_{n_1-K_1}$ is simply connected, there exists an isotopy of $\phi_\rep(\gamma_{K_1})$ to a point $p$.
This isotopy can be composed by $\psi_{\rep}$
and gives an isotopy in $V_{n_1-K_1}$ of $\gamma_{K_1}$ to a point.
Then we compose by $f^{K_1}$ and get an isotopy in $V_{n_1}$ from $\gamma$ to a point.
\end{proof}

\begin{proposition}\label{prop:proper}
  For all $N\in\Z$, the restriction $f: V_N \to V_{N+1}$ is proper.
\end{proposition}
\begin{proof}
  Let $K\subset V_{N+1}$ be compact. Recall that $V_{N+1}\subset B^*_f$.
  Since $f: B^*_f\to B^*_f$ is conjugated to a Blaschke product, it is proper so $K' := f^{-1}(K) \cap B^*_f$ is compact.
  Proving that $f: V_N \to V_{N+1}$ is proper means proving that  $f^{-1}(K) \cap V_N$ is compact.
  If it were not the case, there would be a point $w\in B^*_f$ that belongs to $K'$ and to $\partial V_N$.
  Since  $f(w)\in V_{N+1}$, there  exists $r>0$ such that $B:=B(f(w),r)\subset V_{N+1}$.
  The set $f^{-1}(B)$ contains a ball $B'$ centered on $w$.
  Choose any $w^*\in B' \cap V_N$.
  By definition of $V_N$ there is some $z^* \in U_N$ such that $\psi_\rep(z^*) = w^*$.
  We choose any rectifiable path $\gamma:[0,1]\to B'$ from $w^*$ to $w$, for instance a line segment.
  Let $w^*_{N-n} = \psi_\rep(z^*-n)$ for all $n\in\Z$.
  Since $f:B^*_f\to B^*_f$ is a finite degree branched covering, the path $\gamma$ has  a sequence of iterated lifts $\gamma_n$ starting from $w^*_{N-n}$, i.e.\ such that $\gamma_n = f\circ \gamma_{n+1}$ and $\gamma_n (0)=w^*_{N-n}$.
  We now apply Lemma~\ref{lem:akiv} \textit{i} and deduce that $\gamma_n\subset P_\rep$ for all $n$ big enough and that $\phi_\rep(\gamma_n)\subset U_{N-n}$.
  Then $\gamma_n(1) \in \psi_\rep(U_{N-n})=V_{N-n}$.
  Using that $f^n\circ\gamma_n = \gamma_0$ and that $f(V_k)= V_{k+1}$, we get $w = \gamma_0(1) = f^n(\gamma_n(1)) \in V_N$, which is a contradiction.
\end{proof}

Let
\[V_\Z=\bigcup_{n\in\Z} V_n.\]

\begin{remark}\label{rem:inj1}
  For any $n$ such that $V_n$ contains no critical point, the map $f:V_{n}\to V_{n+1}$ is an isomorphism. Indeed,  we proved that the $V_n$ are simply connected and that $f:V_{n}\to V_{n+1}$ is proper.
\end{remark} 

\begin{lemma}\label{lem:yadespc}
  There is at least one critical point of $f$ in $V_\Z$.
\end{lemma}
\begin{proof}
By Corollary~\ref{cor:nbax}, there is at least one critical point $c$ of $Rf$ in $B^*_{Rf}$.
Let $u_0$ be the unique preimage of $c$ by $E$ in $U_0$.
It is a critical point of $\wt R$, hence of $h_f$, i.e.\ there exists $n\in\N$ such that $\psi_\rep(u_0-n)$ eventually hits a critical point of $f$ under forward iteration of $f$.
\end{proof}

Let $m_c$ and $n_c$ be the smallest and biggest $n\in\Z$ such that $V_n$ contains a critical point of $f$. We have $-\infty<m_c\leq n_c<\infty$.
\begin{lemma}\label{lem:inj2}
  For all $n\in\Z$ with $n\le m_c$, $\psi_\rep :U_n \to V_n$ is injective.
  For all $n\in\Z$ with $n>n_c$, $\phi_\att :V_n \to U_n$ is injective.
\end{lemma}
\begin{proof}
  For the first claim, let $u\neq u'$ in $U_n$. 
  There is $N\in\N$ such that $u-N$ and $u'-N$ belong to a half-plane where $\psi_\rep$ is injective.
  Then $\psi_\rep(u-N)\neq\psi_\rep(u'-N)$ and both belong to $V_{n-N}$.
  Moreover $\psi_\rep(u) = f^N(\psi_\rep(u-N))$ and $\psi_\rep(u') = f^N(\psi_\rep(u'-N))$ and
  $f^N : V_{n-N} \to V_n$ is injective by Remark~\ref{rem:inj1}.
  
  For the second claim, let $v\neq v'$ in $V_n$.
  There is $N\in\N$ such that $f^N(v)$ and $f^N(v')$ belong to an attracting petal where $\phi_\att$ is injective.
  The map $f^N$ is injective on $V_n$ by Remark~\ref{rem:inj1}.
  We thus have $\phi_\att(v) = \phi_\att(f^N(v))-N \neq \phi_\att(f^N(v'))-N =\phi_\att(v')$.
\end{proof}

The critical portrait of a map $h:U\to V$ is the combinatorial data giving the number of critical points, the multiplicities and which critical points have the same image. Below we use the notion of cover-equivalence from Definition~\ref{def:covereq}.

\begin{corollary}\label{cor:portrait critique}
    The map $Rf:B^*_{Rf}\to B^*_{Rf} $ is equivalent to  $f^{q-p}: V_p\to V_{q}$ for any $p,q\in\Z$ such that $p\le m_c$ and $q>n_c$. In particular they have the same critical portrait.
\end{corollary}
\begin{proof}
In the following \emph{non-commuting} diagram, $Rf$ is the composition of all arrows, starting from $B^*_{Rf}$.
Every arrow which is not an $f$ is an isomorphism between the indicated sets.
\[
\begin{tikzcd}
   U_p \arrow[d, "\psi_\rep"']  &  B^*_{Rf} \arrow[l , "E^{-1}"'] &  \arrow[l, "E" '] U_{q} \\
    V_p \arrow[r,"f"] &  \cdots \arrow[r , "f"] & V_{q} \arrow[u,"\phi_\att"']
\end{tikzcd}
\]
\end{proof}

\begin{remark}\label{rem:bp}
Up to now we only knew that $Rf:B^*_{Rf}\to B^*_{Rf} $ was a branched covering with at least one critical point on the hyperbolic simply connected open set $B^*_{Rf}$. Thanks to the statement above, we now know it is conjugated to a Blaschke product: indeed, $f:V_n\to V_{n+1}$ has finite degree.
\end{remark}

By hypothesis, $f:B^*_f\to B^*_f$ is of type $d_m \circ \cdots \circ d_1$, so conjugate to a composition of Blaschke products $G = G_m \circ \cdots \circ G_1$ of degree $\on{deg} G_i = d_i$. 
Let $\phi:B^*_f\to \D$ be a conjugacy. 
We introduce sets indexed by rational numbers, $W_{k/m}$ for $k\in\Z$ as follows.
First, we let
\[\forall n\in\Z, W_{n} = \phi(V_n).\]
Then for all $i\in\{1,m-1\}$ we inductively define
\[W_{n+\frac{i}{m}} = G_{i}(W_{n+\frac{i-1}{m}}).\]
Then actually
\[\forall k\in\Z, W_{\frac{k+1}{m}}=G_{i+1}(W_{\frac{k}{m}})
\]
where $i$ is the remainder of the Euclidean division of $k$ by $m$, since $G_m\circ\cdots\circ G_1(W_n) = G_m\circ\cdots\circ G_1(\phi(V_n)) = G\circ\phi(V_n) = \phi\circ f(V_n) = \phi(V_{n+1}) = W_{n+1}$.
See Figure~\ref{fig:W}.

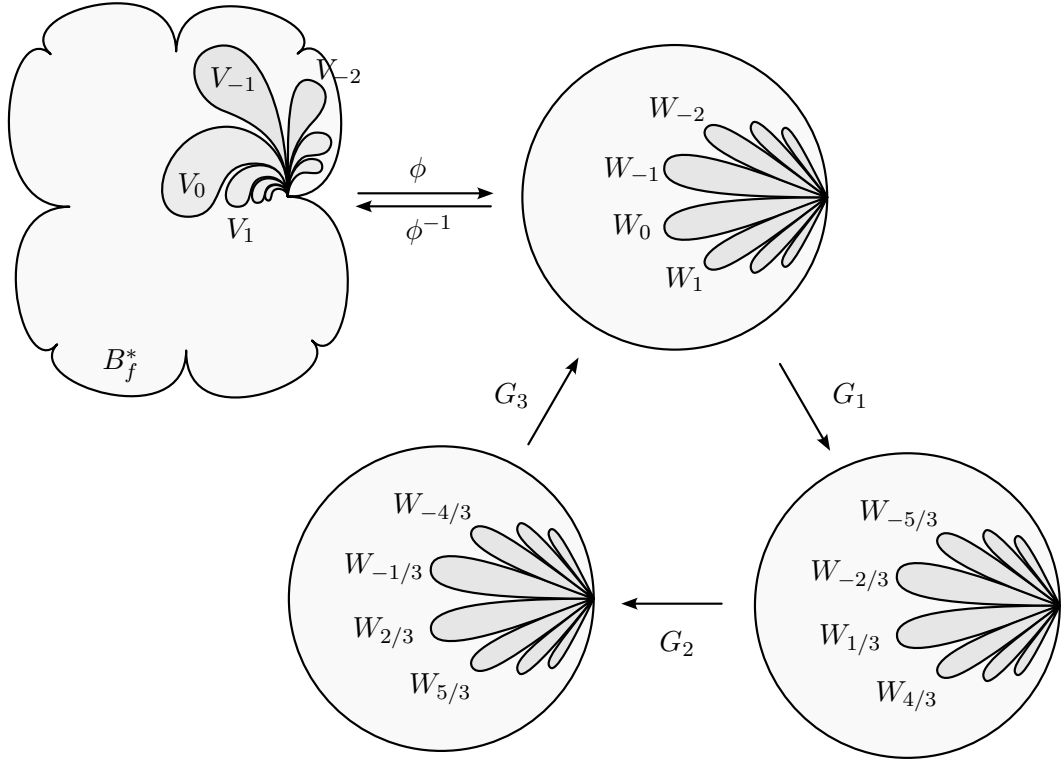
\begin{figure}[htbp]
\small
\centering
\begin{tikzpicture}
\node at (0,0) {\includegraphics[width=14cm]{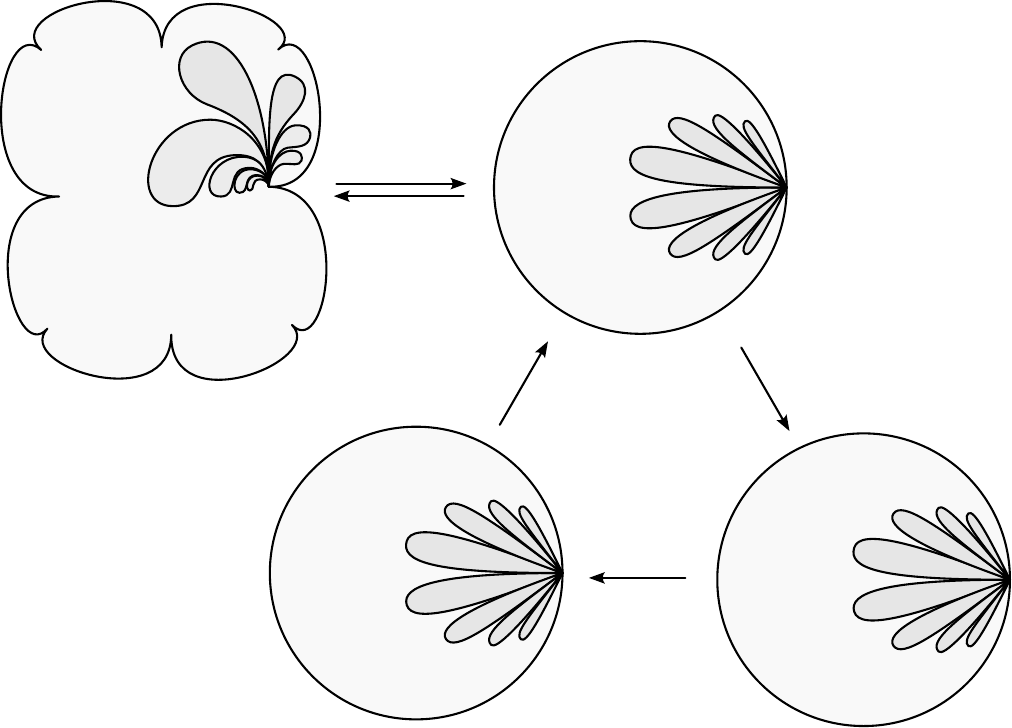}};

\node at (-1.55,2.8) {$\phi$};
\node at (-1.4,2) {$\phi^{-1}$};

\node at (4.2,-0.2) {$G_1$};
\node at (1.9,-3.5) {$G_2$};
\node at (-0.3,-0.2) {$G_3$};

\node at (-5.5,0.2) {$B^*_f$};

\node at (-2.6,4.1) {$V_{-2}$};
\node at (-4,4) {$V_{-1}$};
\node at (-4.55,2.6) {$V_0$};
\node at (-3.9,2) {$V_1$};

\node at (1.9,3.6) {$W_{-2}$};
\node at (1.3,2.8) {$W_{-1}$};
\node at (1.3,2.05) {$W_0$};
\node at (2,1.35) {$W_1$};

\begin{scope}[xshift=2.9cm,yshift=-5.5cm]
\node at (1.95,3.7) {$W_{-5/3}$};
\node at (1.3,2.9) {$W_{-2/3}$};
\node at (1.3,2.05) {$W_{1/3}$};
\node at (2.05,1.3) {$W_{4/3}$};
\end{scope}

\begin{scope}[xshift=-3.3cm,yshift=-5.4cm]
\node at (1.95,3.7) {$W_{-4/3}$};
\node at (1.3,2.9) {$W_{-1/3}$};
\node at (1.3,2.05) {$W_{2/3}$};
\node at (2.05,1.3) {$W_{5/3}$};
\end{scope}
\end{tikzpicture}
\caption{A simplified depiction of the sets $V_n$, and $W_{\frac{k}{m}}$. In reality some of the dark grey sets touch the boundary of the light grey sets in more than one point.}
\label{fig:W}
\end{figure}

We claim that the nice properties of the sequence $V_n$ (simple connectivity) and the restrictions $f:V_{n}\to V_{n+1}$ (properness) still hold for the intermediate $W_{k/m}$ and the restrictions of the $G_i$.
For this we will make use of the following lemma, which will be applied to $X=Y=Z=\D$.

\begin{lemma}\label{lem:topo2}
  Let $f_1:X\to Y$ and $f_2: Y\to Z$ be holomorphic\footnote{We do not seek to determine minimal hypotheses for this statement.} maps between Riemann surfaces.
  Let $A\subset X$ and $C\subset Z$ be non-empty connected open subsets, and assume that $f_2\circ f_1$ restricts to a proper map from $A$ to $C$.
  Let $B=f_1(A)$.
  Then:
  \begin{enumerate}
    \item $f_1 :A\to B$ is proper,
    \item $f_2 :B\to C$ is proper.
    \item $B$ is a connected component of $f_2^{-1}(C)$,
  \end{enumerate}
\end{lemma}
\begin{proof}
  Let us call $f : A\to C$ the restriction to $A$ of $f_2\circ f_1$.
  By hypothesis, $f$ is proper.
  We have $f_2(B) = f_2(f_1(A)) = f(A) \subset C$.
  
  1. Let $K$ be a compact subset of $B$. Then $A\cap f_1^{-1}(K)$ is a closed subset of $A$ by continuity of $f_1$, and a subset of $f^{-1}(f_2(K))$, which is a compact subset of $A$ by continuity of $f_2$ and properness of $f$. So $A\cap f_1^{-1}(K)$ is compact.
  
  2. By definition of $B=f_1(A)$, $f_1:A\to B$ is surjective.
  Hence for every compact subset $K$ of $C$, $B\cap f_2^{-1}(K) = f_1(f^{-1}(K))$, which is compact by properness of $f$ and continuity of $f_1$.

  3. As a continuous image of a connected set, $B$ is connected.
  Since $f_1$ is holomorphic, $B$ is an open subset of $Y$.
  We have $f_2(B)\subset C$ i.e.\ $B \subset f_2^{-1}(C)$.
  Let $B'$ be the connected component of $f_2^{-1}(C)$ that contains $B$.
  If $B'$ were not equal to $B$, there would be some $z$ in the boundary of $B$ relative to $B'$.
  Let $z_n\in B$ that tends to $z$.
  Then $f_2(z_n)$ tends to $f_2(z)$, which belongs to $C$.
  By properness of $f_2: B\to C$, this would imply that we can extract a subsequence of $z_n$ that converges in $B$.
  So $z\in B$ but this cannot be since $B$ is an open subset of $B'$ and $z$ is in the relative boundary.
\end{proof}

We know that $G_m\circ \cdots \circ G_1$ is proper from $W_n$ to $W_{n+1}$ (by conjugacy to $f$).
By repeated application of the first two points of Lemma~\ref{lem:topo2}, we get that for $k=1,\ldots, m$, the restriction $G_k: W_{n+\frac{k-1}{m}} \to W_{n+\frac{k}{m}}$ is proper.

From Point 3.\ of Lemma~\ref{lem:topo2} we have that
for $k\in \{1,\ldots,n\}$, $W_{n+(k-1)/m}$ is a connected component of $G_k^{-1}(W_{n+k/m})$ and by Corollary~\ref{cor:topo} applied to $U=\D$ and the fact that $V_{n+1}$ is simply connected, we deduce by decreasing induction on $k$ that the $W_{n+(k-1)/m}$ is simply connected too.

We consider the fibred dynamical system 
\[
\Gamma: \left\{
\begin{array}{rcl}
\Z/m\Z\times\D & \to & \Z/m\Z\times\D
\vspace*{.3cm}\\
([k-1],z) & \mapsto & ([k], G_{k} (z))
\end{array}
\right.
\]
where $k\in\{1,\ldots,m\}$ and for $n\in\Z$, $[n]$ denotes its class of $n$ modulo $m\Z$. The map $\Gamma$ sends $\{[t]\}\times W_{t/m}$ to $\{[t+1]\}\times W_{(t+1)/m}$ properly.

The sets $\{[t]\}\times W_{t/m}$ are disjoint for distinct $t\in\Z$: indeed this is obvious if $t\not\equiv t'\bmod m$, and if $t\equiv t'\bmod m$ then this follows: in the case $t\equiv 0$ from disjointness of the $V_n$, hence of the $W_n$, for different $n$; in the case $t = n + i/m$ with $0<i<m$, from the fact that $G_{m}\circ\cdots\circ G_{i+1} (W_{n+i/m}) = W_{n+1}$ and the disjointess of the $W_n$. 

Let $r_c$ and $s_c$ be be the smallest and biggest $t\in\Z$ such that $\{[t]\}\times W_{t/m}$ contains a critical point of $\Gamma$.
We have $-\infty<r_c\leq s_c<\infty$ and for $t\in\Z$ such that $t<r_c$ or $t>s_c$, the map $\Gamma$ is an isomorphism from $\{[t]\}\times W_{t/m}$ to $\{[t+1]\}\times W_{(t+1)/m}$.
Corollary~\ref{cor:portrait critique} hence extends into the following statement, with the same line of arguments:
\begin{proposition}\label{prop:ce2}
  The restriction $Rf:B^*_{Rf}\to B^*_{Rf} $ is cover-equivalent to  $\Gamma^{s-r}: \{[r]\}\times W_{r/m}\to \{[s]\}\times W_{s/m}$ for any $r,s\in\Z$ such that $r\le r_c$ and $s>s_c$.
  In particular they have the same critical portrait.
\end{proposition}

Consider now the restriction $\Gamma_t$ of $\Gamma: \{[t]\} \times W_{t/m} \to \{[t+1]\}\times W_{(t+1)/m}$, which as we indicated is proper between simply connected open subsets of copies of $\D$.
These sets are thus isomorphic to $\D$ via some map $\phi_t : \{[t]\} \times W_{t/m} \to \D$ and the composition
\[H_t:= \phi_{t+1} \circ \Gamma_t \circ \phi_t^{-1}:\D\to\D\]
gives a proper self map of $\D$, i.e.\ a Blaschke product.
The degree $d_t$ of $H_t$ is equal to one plus its number $c_t$ of critical points counted with multiplicity:
\[d_t=c_t+1.\]
In particular $H_t$ is an isomorphism when $c_t=0$.
By Proposition~\ref{prop:ce2}, the restriction $Rf:B^*_{Rf}\to B^*_{Rf} $ is cover-equivalent to the composition
\[ H_{s_{c}}\circ H_{s_{c}-1}\circ \cdots \circ H_{r_c+1}\circ H_{r_c}:\D\to \D
.\]
Note that not all those Blaschke product necessarily have degree $>1$.
This composition belongs to the class $\cal B_{d'_p\circ\cdots \circ d'_1}$ where $(d'_p, \ldots,d'_1$) is obtained from the sequence of degrees $(\deg H_{s_c},\deg H_{s_c-1}, \ldots, \deg H_{r_c})$, from which we remove the terms equal to $1$. Hence  $d'_i\ge 2$ for all $i$.

\begin{figure}[htbp]
\centering
\begin{tikzpicture}
\node at (0,0) {\includegraphics[width=15cm]{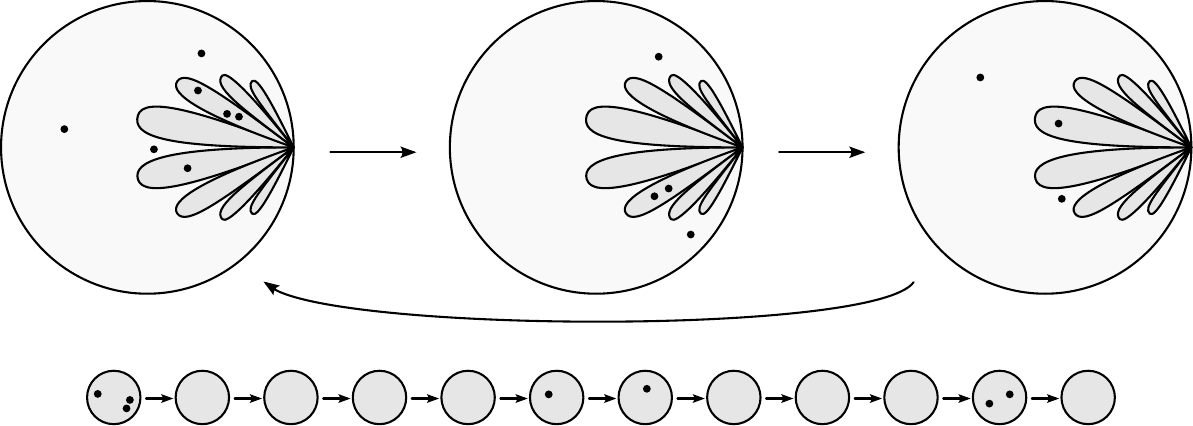}};

\node at (-2.8,1.2) {$G_1$};
\node at (2.8,1.2) {$G_2$};
\node at (2.8,-0.8) {$G_3$};

\begin{scope}[xshift=-6.2cm,yshift=0.8cm]
\node at (0.65,1.1) {\footnotesize $-2$};
\node at (0,0.4) {\footnotesize $-1$};
\node at (0.1,-0.4) {\footnotesize $0$};
\node at (.75,-1.1) {\footnotesize $1$};
\end{scope}

\begin{scope}[xshift=-0.66cm,yshift=0.8cm]
\node at (0.65,1.1) {\footnotesize $-5/3$};
\node at (0,0.4) {\footnotesize $-2/3$};
\node at (0.1,-0.4) {\footnotesize $1/3$};
\node at (.75,-1.1) {\footnotesize $4/3$};
\end{scope}

\begin{scope}[xshift=5cm,yshift=0.8cm]
\node at (0.65,1.1) {\footnotesize $-4/3$};
\node at (0,0.4) {\footnotesize $-1/3$};
\node at (0.1,-0.4) {\footnotesize $2/3$};
\node at (.75,-1.1) {\footnotesize $5/3$};
\end{scope}

\node at (-6.1,-3) {\footnotesize $-2$};
\node at (-5,-3) {\footnotesize $-5/3$};
\node at (-3.9,-3) {\footnotesize $-4/3$};
\node at (-2.71,-3) {\footnotesize $-1$};
\node at (-1.7,-3) {\footnotesize $-2/3$};
\node at (-0.6,-3) {\footnotesize $-1/3$};
\node at (0.62,-3) {\footnotesize $0$};
\node at (1.73,-3) {\footnotesize $1/3$};
\node at (2.86,-3) {\footnotesize $2/3$};
\node at (3.98,-3) {\footnotesize $1$};
\node at (5.10,-3) {\footnotesize $4/3$};
\node at (6.22,-3) {\footnotesize $5/3$};
\end{tikzpicture}
\caption{In this example, all critical points are simple and indicated by black dots.
Initially, the critical point count sequence is $(c_1,c_2,c_3)=(7,4,3)$.
After the gleaning process we get the sequence $(c'_1,\ldots,c'_4) = (3,1,1,2)$.}
\label{fig:W-glean}
\end{figure}
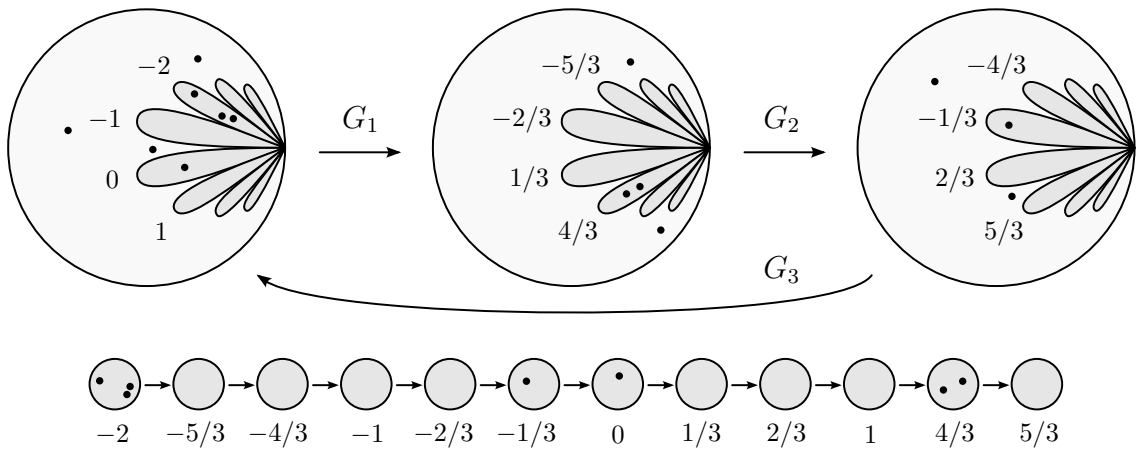

Let us show that $d'_p\circ\cdots \circ d'_1$ is a descendant of $d_m\circ \ldots d_1$.
We check the three points of Definition~\ref{def:glean}:
\\
1.\makebox[2mm]{}By Corollary~\ref{cor:nbax}, $B^*_{Rf}$ contains at least one critical point, so $p\geq 1$.
\\
2.\makebox[2mm]{}By construction $c'_i:=d'_i-1\geq 1$.
\\
3.\makebox[2mm]{}The number $d'_i$ was defined as the degree of $H_{t(i)}$ for some $t(i)\in\Z$.
Let $\beta(i)$ be the representative in $\{1,\ldots, m\}$ of the class $[t(i)]\in\Z/m\Z$.
Then, counting with multiplicity, the sum of $c'_i$ for the $i$ such that $\beta(i)=j$ is at most the number $c_j$ of critical points of $G_{j}$ in $\D$ because $c'_i$ is the number of critical points of $G_{j}$ in $W_{t(i)/m}$, and these sets are disjoint when $i$ varies in $\beta^{-1}(\{j\})$.
In other words:
\[\sum_{i\in\beta^{-1}(\{j\})} c'_i \leq c_j.\]

For the collection of all attracting axes $A$, the corresponding collection of gleanings is disjoint in the sense of Definition~\ref{def:glean2} because the associated $\{[t]\} \times W_{t/m}[A]$ and $\{[t']\} \times W_{t'/m}[A']$ are disjoint if $A\neq A'$ ($t$ may equal $t'$ or not) because they are mapped by a composition of $G_k$ and $\phi^{-1}$ to sets $V_n(A)$ and $V_{n'}(A')$, which are disjoint (Lemma~\ref{lem:Vndisjoint2}).

\section{A generalisation}\label{sec:grl}

There are several ways these results can be generalised, and these generalisations can be mixed. Here we explain what happens for a variation of the renormalization where the maps $f$ and $Rf$ can have any (and not necessarily the same) rational rotation numbers instead of $0$.

We start with the case where $f$ still has only one attracting axis but $Rf=R_{p/q,\pm} f$ has rotation number $p/q$ with $p\wedge q=1$, $q\geq 1$.
For this, we need to expand the definition of $Rf=R_{0,\pm} f$ given in Section~\ref{sec:def}.

In the case $\pm = +$, recall that $R_{0,+}f$ has been defined as a restriction of $\ell_\sigma$ with
\[ \ell_\sigma := \ov E \circ\, T_\sigma\circ h_f\circ \ov E^{-1} .\]
Here instead of requiring that $\ell_{\sigma_0}'(0) = 1$ we choose $\sigma_0$ such that
$\ell_{\sigma_0}'(0) = e^{2\pi ip/q}$ and we set, as before:
\[ R_{p/q,+} f = \ell_{\sigma_0}|_{\dom^*(\ell_{\sigma_0})}\]
completed by $R_{p/q,+} f(0)=0$.
Since we just changed the value of $\sigma_0$, this means we just have
$ R_{p/q,+}f = e^{2\pi ip/q} \times R_{0,+}f$.

In the case $\pm = -$, $\ell_\sigma$ are defined by the same formulae as above but with $E(z)=\exp(2\pi i z)$ replaced by $E(z)=\exp(-2\pi i z)$.
There is some $c_0\in\C^*$ such that for all $\sigma\in\C$, $\ell_\sigma(z) \sim e^{-2\pi i\sigma} c_0 z$ as $z\to 0$.
Again we choose the $\sigma_0\in\Z$, unique modulo $\Z$,
such that the multiplier at $0$ of $R_{p/q,-}f$ is $\exp(2\pi i p/q)$.
As a consequence, $R_{p/q,-}f = e^{2\pi ip/q} \times R_{0,-}f$.

The two cases $\pm\in\{-,+\}$ are summed up as:
\[ R_{p/q,\pm}f = e^{2\pi ip/q} \times R_{0,\pm}f .\]

It is still true that $R_{p/q,\pm} f$ is parabolic at $0$, i.e.\ that $(R_{p/q,\pm} f)^q$ is not the identity near $0$ by Remark~\ref{rem:atLeastOneAxis}.
The attracting/repelling axes of $R_{p/q,\pm}f$ are just those of the map $(R_{p/q,\pm}f)^q$, which is tangent to the identity.
The map $R_{p/q,\pm}f$ now has a number of attracting axes that is a multiple of $q$.
The immediate basin $B^*_{R_{p/q,\pm} f}(A)$ of an axis $A$ is defined as the connected component of the parabolic basin that contains a germ of the axis, or equivalently an attracting petal for this axis.
Note that $R_{p/q,\pm}f$ and $(R_{p/q,\pm}f)^q$ have the same parabolic basin (set of points whose orbit tends to $0$ without falling on $0$).

The sets $B^*_{R_{p/q,\pm}f}(A)$ are open, simply connected and isomorphic to $\D$.
Given an axis $A$, we consider it lives in the tangent space to $\C$ at $0$ and we denote
\[ A'=D_0 (R_{p/q,\pm}f)(A) \]
its image by the differential of $R_{p/q,\pm}f$ at $0$: then $A'$ is another axis.
Of course $A$ is back to itself after exactly $q$ iterations of $D_0(R_{p/q,\pm} f)$.
We replace here the notion of conjugacy by equivalence: the map $R_{p/q,\pm}f: B^*_{R_{p/q,\pm}f}(A) \to B^*_{R_{p/q,\pm}f}(A')$ is equivalent to a Blaschke product (Remark~\ref{rem:bp} generalises to this case), denote it $G(A)$:
\[
\begin{tikzcd}
    B^*_{R_{p/q,\pm}f}(A) \arrow[d, "\phi"] \arrow[r , "R_{p/q,\pm}f"] & B^*_{R_{p/q,\pm}f}(A') \arrow[d, "\psi"]
    \\
    \D \arrow[r , "G(A)"] & \D
\end{tikzcd}
\]
If $q>1$, it does not make sense to try and make $G(A)$ have a  parabolic point because it will not be iterated directly: it is the following composition that is iterated: $G(A_{q-1})\circ \cdots \circ G(A_0)$ along the cycle $A_k$ of $A$ under $D_0 (R_{p/q,\pm}f)$.
Anyway we are interested in the composition type of $G(A)$, i.e.\ of $R_{p/q,\pm} f: B^*_{R_{p/q,\pm}f}(A) \to B^*_{R_{p/q,\pm}f}(A')$, for each $A$.

In the statement below we do not modify the definition of $\DDD_\tau$: we still assume that $f\in \DDD_\tau$ has only one attracting axis.

\begin{theorem}\label{thm:main:grl}
  Let $\tau$ be a composition type.
  If $f\in \DDD_\tau$ then there is a disjoint collection of descendants $\tau[A]$ of $\tau$ for each attracting axis $A$ of $R_{p/q,\pm}f$, such that $G(A)$ has composition type $\tau[A]$.
\end{theorem}

The proof is the same with the following modifications: there are now collections $U_n[A]$, $V_n[A]$, $\{[t]\}\times W_{t/m}[A]$ for each $A$.
The $U$ sets are still pairwise disjoint: for the same axis the argument is the same and for different axes they are mapped by $E$ to disjoint immediate basins $B^*_{R_{p/q,\pm}f}(A)$.
For the $V$ sets and the $W$ sets, this is also the case, with similar arguments.

\medskip 
We now treat the case where $f$ has rotation number $r/s$, with $r\wedge s=1$ and $s\geq 1$.
We assume that $f$ has exactly $s$ attracting axes, i.e.\ \emph{$f$ has only one cycle of attracting axes.} 
(Indeed, in the case of more than one cycle of attracting axes,  parabolic renormalization is a subtle matter depending on the choice of gate structures \cite{Oudkerk}; we will not treat this case here.)
It has thus $s$ repelling axes.
We label them $\Upsilon_\frac{i}{s}$ in the trigonometric order around $0$.
We label the attracting axes $A_\frac{2i+1}{2s}$ so that  the trigonometric circular ordering of the $2s$ axes corresponds to the circular order of the indices in $\R/\Z$.
Then $D_0f$ sends any axis of index $x$ to an axis of the same nature and of index $x+\frac{1}{s}$.
Let $\psi_{\rep,0}$ be the extended repelling inverse Fatou coordinate associated to $\Upsilon_0$ and  $\phi_{\att,1/2s} $ the attracting Fatou coordinate associted to the following attracting axis : $A_{1/2s}$. 
Let the horn map be:
\[h := \phi_{\att,1/2s} \circ \psi_{\rep,0}.\]
Since we took the attracting axis that follows the repelling axis in trigonometric order, the domain of $h$ contains an upper half-plane.
Let $p/q$ be a rational number (possibly different from $r/s$), written in irreducible form. Now we can define the parabolic renormalization as:
\[R_{p/q,\pm}f := E\circ T_{\sigma_0+\frac{p}{q}} \circ h \circ E^{-1}\]
for some $\sigma_0\in\C$ that we choose so that $(R_{0/1}f)'(0)=1$, hence $(R_{p/q,\pm}f)'(0)=e^{2\pi i p/q}$.
As in Remark~\ref{rem:atLeastOneAxis}, $(R_{p/q,\pm}f)^q$ is different from the identity near $0$, i.e.\ $R_{p/q,\pm}f$ has attracting axes.
For an attracting axis $A$ of $R_{p/q,\pm}f$ we define $G(A)$ the same way as above Theorem~\ref{thm:main:grl}: it is a Blaschke product to which the restriction $R_{p/q,\pm}f: B^*_{R_{p/q,\pm}f}(A) \to B^*_{R_{p/q,\pm}f}(A')$ is equivalent.

\begin{theorem}\label{thm:main:grl:2}
  If $f$ has rotation number $r/s$ and only one cycle of attracting axes, denote by $\tau$ a composition type of $f^s$ on one of its immediate basins.
  Then there is a disjoint collection of descendants $\tau[A]$ of $\tau$ for each attracting axis $A$ of $R_{p/q,\pm}f$, such that $G(A)$ has composition type $\tau[A]$.
\end{theorem}

The main result rises an interesting and natural question.

\begin{question*}
Let $\tau$ be a composition type.
Assume that $f:U\to\C$ is a holomorphic map, $0\in U\subset\C$, parabolic at the origin, with only one cycle of attracting axes, and that $f$ is cover-equivalent (see Definition~\ref{def:covereq}) to $Rg:\on{Dom} Rg\to\C$  where $g\in\DDD_\tau$.
Is $f$ necessarily in $\DDD_\tau'$ for a descendant $\tau'$ of $\tau$?
\end{question*}

\addvspace{1cm}
\appendix
\noindent{\LARGE\bf Appendix}

\section{Proof of \e{Lemma~\ref{lem:bppp}}{}}\label{app:pflembpp}

Let us recall the statement: we assume a holomorphic map $f:U\subset\C\to \C$ has a parabolic fixed point $p$ with $f'(p)=1$ and one of whose petals has an immediate basin $B^*$ such that:
\begin{itemize}
  \item $B^*$ is simply connected,
  \item $f: B^*\to B^*$ is proper.
\end{itemize}
We consider a conformal map $\phi:B^* \to\D$ and the map $G=\phi\circ f\circ \phi^{-1}$, which is a finite Blaschke product, in particular it is the restriction to $\D$ of a rational map $F:\wh\C \to \wh\C$ that commutes with $z\mapsto 1/\bar z$.
Lemma~\ref{lem:bppp} claims that $F$ has a parabolic fixed point on $\partial \D$, with two repelling axes.

\medskip

We prove the claim by taking the proof of \cite{Orsay2}, Exposé~IX, Section~II.1, Proposition~4.c and modifying the final argument so that it adapts to the situation above.

For any hyperbolic connected open set $V$ denote $d_V$ the distance associated to its hyperbolic infinitesimal metric.
Recall that holomorphic maps are non-expanding for the hyperbolic distance, in particular
\begin{enumerate}
  \item\label{item:hme} $V\subset W\implies \forall x,y\in V,\ d_V(x,y)\geq d_W(x,y)$
  \item isomorphisms are isometries
\end{enumerate}
Concerning $\D$, the hyperbolic coefficient is $|dz|/(1-|z|^2)$ and
\begin{enumerate}[resume]
  \item\label{item:c} the hyperbolic ball of center $z$ and radius
  $r$ has a Euclidean diameter tending to $0$ as $z$ tends to $\partial\D$ while $r$ stays bounded,
  \item\label{item:d} $d_\D(a,b)\geq |a-b|$,
  \item\label{item:e} if $d_\D(a_n,b_n)\to 0$ then $\frac{|a_n-b_n|}{1-|b_n|^2} \to 0$.
\end{enumerate}
Concerning $\H$, the coefficient is $|dz|/2\Im z$ and
\begin{enumerate}[resume]
  \item \label{item:f} $d_\H(a,b) = \sinh^{-1}\frac{|a-b|}{2\sqrt{\Im a\Im b}}$.
\end{enumerate}
Let $(z_n)_{n\in\N}$ be any orbit in $B^*$.

\begin{claim*}
  $d_{B^*}(z_{n+1},z_{n})\to 0$.
\end{claim*}

\begin{proof} We can assume that the parabolic point is $p=0$. Since $z_n\to 0$ and $f(z)=z+cz^{q+1}+\ldots$ we get $z_{n+1}/z_n \to 1$.
On the other hand, since $z_n$ is asymptotic to the attracting axis $\Delta_\att$ and since $B^*$ contains a small sector based on $0$ and bisected by $\Delta_\att$, we have $B_n:=B(z_n,\kappa |z_n|)\subset B^*$ for some constant $\kappa>0$ and for all $n$ big enough.
So by item~\ref{item:hme}, $d_{B^*}(z_{n+1},z_{n}) \leq d_{B_n}(z_{n+1},z_{n}) = d_{B(1,\kappa)}(z_{n+1}/z_n,1)\to 0$.
\end{proof}

Denote $w_n = \phi(z_n)$, which is an orbit of $G$.
Since $\phi:B^*\to \D$ is an isometry for the respectives hyperbolic distances, the lemma above implies
\[d_\D(w_{n+1},w_n)\to 0.\]
Since $z_n\to \partial B^*$, the Euclidean distance from $w_n$ to $\partial \D$ tends to $0$.
(For the next argument we may invoke the Denjoy-Wolff theorem but let us stick to \cite{Orsay2} for a while.)
Because $|w_{n+1}-w_n|\leq d_\D(w_{n+1},w_n)$, the accumulation set of $w_n$ in $\D$ is connected.
Moreover by continuity it consists of fixed points of $F$.
The set of fixed points is finite, hence $w_n$ has only one accumulation point and since $\ov\D$ is compact, we get: $w_n$ tends to a point $a\in\partial \D$ such that $F(a)=a$.
We recall a classical and simple argument:

\begin{claim*}
  Every other orbit $w'_n$ in $\D$ tends to $a$.
\end{claim*}

\begin{proof}
 The map $G$ is non-expanding for the hyperbolic metric of $\D$ so $d_\D(w'_n,w_n)\leq r:=d_\D(w'_0,w_0)$ and $w_n\to a\in\partial\D$ so $|w'_n-w_n|\to 0$ item~\ref{item:c}.
\end{proof}

On one hand $\frac{w_{n+1}-w_n}{w_n-a} =  \frac{F(w_n)-a}{w_n-a}-1 \to F'(a)-1$, on the other hand, $\left|\frac{w_{n+1}-w_n}{w_n-a}\right| \leq \frac{|w_{n+1}-w_n|}{1-|w_n|} \leq 2\frac{|w_{n+1}-w_n|}{1-|w_n|^2} \to 0$ by item~\ref{item:e} since $d_\D (w_{n+1},w_n)\to 0$.
So $F'(a)=1$, hence $a$ is parabolic.

It cannot have more than 2 attracting axes: indeed the basins of each attracting axis is non-empty and they are disjoint open sets, so each is a union of connected components of the basin $B(a)$.
But $B(a)$ already contains $\D$ and since $F$ commutes with $z\mapsto 1/\bar z$, it also contains $\wh{\C}-\ov\D$.
It hence has at most two connected components.

To prove that it has exactly 2 attracting axis and not just 1, we use another argument than the one in \cite{Orsay2} (though we believe the one in \cite{Orsay2} can be adapted).
By contradiction assume there is only one attracting axis.
By the symmetry $z\mapsto -1/\bar z$ this axis is tangent to $\partial \D$ at $a$, for otherwise there would be another attracting axis.
The following claim contradicts that $d_\D(w_{n+1},w_n)\to 0$, which finishes the proof.

\begin{claim*}
For any finite Blaschke product $G$ whose extension $F$ has a parabolic point $a\in\partial \D$ with only one attracting axis, for any orbit $w_n\in \D$ tending to $a$, $d_\D(w_{n+1},w_n)$ tends to a positive value.
\end{claim*}

\begin{proof}
Let $d_n := d_\D(w_{n+1},w_{n})$.
By item~\ref{item:hme}, $d_{n+1} \leq d_n$, in particular it converges, and we want to prove that the limit is non-zero.
A change of coordinates by an appropriate homography sends $\D$ to $\H$, the point $a$ to $\infty$, and conjugates $F$ to a rational map $H$ such that
\[H(u) \underset{u\to\infty} = u+1+\frac{b}{u}+\cal O(\frac{1}{u^2})\]
where $b\in\R$.
It follows that any orbit $u_n\in H$ satisfies \[u_n=n+b\log n + \tau + \cal O(1/n)\]
for some $\tau\in\C$ that depends on $u_0$.
It follows that $\Im(u_n)$ converges to $\Im \tau$.
In particular $\Im(u_n)$ is bounded from above, while $u_{n+1}-u_n$ tends to $1$, so $d_\H(u_{n+1},u_n)$ is bounded from below by a positive constant by item~\ref{item:f}.
\end{proof}

\section{Comparison of the A, B, \ldots\ E classification versus the composition type}\label{app:comp}

For any slice $P:=\Per_1(e^{2\pi i p/q})$ in the space of cubic polynomials, for each $f\in P$, the map $Rf$ has either one or two critical values.
For each critical value $v$ of $Rf$, $Rf^{-1}(v)$ contains infinitely many critical points (every point in $\psi_\rep^{-1}(z)$ where $z$ is an iterated preimage of the critical point of $f$ which causes $v$ to be a critical value of $Rf$; knowing that $\psi_\rep$ is surjective).
But as we saw, the immediate basin of $Rf$ only contains finitely many critical points.
It follows that $Rf$ cannot have all its critical points in the immediate basin, so A and B have to be redefined.
Every $Rf$ of type A or B could be, in some sense, considered being of type C too because infinitely many critical points are in the parabolic basin but not in the immediate basin (they map to it in one iteration of $f$).
To avoid that, denote $\cal P$ the pearl necklace of the family $f$, i.e.\ $f$ is of type A or B.
Denote $\cal P'$ the set of parameters for which $Rf$ has two distinct critical values. Such parameters are necessarily in $\cal P$ and $\cal P - \cal P'$ is a set of 
isolated parameters whose accumulation is contained in the boundary of $\cal P$.
We then redefine the types for $f\in \cal P'$ (to avoid complications we do not cover the case $f\in \cal P-\cal P'$):
\begin{itemize}
    \item \textbf{Type A$'$}. (Adjacent, modified) Both critical values of $Rf$ belong to the immediate basin, and both have a critical preimage in a common component of the immediate basin.
    \item \textbf{Type B$'$}. (Bitransitive, modified) Both critical values belong to the immediate basin, and both have a critical preimage in different components of the immediate basin.
    \item \textbf{Type C$'$}. (Capture, modified) Both critical values belong to the parabolic basin of $0$, only one has a critical preimage in the immediate basin.
    \item \textbf{Type D$'$}. (Disjoint, modified) $Rf$ has only one critical value in the parabolic basin of $0$.
\end{itemize}

Recall that the pearl necklace of the family $Rf$ for $f\in \Per_1(e^{2\pi ip/q})$ is defined in Definition~\ref{def:pnRf} by the composition type of $Rf^q$ on any component of its immediate basin not being of composition type $2$ (then it is of type $2\circ 2$ or $3$).
It is contained in $\cal P$.
The intersection of this necklace with $\cal P'$ is also the set of parameters for which $Rf$ is of type A$'$ or B$'$.

Consider first the slice $P:=\Per_1(e^{2\pi i 1/3})$. 
Consider the renormalization $R=R_{0,+}$, which gives rotation number $0/1$ to $Rf$.
The pearl necklace of the family $Rf$ for $f\in P$ has components of type B$'$ and one of type A$'$.
If $Rf$ is not of type E, then it has only one attracting axis, so it can never be of class B$'$.
Every component of the pearl necklace for the family $Rf$, $f\in P$ is thus of type A$'$.
On Figure~\ref{fig:ipe2}, we added to the middle row of Figure~\ref{fig:iteratedParamEnrich}, an indication of the A$'$, \ldots, D$'$ classification of $f$ and of $Rf$.
We also show a sketch the pearl necklace of a nearby slice $\Per_1(e^{2\pi i\theta_n})$, for $\theta_n = [0;3,n]$ for some $n>0$ for comparison.

\begin{figure}
  \centering
  \begin{tikzpicture}
    \node at (0,0) {\includegraphics[width=14cm]{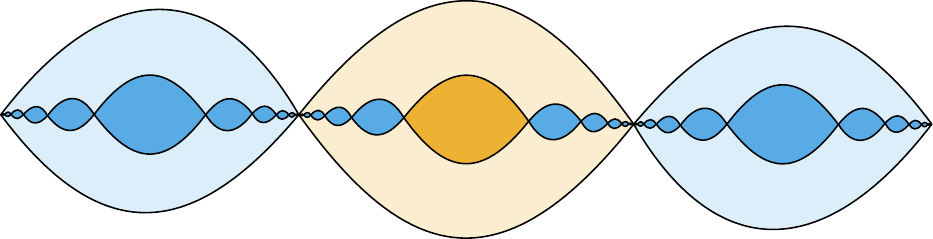}};
    
    \node at (-4.7,-1.75) {B};
    \node at (0,-2.1) {A};
    \node at (4.7,-2.0) {B};

    \node at (-4.7,1.2) {C$'$/D$'$};
    \node at (0,1.25) {C$'$/D$'$};
    \node at (4.7,1) {C$'$/D$'$};
    
    \node at (-3.55,0.075) {A$'$};
    \node at (-4.75,0.05) {A$'$};
    \node at (-5.95,0.075) {A$'$};
    \node at (-1.35,0.05) {A$'$};
    \node at (0,0) {A$'$};
    \node at (1.35,-0.025) {A$'$};
    \node at (3.55,-0.05) {A$'$};
    \node at (4.75,0) {A$'$};
    \node at (5.95,-0.05) {A$'$};
  \end{tikzpicture}
  \vskip1cm
  \centering
  \begin{tikzpicture}
    \node at (0.07,-0.02) {\includegraphics[width=14.2cm]{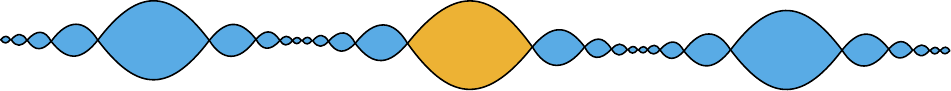}};
    
    \node at (-3.55,0.075) {B};
    \node at (-4.75,0.05) {B};
    \node at (-5.95,0.075) {B};
    \node at (-1.35,0.05) {B};
    \node at (0,0) {A};
    \node at (1.35,-0.025) {B};
    \node at (3.55,-0.05) {B};
    \node at (4.75,0) {B};
    \node at (5.95,-0.05) {B};
  \end{tikzpicture}
  \caption{Top: slice $\Per_1(e^{2\pi i 1/3})$. Primed letters indicate the type of $Rf$ for a parameter in the component. Non-primed letters indicate the type of $f$.
  We did not distinguish type C$'$ from type D$'$.
  Since $Rf$ has rotation number $0/1$, it never can have type B$'$.
  The colour indicates the composition type: blue for $2\circ 2$ and amber for $3$. Light blue means $Rf$ has composition type $2$ and $f$ composition type $2\circ 2$, while light amber means $Rf$ has composition type $2$ and $f$ composition type $3$.
  \\
  Bottom: slice $\Per_1(e^{2\pi i\theta_n})$, for $\theta_n = [0;3,n]$ for some $n>0$. The letters indicate the type of components for $f$. They are all of type B except the amber one. The colour indicates the composition type of $f$: blue for $2\circ 2$ and amber for $3$.}
  \label{fig:ipe2}
\end{figure}
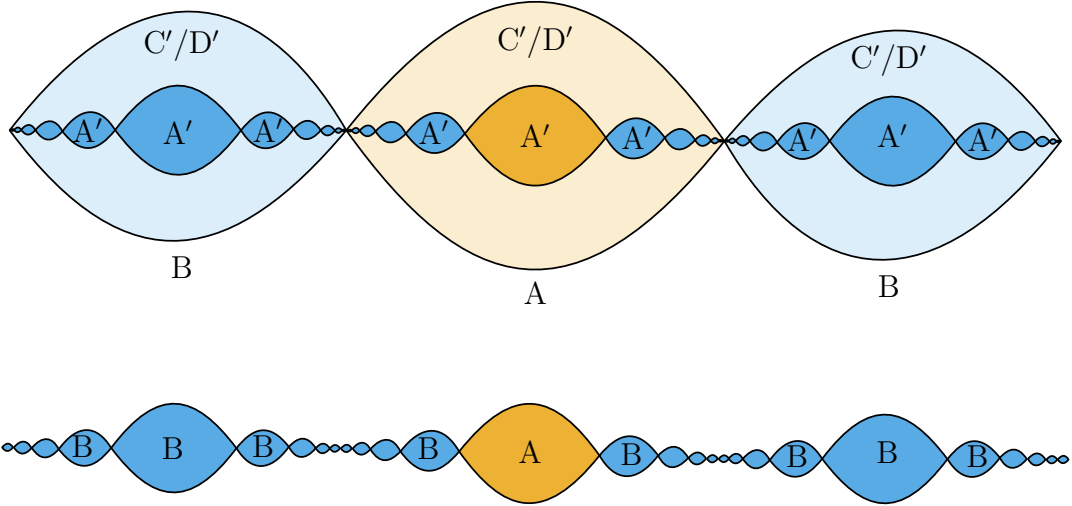

Consider then the slice $P:=\Per_1(e^{2\pi i 0/1})$ and the renormalization $R=R_{1/2,+}$.
Assuming that for some $f\in P$, the gleaning while passing from $f$ to $Rf$ occurs as on Figure~\ref{fig:glanage-3}, we would get a map that has type A$'$ but whose composition type is $2\circ 2$.

\begin{figure}
  \centering
  \begin{tikzpicture}
    \node at (0,0) {\includegraphics[width=14cm]{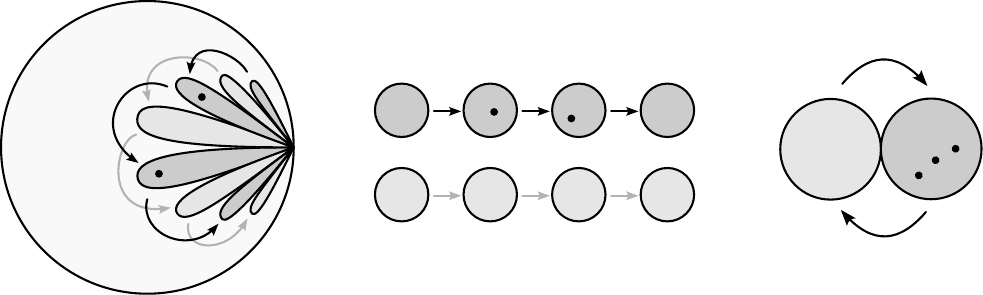}};
    \node at (-5.9,0) {$f$};
    \node at (5.6,1.6) {$Rf$};
    \node at (5.6,-1.6) {$Rf$};
  \end{tikzpicture}
  \caption{Black dots mark critical points.}
  \label{fig:glanage-3}
\end{figure}
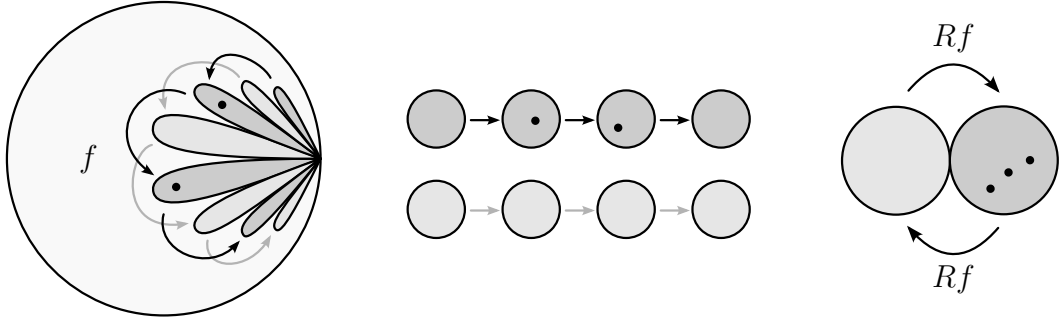

The best way to keep information is probably to keep track of the way the gleaning occurred, in the following sense. Instead of looking at properties of $Rf$ independently of $f\in \Per_1(\exp(2\pi i p/q))$, attach to $Rf$ how its restriction to its immediate basin is equivalent to a composition of $f:V_{n}\to V_{n+1}$ where $V_n$ is the sequence of virtual basins.
When iterating $R$, i.e.\ looking at $RRf$, one would look at the next gleaning still in terms of $f$, not of $Rf$, obtaining a double indexed collection of virtual basins $V_{n,m}\subset V_n$ and keeping track of the indexes in which the critical points appear.
For $R^3 f$ one has triple indexed virtual basins, etc.
However, this does not fit well with the point of view of invariant classes for renormalization.
The composition type is a way to keep part of this information and that better fits.

\section{Shishikura's proof of covering properties of the renormalized map}\label{app:Mitsu-proof}

We prove here Proposition~\ref{prop:rfft} by adapting the proof of Proposition~5.3 page~246 of \cite{Shi}, with more detailed arguments and slight differences.
Note that what is called an immediate basin in that article is, in our terms, an immediate basin on which $f$ is a proper self-map.
Let us denote, after \cite{Shi}, $\tilde B^u\subset\C$ the connected component of $\dom(h_f)$ that contains an upper half-plane: it is $T_1$-invariant and $B^u\subset \C/\Z$ denotes its quotient by $T_1$.

The main idea is to prove first that the restrictions of the extended attracting Fatou coordinate $\phi_\att : B^*_f\to \C$ and extended repelling Fatou parametrization $\psi_\rep : \tilde B^u\to B^*_f$ have no asymptotic values, and check that their composition, which has no asymptotic values by Lemma~\ref{lem:noavcomp}, has a set of critical values without accumulation, so by Corollary~\ref{cor:ramcov} it is a branched covering. Then we transfer this to the quotient by $E:z\mapsto \exp(2\pi i z)$ and deal with the two asymptotic values $0$ and $\infty$ of $Rf$ over $\hat \C$.

Denote $g$ the restriction $f:B^*_f\to B^*_f$

\medskip\noindent\textit{1. $\phi_\att : B^*_f\to \C$ has no asymptotic value:} Let $P\subset B^*_f$ be an attracting petal, on which $\phi_\att$ is thus injective and denote $A$ its image.
To simplify one may take $A$ to be a right half-plane.
Denote $\phi_0 $ the restriction of $\phi$ to $P$.
Let $U_n= B^*_f\cap f^{-n}(P)$ and $V_n=T_{-n} A$.
Then $\phi_n := \phi_\att|_{U_n} = T_{-n}\circ \phi_0 \circ g^n:\ U_n \to V_n$.
The map $g$ being proper, $g^n$ is proper and the restriction $g^n:g^{-n}(P)\to P$ is proper, so $\phi_n$ is proper from $U_n$ to $V_n$, and in particular it has no asymptotic values.
By Lemma~\ref{lem:noavcomp}, $\phi_\att : B^*_f\to \C$ has no asymptotic value.

\medskip\noindent\textit{2. $\psi_\rep : \tilde B^u \to B^*_f$ has no asymptotic value:} Let $z\in B^*_f$ and assume by way of contradiction that there is a path $\gamma :[0,1)\to \tilde B^u$ that leaves every compact subset of $\tilde B_u$ and such that $\psi_\rep\circ\gamma$ tends to $z$.
Let $U$, $U'$ be simply connected neighbourhoods of $z$ in $B^*_f$ such that $U$ is compactly contained in $U'$ and $U'$ compactly contained in $B^*_f$.
Cutting away the initial part of the path, we may assume that $\psi_\rep\circ\gamma$ takes values in $U$.
Let $w=\gamma(0)$, $w_n=w-n$ and $z_n=\psi_\rep(w_n)$.
Let $U_n\subset U'_n\subset B^*_f$ be respectively the connected component of $g^{-n}(U)\subset g^{-n}(U')$ that contains $z_n$.
By Corollary~\ref{cor:topo}, $U_n$ and $U'_n$ are simply connected.
Let $C$ be the finite set of critical points of $g$.
Since their orbit tend to $0$, there exist $n_0\in \N$ such that $\forall c\in C$ and $\forall n\in\N$, if $g^n(c)\in U'$ then $n\leq n_0$.
Then for all $n>n_0$, $U'_n$ contains no critical point of $g$.
It follows that $g:U'_{n}\to U'_{n-1}$ is an isomorphism (recall $U_n$ and $U'_n$ are simply connected).
The sequence of inverse isomorphisms $\chi_{n-n_0} : g^{-(n-n_0)}: U'_{n_0} \to U'_n$ takes values in the hyperbolic set $B^*_f$, so is normal as a sequence of maps from $U'_{n_0}$ to $\hat\C$.
It maps $z_{n_0}$ to $z_n$, which tends to $0\in\partial B^*_f$.
Hence it tends to $0$ uniformly on compact subsets of $U'_{n_0}$, hence uniformly on $\ov U_{n_0}$.
Fix a repelling petal $P$ and a repelling Fatou coordinate $\phi$ on $P$ that is an inverse branch of $\psi_\rep$, and denote $A=\phi(P)$.
The local dynamics of $f_0$ near $0$ being conjugate to the translation $z\mapsto z+1$ near $\infty$, we get by a compactness argument on $\ov U_{n_0}$ that $\exists n_1\geq n_0$, $\forall n\geq n_1$, $\ov U_n\subset \chi_{n-n_0}(\ov U_{n_0}) \subset P$.
Now, $\psi_\rep\circ \gamma = g^n\circ \psi_\rep\circ T_{-n}\circ \gamma$ takes values in $U$, and the path $\psi_\rep\circ T_{-n}\circ \gamma$ starts from $z_n$ hence takes values in $U_n$.
There is $n_2$ such that for all $n\geq n_2$ the path $T_{-n}\circ\gamma$ starts from a point of $A$ (since $P$ is a repelling petal).
We may choose $n_2\geq n_1$ and then $\forall n\geq n_2$, $T_{-n}\circ\gamma$ takes values in $\phi(U_n)$ by a connectedness argument.
Since $\gamma$ leaves every compact subset of $\dom(\psi_\rep)$, which is $T_1$-invariant, if follows that $T_{-n}\circ\gamma$ leaves in particular every compact subset of $\phi(U_n)$.
The map $\psi_\rep$ is an isomorphism from $\phi(U_n)$ to $U_n$.
So $\psi_\rep\circ T_{-n}\circ \gamma$ leaves every compact subset of $U_n$.
Since $g^n$ is proper from $U_n$ to $U$ and $\psi_\rep\circ \gamma = g^n\circ \psi_\rep\circ T_{-n}\circ \gamma$, $\psi_\rep\circ \gamma$ leaves every compact subset of $U$, contradicting that it converges.

\medskip\noindent\textit{3. Composing and quotienting:} The composition $h_f = \phi_\att \circ \psi_\rep: \tilde B^u\to \C$ has thus no asymptotic value over $\C$ by Lemma~\ref{lem:noavcomp}.
Let $\hat h_f$ be the quotient map by the canonical projection $\pi :\C\to \C/\Z$ (so $\hat h_f\circ\pi = \pi\circ h_f$).
Then $\hat h_f$ cannot have an asymptotic value either: if a path $\gamma$ in $B^u$ leaves every compact subset of $B^u$ then it lifts by $\pi$ to a path $\tilde \gamma$ leaving every compact of $\tilde B^u$, and $h_f\circ \tilde\gamma$ cannot converge in $\C$, so $\hat h_f\circ \gamma = \pi \circ h_f\circ\tilde\gamma$ cannot converge either.
The set of critical values $cv(\hat h_f)$ is the image of the set of critical values of $g$ by $\pi\circ \phi_\att$, so is finite.
It follows that $\hat h_f$ is a branched covering from $B^u$ to $\C/\Z$.
Recall now the definition $Rf=\ell_\sigma = \ov E \circ\, T_\sigma\circ \hat h_f\circ \ov E^{-1}$, extended at $0$ by $\ell_\sigma(0)=0$ and $\sigma$ is chosen so that $\ell_\sigma'(0)=1$. The set of critical values of $\ell_\sigma$ is $\ov E\circ T_\sigma(cv(\hat h_f))$ so is finite.
Over $\C^*$, $Rf$ has no asymptotic values, so the only possible asymptotic values over $\hat \C$ are $0$ and $\infty$ (see Point~\ref{item:prop:av:2} of Proposition~\ref{prop:av}).

\medskip\noindent\textit{4. $0$ and $\infty$ are indeed asymptotic values of $Rf$:} This did not need to be addressed in \cite{Shi}, and is not essential here either, however we sketch a justification.
Let $P$ be an attracting petal of $f$ mapped to a right half-plane by the attracting Fatou coordinate and containing no critical value of $f$.
The set $g^{-1}(P)$ has $\deg g>1$ components, one contains $P$, call it $P_{-1}$, let $P'$ be any other.
One can draw in $P'$ a path $\gamma$ whose image by $\pi\circ\phi_\att$ tends to the top end of the cylinder (the same holds with the bottom end).
Since $\psi_\rep:\tilde B^u\to B^*_f$ has no asymptotic values, it is surjective and moreover any path lifts by $\psi_\rep$.
Let $\delta$ be a lift of $\gamma$.
Then $\Im\delta(t)$ stays bounded away from $+\infty$, for otherwise this would imply, since $\psi_\rep$ sends some high enough half-plane to a sepal of $f$, that $\gamma$ intersect $P_{-1}$, but $P_{-1}$ and $\gamma$ are disjoint.
Hence $E\circ \delta$ stays bounded away from $0$ and since its image by $Rf$ tends to $0$ or $\infty$, it must leave every compact subset of $\dom Rf$.

\section{Example of non-tame situation}\label{app:non-tame}

Consider the transcendental entire map $f:\C\to\C$, $z\mapsto z e^z$.
It singular values are the asymptotic values\footnote{There cannot be an asymptotic value $a\in\C^*$ because if $z(t)=x(t)+iy(t)$ tends to $\infty$ and $f(z)$ tends to $a$ then $|f(z)| = |z| e^{x} \tends |a|$ so $e^x$ tends to $0$, i.e.\ $x$ tends to $-\infty$. Then there is a continuous lift $\phi(t)$ of $\arg z$ for all $t$ big enough such that $\phi(t)\in [\pi/2,3\pi/2]$ and since $\arg f(z)\tends \arg( a )\bmod 2\pi$, we get that the continuous function $y(t)+\phi(t)$ converges mod $2\pi$ hence it converges in $\R$, hence $y$ converges. Hence $|z|\sim |x|$ and $|f(z)|\sim |x| e^x \tends 0$, leading to a contradiction.} $\infty$ and $0$ and the critical value $f(-1)=-e^{-1}$, where $-1$ is the unique critical point.
The parabolic basin is necessarily simply connected by Proposition~\ref{prop:scparabo}.
Denote $B^*$ the immediate basin. By Fatou's theorem, the restriction $f: B^* \to B^*$ has at least one singular value.

\begin{figure}[htbp]
    \centering
    \includegraphics[width=\textwidth]{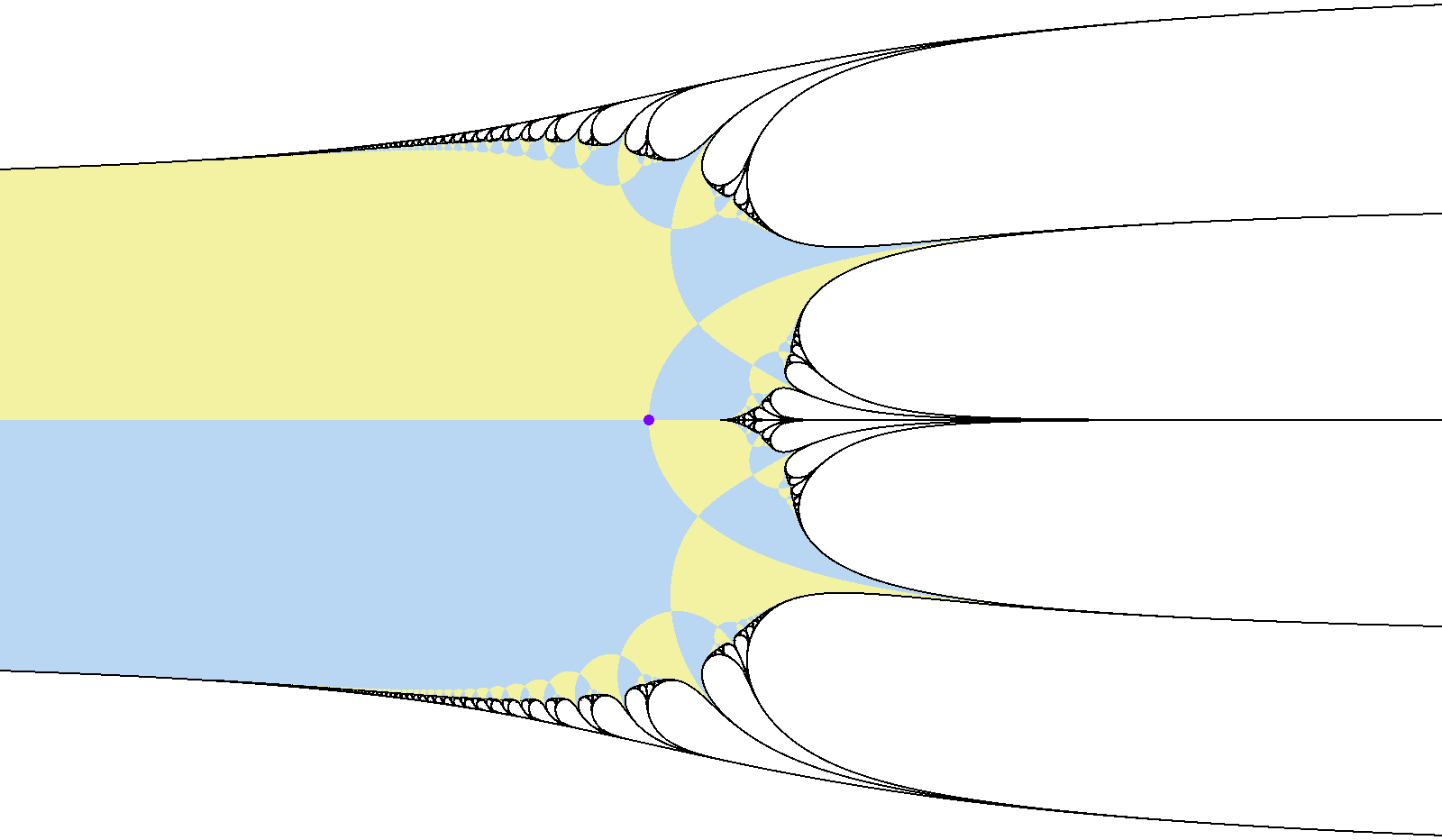}
    \caption{The immediate basin $B^*$ of $0$ for the entire map $f(z)=z e^z$. We coloured the interior with the parabolic chessboard pattern: yellow for points such that $\Im(\phi(z))<\Im(\phi(-1))$, blue for the others. The black part shows the boundary of $B^*$. It is so thin in the strands that we only see black curves, but bear in mind that this set is fractal everywhere, by self-similarity. A purple dot marks the position of the critical point.}
    \label{fig:zez}
\end{figure}

By Corollary~\ref{cor:svbas}, we get that the restriction $f:B^*\to B^*$ has no asymptotic value.
Hence, $f$ has a critical point in $B^*$: by uniqueness of the critical point, $-1\in B^*$.
The set $B^*$ is not the whole plane, since it cannot contain $0$, so it is hyperbolic.
Since $f(z) = z+z^2+\ldots$, there is only one attracting axis.
By Lemma~\ref{lem:rcdocv1}, $f:B^*\to B^*$ is equivalent to $z\mapsto z^2$ from $\D$ to $\D$, so
\[f\in \DDD_2.\]
By Shishikura's theorem: $Rf\in \DDD_2$ too.
We show on Figure~\ref{fig:zez} a computer rendering of $B^*$.
Figure~\ref{fig:zez2} shows an enlargment near the parabolic point.

On Figure~\ref{fig:zez3} we show the preimage of the basin by the extended repelling Fatou coordinates.
On Figure~\ref{fig:zez4} we overlay the preimage by $E$ of the immediate basin of $Rf$.

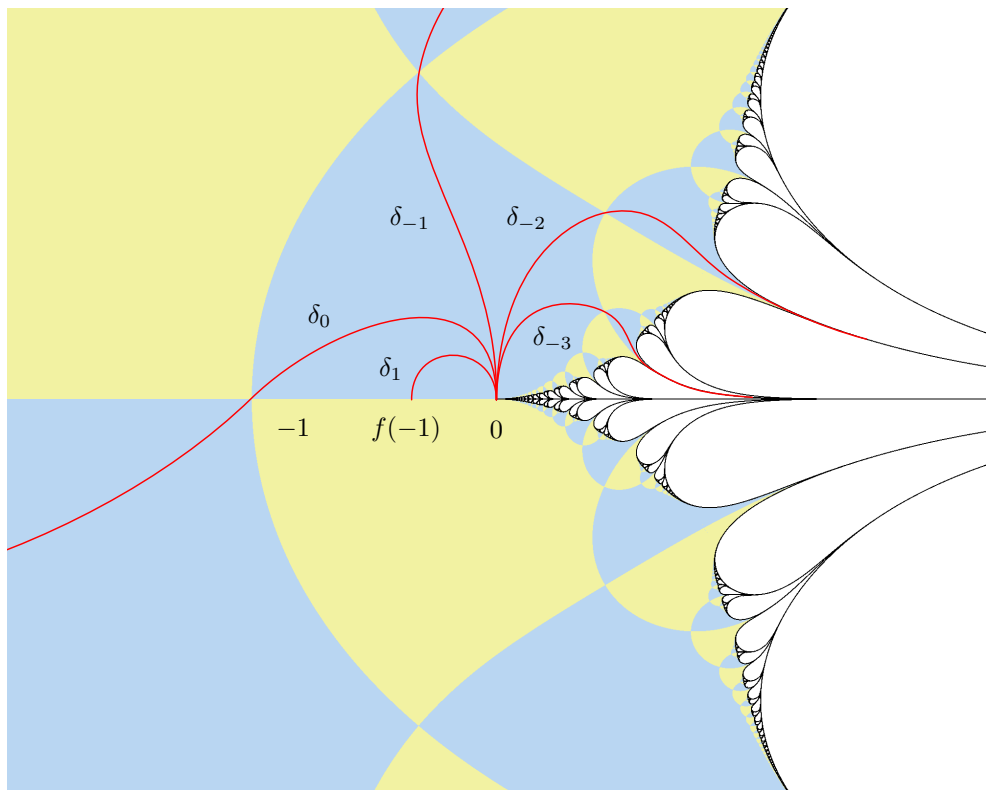
\begin{figure}[htbp]
    \centering
    \footnotesize
    \begin{tikzpicture}
      \node at (0,0) {\includegraphics[width=13cm]{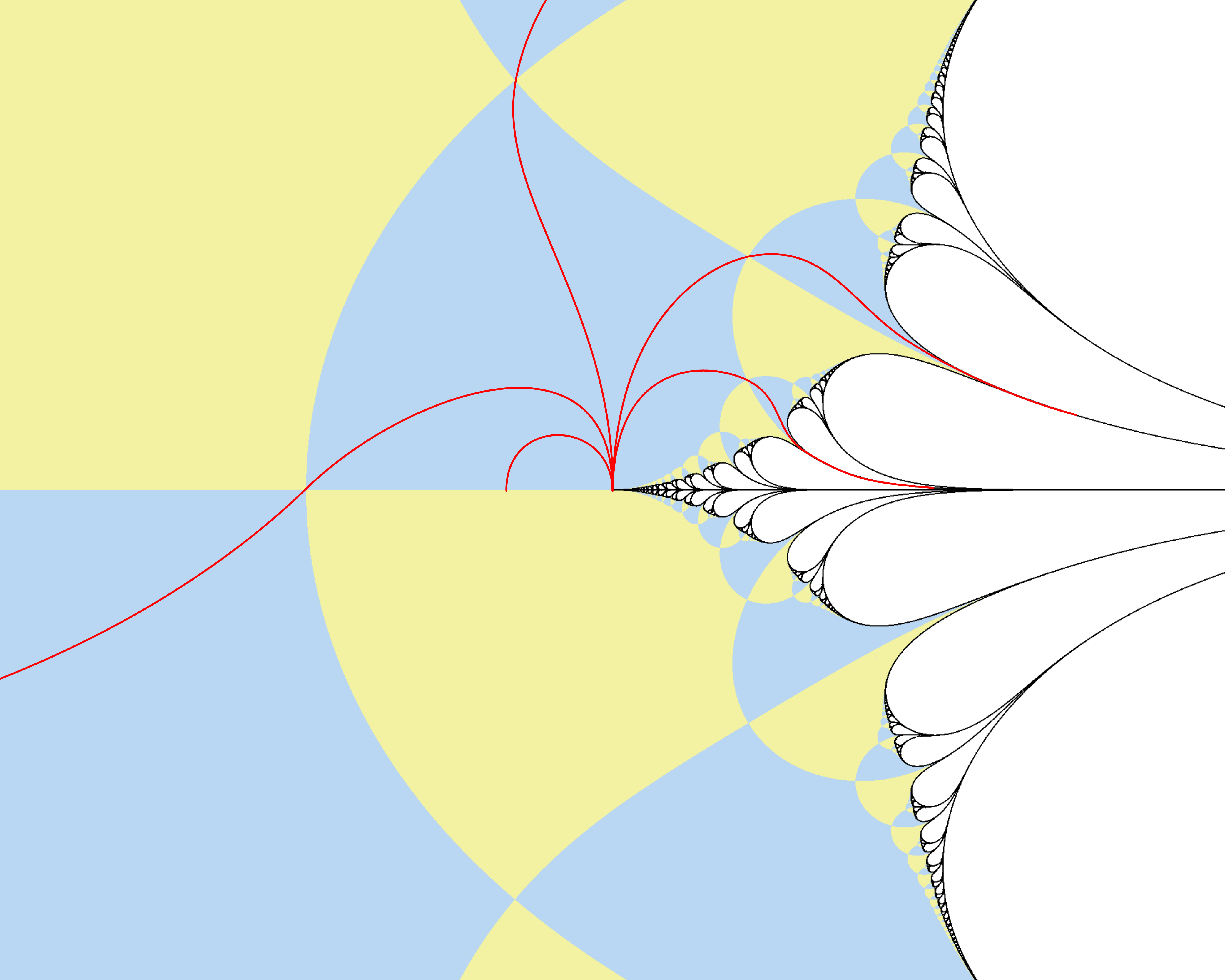}};
      \node at (-2.7,-0.4) {$-1$};
      \node at (-1.2,-0.4) {$f(-1)$};
      \node at (0,-0.4) {$0$};
      \node at (-1.4,0.4) {$\delta_1$};
      \node at (-2.35,1.1) {$\delta_0$};
      \node at (-1.15,2.4) {$\delta_{-1}$};
      \node at (0.4,2.4) {$\delta_{-2}$};
      \node at (0.75,0.8) {$\delta_{-3}$};
    \end{tikzpicture}
    \caption{Here, we zoomed on the parabolic point $0$ of $f$. The curves $\delta_n$ are indicated in red, they were sketched by hand.}
    \label{fig:zez2}
\end{figure}

\begin{figure}[htbp]
    \centering
    \includegraphics[width=\textwidth]{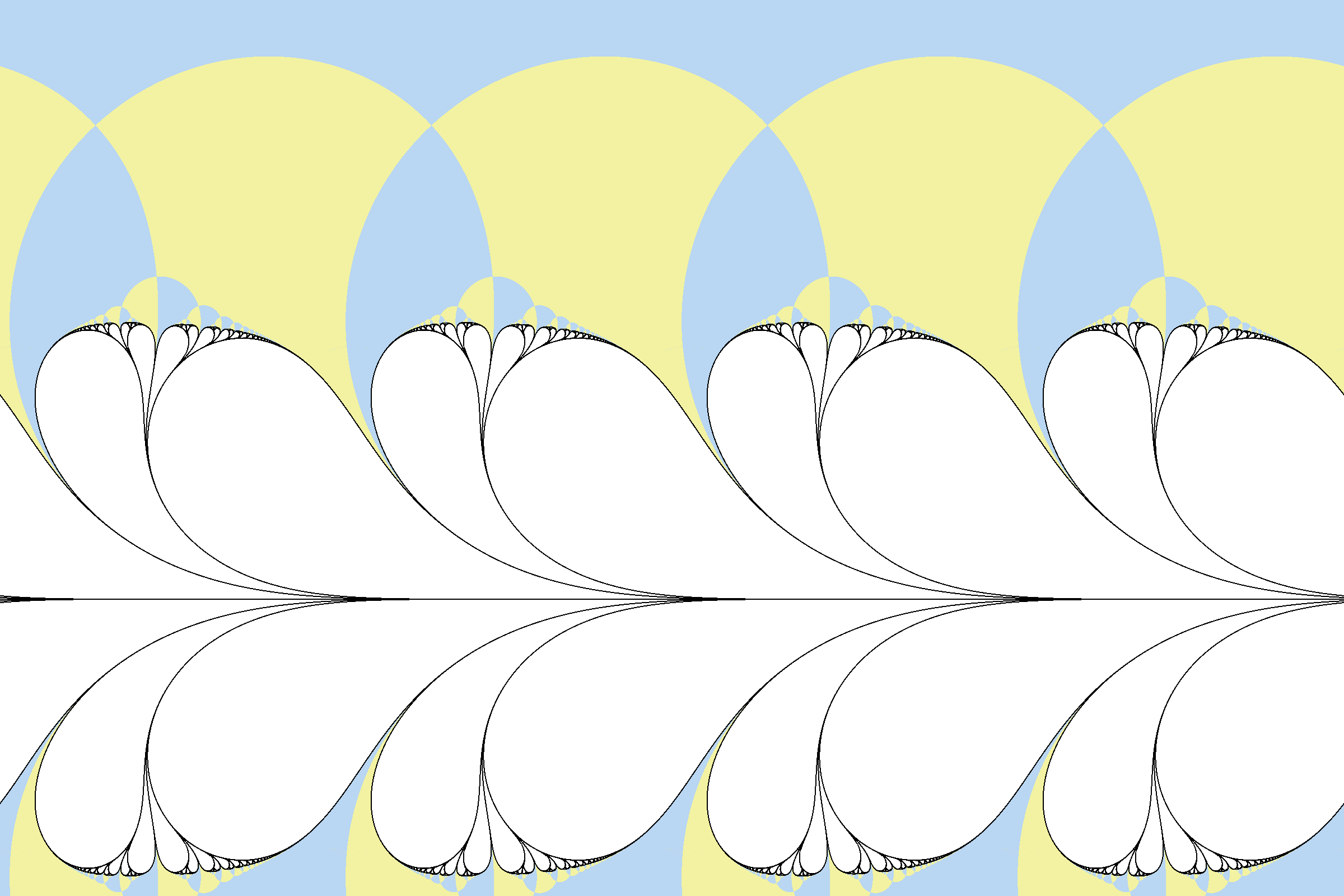}
    \caption{Preimage by the extended repelling Fatou coordinates $\psi_\rep$, of the immediate basin of $f$.}
    \label{fig:zez3}
\end{figure}

\begin{figure}[htbp]
    \centering
    \includegraphics[width=\textwidth]{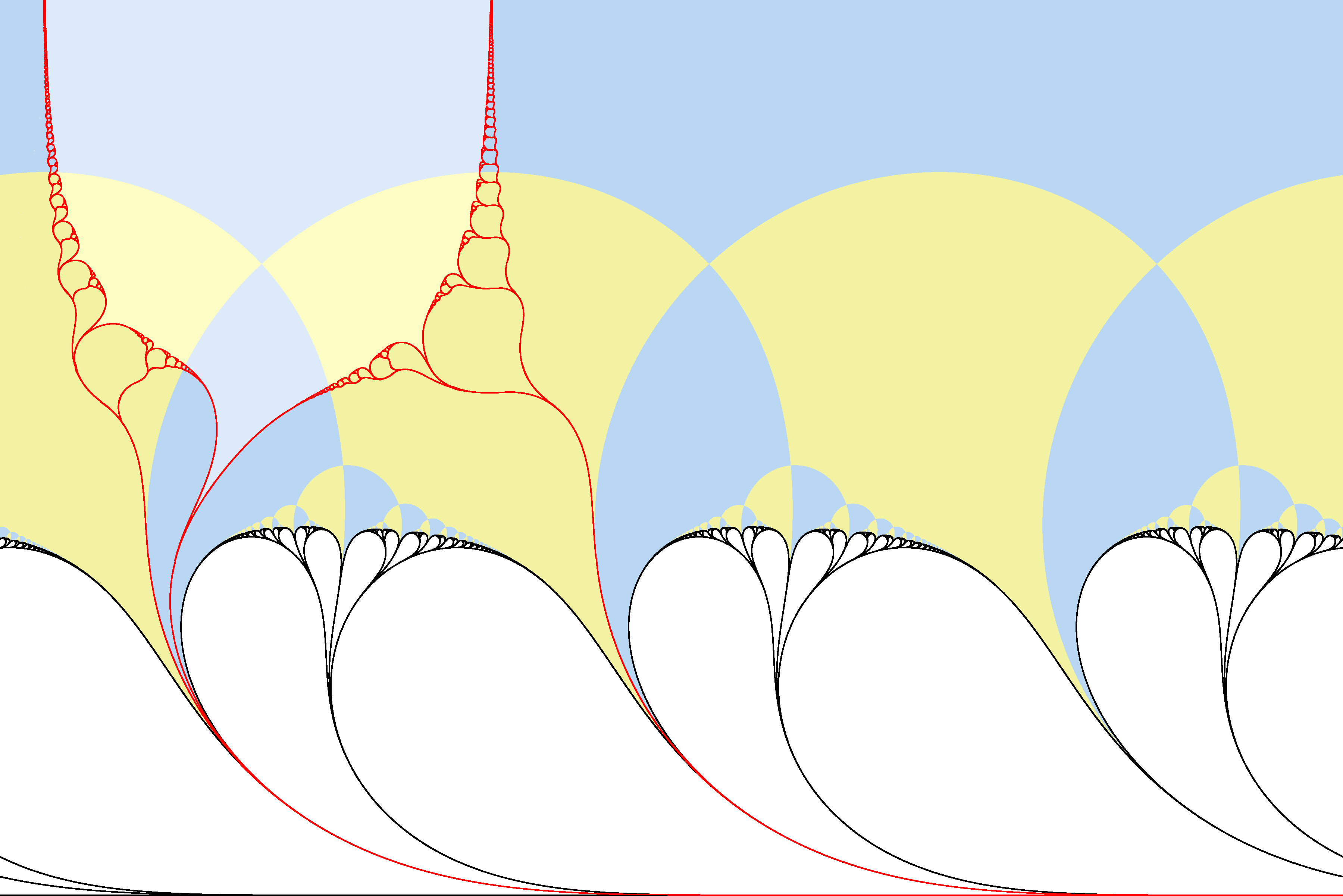}
    \caption{The attracting basin of $Rf$ is overlaid in semi-transparent over an enlargement of the previous picture. Its boundary is indicated by red pixels.}
    \label{fig:zez4}
\end{figure}

It is easy to prove, using that $f(\R)\subset\R$ and an analysis of the real dynamical system $f|_\R$, that $B^*\cap\R = (-\infty,0)$.
So $B^*$ is disjoint from $f^{-1}([0,+\infty])$, which consists in $[0,+\infty]$, a sequence of curves $C_k$ indexed by $k>0$ of parametrized equation $z=x+iy$, $x=(2\pi k-\theta)\on{cotan}(\theta)$, $y=2\pi k-\theta$, $\theta\in(0,\pi)$ and their images $C_{-k}$ by $z\mapsto \bar z$.
It follows that $B^*$ is contained between $C_{-1}$ and $C_{1}$, in particular $|\Im(z)|<2\pi$ for all $z\in B^*$.

Since $f\in\DDD_2$, the conformal mapping $\phi:\D\to B^*$ normalized by $\phi(0)=-1$ and $\phi'(0)>0$ conjugates\footnote{The map $\phi^{-1}\circ f\circ\phi$ is a degree 2 Blaschke product with critical point $0$ and its rational extension has a parabolic point with two attracting axes according to Lemma~\ref{lem:bppp}.
It preserves the real axis and sends $0$ to the right of $0$. There is only one map satisfying these conditions.} $B_2:z\mapsto (z^2+1/3)/(1+z^2/3)$ to $f:B^*\to B^*$.
The map $\phi$ sends $(-1,0]$ to $(-\infty,-1]$, in particular it has a radial limit at $z=-1$ equal to $\infty$.
It follows that the iterated preimages by $B_2$ of $(0,1]$, which form arcs drawn in the parabolic chessboard graph, have their endpoints forming a dense subset of $\partial D$, and these arcs $\gamma$ are sent by $\phi$ to arcs that an iterate of $f$ maps to $(-\infty,-1]$, hence the $\gamma$'s also tend to $\infty$.
As such, $\phi$ cannot have a continuous extension from $\ov\D$ to $\wh\C$, for this extension would be constant and equal to $\infty$, contradicting conformal mapping theory.
By Carathéodory's theorem, the boundary of $B^*$ in $\wh\C$ is not locally connected.

Consider a path $\gamma$ in $B^*_{Rf}$ from the (unique) critical value of $Rf$ and tending to $0$, tangentially to its attracting axis.
Let us focus on the set $(Rf)^{-1}(\gamma)\cap B^*_{Rf}$.
It consists in two curves from the critical point of $Rf$ in $B^*_{Rf}$, one tending to $0$, and another on which we focus now.
Recall $Rf$ is a restriction of the conjugate by $\ov E:\C/\Z\to\C^*$ of the map $T_\sigma \circ \phi_\att \circ \psi_\rep$ taken modulo $\Z$, for some $\sigma\in\C$.
The preimage of $\gamma$ by $\ov E$ is a curve $\tilde \gamma$ tending to the upper end of the cylinder, with a real part that converges.
Denote $U\subset\C$ a lift of $B^*_{Rf}$ by $E:\C\to\C^*$, $z\mapsto \exp(2\pi iz)$ and denote $U_n= T_n(U)$ where $T_n(z)=z+n$. 
The preimage of $\tilde \gamma$ under the canonical projection $\C\to\C/\Z$ is a collection of curves $T_n(\hat\gamma)$, $n\in\Z$.
Taking a further preimage by $T_\sigma$, then $\phi_\att$, we get even more curves, but we are only  interested in the ones that are contained in $V_n=\psi_\rep(U_n)$ for some $n\in\Z$.
Such curves $\delta_n$ will connect to the parabolic point $0$ of $f$ in $B^*$, arriving tangentially to the direction of argument $\pi/2$.
One of them connects to $f(-1)$.
Up to reindexing we may assume it is $\delta_1$.
Then $\delta_0$ is in two parts, one bounded from $z=0$ to $z=-1$, and one unbounded from $z=-1$ and tending to $\infty$ within $B^*$, with real part tending to $-\infty$.
The $\delta_n$ for $n<0$ all are in two parts, one bounded, one with real part tending to $+\infty$.
They are attached to corners of the main upper chessboard box, which belong to any fixed repelling petal for all $n<n_0<0$.
None of the $\delta_n$ can be contained in the petal as they are all unbounded.
It follows that: \emph{none of the $V_n$ can be contained in a bounded repelling petal} either.

Take a petal whose image in repelling Fatou coordinates is a left half-plane. Since every $V_n$ gets out of the petal this implies by connectedness that $U_0$ has points of arbitrarily high real part.

\printbibliography

@article {blokh,
    AUTHOR = {Blokh, Alexander and Oversteegen, Lex and Timorin, Vladlen},
     TITLE = {Slices of the parameter space of cubic polynomials},
   JOURNAL = {Trans. Amer. Math. Soc.},
  FJOURNAL = {Transactions of the American Mathematical Society},
    VOLUME = {375},
      YEAR = {2022},
    NUMBER = {8},
     PAGES = {5313--5359},
      ISSN = {0002-9947,1088-6850},
   MRCLASS = {37F20 (37F10 37F46 37F50)},
  MRNUMBER = {4469222},
       DOI = {10.1090/tran/8519},
       URL = {https://doi.org/10.1090/tran/8519},
}

@article{BH,
 author = {Branner, Bodil and Hubbard, John H.},
 title = {The iteration of cubic polynomials Part II: patterns and parapatterns},
 volume = {169},
 journal = {Acta Mathematica},
 number = {},
 publisher = {Institut Mittag-Leffler},
 pages = {229 -- 325},
 year = {1992},
 doi = {10.1007/BF02392761},
 URL = {https://doi.org/10.1007/BF02392761},
}

@article{Arnaud,
    AUTHOR = {Ch{\'e}ritat, Arnaud},
     TITLE = {Near parabolic renormalization for unicritical holomorphic maps},
   JOURNAL = {Arnold Math. J.},
  FJOURNAL = {Arnold Mathematical Journal},
    VOLUME = {8},
      YEAR = {2022},
    NUMBER = {2},
     PAGES = {169--270},
      ISSN = {2199-6792,2199-6806},
   MRCLASS = {37F25},
  MRNUMBER = {4446269},
       DOI = {10.1007/s40598-020-00172-6},
       URL = {https://doi.org/10.1007/s40598-020-00172-6},
}

@article{CE,
 author = {Ch{\'e}ritat, Arnaud and Epstein, Adam Lawrence},
 title = {Bounded type Siegel disks of finite type maps with few singular values},
 fjournal = {Science China. Mathematics},
 journal = {Sci. China, Math.},
 issn = {1674-7283},
 volume = {61},
 number = {12},
 pages = {2139--2156},
 year = {2018},
 doi = {10.1007/s11425-018-9381-4},
 zbMATH = {7000638},
 Zbl = {1404.37050},
}

@misc{CLM,
 title={Horn maps of holomorphic functions locally pseudo-conjugate on their parabolic basins}, 
 author={Arnaud Chéritat and Dimitri Le Meur},
 year={2025},
 eprint={2210.11211v3},
 eprinttype = {arXiv},
 archivePrefix={arXiv},
 primaryClass={math.DS},
 url={https://arxiv.org/abs/2210.11211v3},
}

@article{DouadyImplosion, 
  title={Does a Julia set depend continuously on the polynomial?},
  journal={In Complex dynamical systems (Cincinnati, OH, 1994), volume 49 of Proc. Sympos. Appl. Math., Amer. Math. Soc., Providence, RI},
  author={Adrien Douady},
  year={1994},
  pages={91--138},
}

@book {Orsay2,
    AUTHOR = {Douady, A. and Hubbard, J. H.},
     TITLE = {\'Etude dynamique des polynômes complexes. Partie II},
    SERIES = {Publications Mathématiques d'Orsay [Mathematical
              Publications of Orsay]},
    VOLUME = {85-4},
      NOTE = {With the collaboration of P. Lavaurs, Tan Lei and P. Sentenac},
 PUBLISHER = {Université de Paris-Sud, Département de Mathématiques,
              Orsay},
      YEAR = {1985},
     PAGES = {v+154},
   MRCLASS = {58F08 (30D05 39B10)},
  MRNUMBER = {812271},
MRREVIEWER = {M.\ Rees},
}

@article{Ec,
  author = {Ecalle, Jean},
  title = {Invariants holomorphes simples des transformations de multiplicateur 1},
  fjournal = {Comptes Rendus Hebdomadaires des Séances de l'Académie des Sciences, Série A},
  journal = {C. R. Acad. Sci., Paris, Sér. A},
  issn = {0366-6034},
  volume = {276},
  pages = {375--378},
  year = {1973},
  language = {French},
  zbMATH = {3397852},
  Zbl = {0252.30010}
}

@article{EL,
  author = {Eremenko, A. Eh. and Lyubich, M. Yu.},
  title = {The dynamics of analytic transformations},
  fjournal = {Leningrad Mathematical Journal},
  journal = {Leningr. Math. J.},
  issn = {1048-9924},
  volume = {1},
  number = {3},
  pages = {563--634},
  year = {1990},
  language = {English},
  zbMATH = {4182434},
  Zbl = {0717.58029}
}

@article{AdamThesis,
  author = {Epstein, Adam L.},
  year = {1993},
  title = {Towers of Finite Type Complex Analytic Maps},
  journal = {PhD thesis, City University of New York},
}

@article{lavaurs,
  author = {Pierre Lavaurs},
  year = {1989},
  title = {Systèmes dynamiques holomorphes: explosion de points périodiques paraboliques},
  journal = {PhD thesis,
  Thèse de doctorat de l’Université de Paris-Sud, Orsay, France}
}

@book{LY,
  author = {Lanford, Oscar E. III and Yampolsky, Michael},
  title = {Fixed point of the parabolic renormalization operator},
  fseries = {SpringerBriefs in Mathematics},
  series = {SpringerBriefs Math.},
  issn = {2191-8198},
  year = {2014},
  publisher = {Cham: Springer},
  doi = {10.1007/978-3-319-11707-2},
  zbMATH = {6351318},
  Zbl = {1342.37051}
}

@article{MilnorH,
    author={J. Milnor}, 
    title={Hyperbolic components, in: Conformal Dynamics and Hyperbolic Geometry}, 
    journal={Contemp. Math.}, 
    volume={573}, 
    publisher={Amer. Math. Soc., Providence, RI}, 
    year={ 2012}, 
    pages={ 183–232.}
}

@incollection{MilnorCubic,
 author = {Milnor, John},
 title = {Cubic polynomial maps with periodic critical orbit. I},
 booktitle = {Complex dynamics. Families and friends},
 isbn = {978-1-56881-450-6},
 pages = {333--411},
 year = {2009},
 publisher = {Wellesley, MA: A K Peters},
 language = {English},
 zbMATH = {5652527},
 Zbl = {1180.37073},
}

@Book{Milnorbook,
  Author = {Milnor, John},
  Title = {Dynamics in one complex variable},
  Edition = {3rd ed.},
  FSeries = {Annals of Mathematics Studies},
  Series = {Ann. Math. Stud.},
  Volume = {160},
  Year = {2006},
  Publisher = {Princeton, NJ: Princeton University Press},
  DOI = {10.1515/9781400835539},
  zbMATH = {2244718},
  Zbl = {1085.30002}
}

@incollection {Oudkerk,
    AUTHOR = {Oudkerk, Richard},
     TITLE = {The parabolic implosion: Lavaurs maps and strong convergence
              for rational maps},
 BOOKTITLE = {Value distribution theory and complex dynamics (Hong Kong,
              2000)},
    SERIES = {Contemp. Math.},
    VOLUME = {303},
     PAGES = {79--105},
 PUBLISHER = {Amer. Math. Soc., Providence, RI},
      YEAR = {2002},
      ISBN = {0-8218-2980-7},
   MRCLASS = {37F10 (30D05 37F35 37F50 39B12)},
  MRNUMBER = {1943528},
MRREVIEWER = {Peter\ Ha\"issinsky},
       DOI = {10.1090/conm/303/05239},
       URL = {https://doi.org/10.1090/conm/303/05239},
}

@incollection{PetersenTan,
  title = "Branner-Hubbard Motions and attracting dynamics.",
  author = "Petersen, {Carsten Lunde} and Lei Tan",
  year = "2006",
  pages = "45--70",
  editor = "Petersen, {Carsten Lunde} and Hjorth, {Poul G.}",
  booktitle = "Dynamics on the Riemann Sphere",
  publisher = "European Mathematical Society Publishing House",
}

@article{Rempe,
 author = {Rempe-Gillen, Lasse and Sixsmith, Dave},
 title = {Hyperbolic entire functions and the Eremenko-Lyubich class: Class \(\mathcal {B}\) or not class \(\mathcal {B}\)?},
 fjournal = {Mathematische Zeitschrift},
 journal = {Math. Z.},
 issn = {0025-5874},
 volume = {286},
 number = {3-4},
 pages = {783--800},
 year = {2017},
 doi = {10.1007/s00209-016-1784-9},
 zbMATH = {6780312},
 Zbl = {1392.37039}
}

@article{Roesch1,
     author = {P. Roesch},
     title = {Hyperbolic components of polynomials with a fixed critical point of maximal order},
     journal = {Annales scientifiques de l'\'Ecole Normale Supérieure},
     pages = {901--949},
     publisher = {Elsevier},
     volume = {Ser. 4, 40},
     number = {6},
     year = {2007},
     doi = {10.1016/j.ansens.2007.10.001},
     zbl = {1151.37044},
     mrnumber = {2419853},
     language = {en},
     url = {http://www.numdam.org/articles/10.1016/j.ansens.2007.10.001/}
}

@incollection{Shi1,
 author = {Shishikura, Mitsuhiro},
 title = {Bifurcation of parabolic fixed points},
 booktitle = {The Mandelbrot set, theme and variations},
 isbn = {0-521-77476-4},
 pages = {325--363},
 year = {2000},
 publisher = {Cambridge: Cambridge University Press},
 zbMATH = {2096506},
 Zbl = {1062.37043}
}

@article{Shi,
 author = {Shishikura, Mitsuhiro},
 title = {The Hausdorff dimension of the boundary of the Mandelbrot set and Julia sets},
 fjournal = {Annals of Mathematics. Second Series},
 journal = {Ann. Math. (2)},
 issn = {0003-486X},
 volume = {147},
 number = {2},
 pages = {225--267},
 year = {1998},
 doi = {10.2307/121009},
 zbMATH = {1187679},
 Zbl = {0922.58047}
}

@article {Vo,
    AUTHOR = {Voronin, S. M.},
     TITLE = {Analytic classification of germs of conformal mappings $({\bf
              C},\,0)\rightarrow ({\bf C},\,0)$},
   JOURNAL = {Funktsional. Anal. i Prilozhen.},
  FJOURNAL = {Akademiya Nauk SSSR. Funktsional\cprime ny\u i\ Analiz i ego
              Prilozheniya},
    VOLUME = {15},
      YEAR = {1981},
    NUMBER = {1},
     PAGES = {1--17, 96},
      ISSN = {0374-1990},
   MRCLASS = {58C27 (30C35)},
  MRNUMBER = {609790},
MRREVIEWER = {G.\ N.\ Khimshiashvili},
}

@misc{RunzeSiegel,
      title={Rigidity of bounded type cubic Siegel polynomials}, 
      author={Jonguk Yang and Runze Zhang},
      year={2024},
      eprint={2311.00431},
      archivePrefix={arXiv},
      primaryClass={math.DS},
      url={https://arxiv.org/abs/2311.00431}, 
}

@article{Zakeri,
  title={Dynamics of Cubic Siegel Polynomials},
  author={Saeed Zakeri},
  journal={Communications in Mathematical Physics},
  year={1999},
  volume={206},
  pages={185-233}
}

@misc{Runze,
  author = {R. Zhang},
  howpublished = "\url{http://arxiv.org/abs/2210.14305}",
  title = {Parabolic Components in Cubic Polynomial Slice {P}er$_1(1)$},
  year = {2022},
}

@misc{RunzeImplo,
      title={Parabolic Implosion in the Parameter Space of Cubic Polynomials}, 
      author={Runze Zhang},
      year={2025},
      eprint={2508.16430},
      archivePrefix={arXiv},
      primaryClass={math.DS},
      url={https://arxiv.org/abs/2508.16430}, 
}

@article{runze3,
author = {Zhang, Runze},
title = {On dynamical parameter space of cubic polynomials with a parabolic fixed point},
journal = {Journal of the London Mathematical Society},
volume = {110},
number = {6},
pages = {e70038},
doi = {https://doi.org/10.1112/jlms.70038},
url = {https://londmathsoc.onlinelibrary.wiley.com/doi/abs/10.1112/jlms.70038},
eprint = {https://londmathsoc.onlinelibrary.wiley.com/doi/pdf/10.1112/jlms.70038},
year = {2024}
}

\end{document}